\RequirePackage{fix-cm}
\documentclass{svjour3bis}                     
\smartqed  
\usepackage{graphicx}
\usepackage{mathptmx}      
\journalname{myjournal}
\usepackage{enumitem}
\usepackage{tikz}
\usepackage{comment}
\usepackage{amsmath}
\usepackage{stmaryrd}
\usepackage{bbm}
\usepackage{esint}
\usepackage{marginnote}
\usepackage{amssymb}
\usepackage[hidelinks]{hyperref}

\renewcommand{\L}[1]{\mathbf{L^{\pmb #1}}}
\newcommand{\C}[1]{\mathbf{C^{\pmb #1}}}
\newcommand{\Cc}[1]{\mathbf{C_{\rm c}^{#1}}}
\newcommand{\Lloc}[1]{\mathbf{L_{loc}^{\pmb #1}}}
\renewcommand{\geq}{\geqslant}
\renewcommand{\leq}{\leqslant}
\DeclareMathOperator{\TV}{TV}
\DeclareMathOperator{\sign}{sign}
\DeclareMathOperator\supp{supp}

\begin{document}
\allowdisplaybreaks
\newcommand{\delaynewpage}[1]{\enlargethispage{#1\baselineskip}}

\title{On existence, stability and many-particle approximation of solutions of 1D Hughes model with linear costs
}

\titlerunning{On existence, stability and approximation of Hughes' model}        

\author{Boris Andreianov \and Massimiliano~D.\ Rosini \and Graziano Stivaletta}

\institute{B.\ Andreianov \at
Institut Denis Poisson (CNRS UMR7013), Universit{\'e} de Tours, Universit{\'e} d'Orl{\'e}ans, Parc Grandmont 37200 Tours, France
\at
Peoples' Friendship University of Russia (RUDN University) 6 Miklukho-Maklaya St, Moscow, 117198, Russian Federation \\
\email{boris.andreianov@univ-tours.fr}
\and
M.D.\ Rosini \at
Department of Mathematics and Computer Science, University of Ferrara I-44121 Italy
\at
Uniwersytet Marii Curie-Sk\l odowskiej, Plac Marii Curie-Sk\l odowskiej 1 20-031 Lublin, Poland
\\
\email{massimilianodaniele.rosini@unife.it}
\and
G.\ Stivaletta \at
Department of Information Engineering, Computer Science and Mathematics, University of L'Aquila I-67100, Italy
\\
\email{graziano.stivaletta@graduate.univaq.it}
}

\date{July 17, 2021}

\maketitle

\delaynewpage{3}
\begin{abstract}
This paper deals with the one-dimensional formulation of Hughes model for pedestrian flows in the setting of entropy solutions, which authorizes  non-classical shocks at the location of the so-called turning curve.
We consider linear cost functions, whose slopes $\alpha\geq0$ correspond to different crowd behaviours.

We prove	 existence and partial well-posedness results in the framework of entropy solutions.
The proofs of existence are based on a a sharply formulated many-particle approximation scheme with careful treatment of interactions of particles with the turning curve, and on local reductions to the well-known Lighthill-Whitham-Richards model.
For the special case of $\mathbf{BV}$-regular entropy solutions without non-classical shocks, locally Lipschitz continuous dependence of such solutions on the initial datum $\overline{\rho}$ and on the cost parameter $\alpha$ is proved. 
Differently from the stability argument and from existence results available in the literature, our existence result allows for the possible presence of non-classical shocks. 
First, we explore convergence of the many-particle approximations under the assumption of uniform space variation control.
Next, by a local compactness argument that permits to circumvent the possible absence of global $\mathbf{BV}$ bounds, we obtain existence of solutions for general measurable data.

Finally, we illustrate numerically that the model is able to reproduce typical behaviours in case of evacuation.
Special attention is devoted to the impact of the parameter $\alpha$ on the evacuation time.

\end{abstract}

\section{Introduction}

In recent years, the modelling of large human crowds attracted considerable scientific interest.
This is due to its potential applications in structural engineering and architecture, see for instance \cite{MR3308728,MR3076426,MR3642940,MR3932134} and the references therein.

\subsection{Microscopic and macroscopic modeling of pedestrian flows}
Several models for pedestrian flow are already available in the literature, see again \cite{MR3308728,MR3076426,MR3642940,MR3932134} and the references therein.
Models for pedestrian motion split into two groups: macroscopic and microscopic modeling.
Differently from the microscopic modeling, macroscopic modeling is suited for the derivation of general results, such as for the evacuation time optimization, and is useful in understanding realistic crowds involving large numbers of pedestrians.
On the other hand, crowd dynamics are essentially microscopic and it is therefore easier to  motivate microscopic rather than some macroscopic assumptions.
Moreover, macroscopic modeling relies on the continuum assumption,
which, unlike in fluid dynamics, is not fully justified in the present framework with large but not huge number of agents.

We address the one-dimensional version of Hughes model \cite{Hughes02}. Our main goal is
to provide an original existence result for general data and, at the same time, to validate the passage to the limit from a well-assessed microscopic follow-the-leader model to a macroscopic model for pedestrian flow. This is similar to taking the hydrodynamic limit of Boltzmann equations. As a matter of fact, a byproduct of our analysis is a further justification of the Hughes model through the hydrodynamic limit procedure; another byproduct of our analysis is the introduction of a microscopic (many-particle) Hughes model with a sharp description of the particle switching dynamics.
Here we do not require any assumption on the smallness (not even on the finiteness) of the total variation of the initial datum. To the best of authors' knowledge, we prove the first existence result for Hughes model dealing with non-classical shocks.
Our secondary goal is to initiate the analysis of stability of solutions to the Hughes model; this part of the paper does not deal with non-classical shocks. Our last goal is to numerically explore the influence on the crowd evacuation time of the cost parameter $\alpha$ that encodes the agents' sensitivity to crowdedness.

\subsection{A state of the art} Hughes' model treats the crowd as a fluid made of ``thinking'' particles.
It was introduced by R.L.~Hughes in 2002 \cite{Hughes02}.
The model is given by a non-linear conservation law with discontinuous flux, coupled with an eikonal equation via a cost functional, see \eqref{e:modela} below.
Analytical study and numerical approximation of discontinuous-flux conservation laws (which should be seen, in a more precise way, as conservation laws with inner interfaces, cf. \cite{MR3416038}) is by now quite well developed, see in particular \cite{MR2195983,MR2086124,MR2291816,MR2807133,MR3416038,MR1770068,MR1961002,MR2024741,MR2209759,MR2300671,MR3425264} and the references therein. However, even in the simplest case of space dimension one, Hughes' model features a moving interface, whose dynamics is coupled to the solution of the conservation law. For this reason, existence and uniqueness analysis for the Hughes' model appears to be challenging.  
In higher dimension, further difficulties stem from a richer structure of singularities of solutions of the eikonal equation.
We stress that, for the unique viscosity solution of the eikonal equation, no more than Lipschitz continuity can be expected.
These various difficulties motivated the development of several attempts to study, both analytically and numerically, the Hughes' model and its regularized variants.

Let us briefly recall the existence results for the Hughes model available in the literature.
In \cite{MR2737207} the authors present an existence and uniqueness theory for a regularized version of the Hughes  model.
Also in \cite{MR3199781,MR3460619,MR3823842} the authors considered a modified version of the model.
To the best of author's knowledge, there are only two existence results for the original formulation of the one-dimensional Hughes' model. Both of them require restrictive assumptions on the initial data, needed in order to exclude the appearance of non-classical shocks. The first of these results is obtained in \cite{AmadoriGoatinRosini}, where
the authors follow the wave-front tracking approach \cite{MR303068}.
The second existence result is obtained in \cite{DiFrancescoFagioliRosiniRussoKRM}.
There, the authors exploit the fact that Hughes' model is based upon  the Lighthill-Whitham-Richards model (abbreviated to ``LWR model'' in the sequel) \cite{LWR1,LWR2} for vehicular traffic; more precisely, in space dimension one it can be seen as two exemplars of LWR model subject to elaborated coupling across the turning curve.
This makes it possible to apply to the Hughes  model the particular instance of ``follow-the-leader'' many-particle approximation of the LWR model developed and analyzed  in \cite{DiFrancescoRosini,DiFrancescoFagioliRosini-BUMI} (see also the different analytical approach of \cite{HoldenRisebro1,HoldenRisebro2}).
At last, we recall that various numerical approaches to the Hughes model can be found in the literature, e.g., \cite{HUANG2009127,MR3055243,MR3698447,MR3277564,MR3619091,MR3149318,MR3177723}.

\subsection{The model}
The one-dimensional Hughes model \cite{Hughes02} describes evacuation of a bounded and crowded corridor, parametrized by $x \in \mathfrak{C} := (-1,1)$, through two exits located at the extremities of the corridor, i.e., at $x = \pm1$.
The model is given by the scalar conservation law with discontinuous flux coupled with the eikonal equation
\begin{subequations}
\label{e:model}
\begin{align}
& \rho_{t}-\left( \rho \, v( \rho )\frac{\phi_{x}}{|\phi_{x}|}\right)_{x} = 0,&
&|\phi_{x}| = c( \rho ),&
&x \in \mathfrak{C},\ t > 0.
\label{e:modela}
\end{align}
Here $t \geq 0$ denotes the time, $x \in \mathfrak{C}$ is the space variable, $\rho = \rho(t,x) \in [0, \rho_{\max}]$ stands for the macroscopic (averaged) crowd density, with $\rho_{\max}>0$ being the maximal density.
The map $\rho \mapsto v( \rho )$ is the absolute value of the velocity, and is assumed to be decreasing, as higher velocities correspond to lower densities.
The map $\rho \mapsto c( \rho )$ is the \emph{running cost} and is assumed to be increasing since densely crowded regions lead to an augmentation of travel time.

Beside \eqref{e:modela}, we consider the homogeneous Dirichlet boundary conditions at the exits
\begin{align}
& \rho(t,-1) = \rho(t,1) = 0,&
&\phi(t,-1) = \phi(t,1) = 0,&
&t > 0,
\label{e:modelb}
\end{align}
and initial condition
\begin{equation}
\rho (0,x) = \overline{\rho}(x),
\label{e:modelc}
\end{equation}
\end{subequations}
assigned for $x \in\mathfrak{C}$. Concerning the initial datum $\overline{\rho}$, we assume that
\begin{align}
\tag{I}\label{I}
\overline{\rho} \colon \mathfrak{C} \to [0, \rho_{\max}] \; \hbox{ is measurable},
\qquad
L &:= \left\|\overline{\rho}\right\|_{\L{1}(\mathfrak{C})} > 0,&
R_{\max} &:= \left\|\overline{\rho}\right\|_{\L{ \infty}(\mathfrak{C})} \leq \rho_{\max}.
\end{align}
Since $\overline{\rho}$ attains its values in $[0,R_{\max}]$ and the problem obeys a maximum principle (see \cite{ElKhatibGoatinRosini}), also the solution $\rho$ attains its values in $[0,R_{\max}]$.

The choice \eqref{e:modelb} of boundary conditions for $\rho$ allows for a reformulation of the problem. Indeed, the homogeneous Dirichlet condition at the exits can be seen as the open-end condition, and permits to suppress the boundary while extending the solution to the whole space $\mathbb{R}$ (see Proposition~\ref{p:refo}).

In the literature, the running cost function $c$ and the velocity map $v$ are often assumed to satisfy the following conditions: 
\begin{gather}
\tag{C}\label{C} 
c \in \C{2}\left([0,R_{\max}];[1, \infty) \right),\ c(0) = 1
\hbox{ and }
\forall \rho \in [0,R_{\max}]\ c'( \rho ),c''( \rho ) \geq 0
\\[5pt]
\tag{V}\label{V1}
\begin{aligned}
&v \in \C{2}([0,\rho_{\max}];[0,v_{\max}]) \hbox{ with } v_{\max} := v(0) > 0,\ v( \rho_{\max})=0,
\\
& \text{$v$ strictly decreasing on $[0,\rho_{\max}]$ and~}
\forall \rho \in [0, \rho_{\max}]\ 2v'( \rho ) + \rho \, v''( \rho ) < 0.
\end{aligned}
\end{gather}
Note that for  $c$ given by \eqref{e:crazy}, properties \eqref{C} require the restriction $R_{\max}<\rho_{\max}$.

The last inequality in \eqref{V1} means that $f (\rho):= \rho \, v( \rho )$ is strictly concave; thus, since $f$ vanishes at $0$ and $\rho_{\max}$, there exists a unique $\rho_c := \mathrm{argmax}_{[0, \rho_{\max}]} \{f\}$. 
In Section~\ref{s:existBVloc}, we will strengthen the last inequality of \eqref{V1} by assuming that the map $\rho \mapsto \rho \, v'( \rho )$ is non-increasing, i.e.,
\begin{gather}
\tag{$\mathrm{V}'$}\label{Vbis}
v'( \rho ) + \rho \, v''( \rho ) \leq 0 \hbox{ for any } \rho \in [0, \rho_{\max}].
\end{gather}

Following \cite{DiFrancescoFagioliRosiniRussoKRM,AmadoriDiFrancesco,AmadoriGoatinRosini,ElKhatibGoatinRosini}, due to the explicit resolution of the eikonal equation in \eqref{e:modela} (with the boundary conditions on $\phi$ in \eqref{e:modelb}), the one-dimensional Hughes model \eqref{e:model}  is equivalently formulated as follows (see Definition~\ref{d:entro} on page~\pageref{d:entro}): given the so called \emph{turning curve}
$\xi \colon t \in [0,\infty)\to \mathfrak{C}$  implicitly defined by
\begin{equation}
\label{e:cost0}
\int_{-1}^{\xi(t)}c\bigl( \rho(t,y)\bigr) \, {\rm{d}} y = \int_{\xi(t)}^{1}c\bigl( \rho(t,y)\bigr) \, {\rm{d}} y,
\end{equation}
under the initial condition \eqref{e:modelc} and the homogeneous Dirichlet boundary condition
\begin{equation}\label{eq:Dirichlet-BLN}
\rho(t,\pm 1)=0 
\end{equation}
(the latter being understood in the sense of Bardos-LeRoux-N\'ed\'elec \cite{MR542510}, ``BLN'' for short), $\rho$ solves in the appropriate entropy sense
\begin{equation}\label{eq:reformul}
\rho_{t}+\bigl(\sign\left(x-\xi(t)\right) \, \rho \, v( \rho )\bigr)_{x} = 0,
\end{equation}
for $t > 0$ and $x \in\mathfrak{C}$.
Further, in the context of the Hughes model, the BLN interpretation of the boundary condition \eqref{eq:Dirichlet-BLN} can be seen as the ``open-end'' condition. 
More precisely, thanks to Proposition~\ref{p:refo} in Section~\ref{s:noboundary}, we will further reformulate the Cauchy-Dirichlet problem \eqref{eq:Dirichlet-BLN}, \eqref{eq:reformul} with initial data $\overline{\rho} \colon \mathfrak C \to [0, \rho_{\max}]$ as the Cauchy problem \eqref{eq:reformul} written for $t>0$, $x\in \mathbb{R}$, with initial data $\overline{\rho}$ extended by zero to $\mathbb{R}\setminus \mathfrak{C}$.

Let us now focus on the choice of the cost function. One frequently used cost function is
\begin{equation}
\label{e:crazy}
c( \rho ) := 1/v( \rho ),
\end{equation}
(under the assumption that $R_{\max}< \rho_{\max}$, ensuring that $v( \rho )$ is separated from $0$), see for instance \cite{AmadoriDiFrancesco,AmadoriGoatinRosini,ElKhatibGoatinRosini,DiFrancescoFagioliRosiniRusso,MR3698447,MR3644595,MR3619091,DiFrancescoFagioliRosiniRussoKRM,MR3460619,MR3451862,MR3277564,MR3199781,MR3055243,MR2737207,Hughes02}.
Existence results on the resulting model \eqref{e:model}, \eqref{e:crazy} are available under the assumption that the initial datum $\overline{\rho}$ has small total variation, see \cite{AmadoriGoatinRosini,DiFrancescoFagioliRosiniRussoKRM,DiFrancescoFagioliRosiniRusso,MR3644595}.
In these references, the existence proofs rely on the property that non-classical shocks will not appear if the initial datum has sufficiently small total variation,
thus ruling out the most interesting specificity of the Hughes model which corresponds, at the microscopic level, to pedestrians switching direction of their motion.

In \cite{ElKhatibGoatinRosini} the authors introduce the piecewise linear cost function
\[c( \rho ):=\begin{cases}
1&\hbox{if } \rho < \rho_{\rm c},
\\
\rho / \rho_{\rm c}&\hbox{if } \rho \geq \rho_{\rm c},
\end{cases}\]
motivated by the fact that this choice minimizes the evacuation time
in some basic situations.

In this paper we concentrate to the case of a linear cost 
\begin{equation}\label{e:cost}
c( \rho ) := 1+\alpha \, \rho .
\end{equation}	
The motivation stems from the physical meaning of its slope $\alpha \geq0$.
Indeed, $\alpha$ measures the importance given to avoid regions of high number of pedestrians.
This allows us to reproduce different behaviours with the same model, by just letting vary the parameter $\alpha$.
In fact, taking $\alpha=0$ corresponds to a panic behaviour, when people 
simply move towards the closest exit.
On the other hand, as $\alpha>0$ grows, so does the importance
of avoiding exits attracting high number of pedestrians; note that formally, taking $\alpha=\infty$ implies that we have the same number of pedestrians in the two groups corresponding to the two exits. As a timely example, social distancing rules enforced during pandemic situations correspond to higher values of $\alpha$ compared to ordinary times.

\subsection{Summary of the results and outline of the paper}

One of the main analytical features of the Hughes  model is the possible development for the solution of non-classical shocks \cite{MR1927887}.
These have a physical counterpart: pedestrians may switch direction during the evacuation of a bounded corridor $\mathfrak{C}$ through its two exits placed at $\partial\mathfrak{C}$.
In fact, pedestrians choose their direction of motion according to a weighted distance encoding the overall distribution of the crowd in $\mathfrak{C}$.
So pedestrians choose their path towards the fastest exit, taking into account the distance from the two exits as well as avoiding crowded regions.

Both in \cite{AmadoriGoatinRosini} and \cite{DiFrancescoFagioliRosiniRussoKRM} the presence of non-classical shocks is prevented by requiring some sufficient conditions, which in turn result in considering initial data with sufficiently small total variation.
On the contrary, in the present paper we assume that the initial datum has arbitrary (possibly even infinite) total variation and we consider the possible arise of non-classical shocks.
Furthermore, differently from \cite{AmadoriGoatinRosini,DiFrancescoFagioliRosiniRussoKRM}, here we allow the density to attain the value $\rho_{\max}$.

We borrow from the previous literature the adequate notion of entropy solution to the one-dimensional Hughes model, which includes entropy conditions along the turning curve.
We 
find it important to distinguish between general entropy solutions and $\mathbf{BV}$-regular solutions. Indeed, we illustrate the need of $\mathbf{BV}$-regularity for the sake of stability and uniqueness analysis, proving a special instance of stability result for solutions without non-classical shocks. It contains many of the difficulties of the general case. We prove in this special case the locally Lipschitz continuous dependence, in the $\L1$-topology, of the solution $\rho$ on the initial datum $\overline{\rho}$ and on the cost parameter $\alpha$.  
This implies uniqueness of solutions, within this restricted class, and a kind of continuous dependence of the solution on parameter $\alpha$.
It is however not clear that, in general, the evacuation time depends continuously on $\alpha$; we will come back to this aspect at the end of the paper.

Our strategy of existence analysis relies on a many-particle approximation; it continues the line initiated in  \cite{DiFrancescoFagioliRosiniRussoKRM,DiFrancescoFagioliRosiniRusso,MR3644595}, with important differences.
First, comparatively to \cite{DiFrancescoFagioliRosiniRussoKRM}, we propose a sharper many-particle Hughes model based upon an original definition of the approximate turning curve (namely, we rely upon \eqref{e:turning} rather than  \eqref{e:xiturning}): this allows us to link direction switches of particles (i.e., crossings of particles' paths with the turning curve) to the instants when one of the particles leaves the domain  $\mathfrak{C}$.
This also allows us to prove rigorously the global in time existence of a discrete solution, whereas in \cite{DiFrancescoFagioliRosiniRussoKRM} non-accumulation of switching times in the construction of the discrete solution was implicitly assumed; another consequence is that in the discrete model the evacuation time is bounded.
Second, in the present case with a linear cost, the functional defined in \cite[(9)]{DiFrancescoFagioliRosiniRussoKRM} becomes trivial and hence, it becomes useless.
Further, in contrast to the standard approaches found in the literature on the Hughes model (for both the wave-front tracking and the many-particle approximations), we propose an existence result which does not rely upon a global control of the space variation of the approximate solutions, see Theorem~\ref{t:existBVloc}. To do so, we develop original arguments based upon local variation control provided by the local regularization effect of the LWR model.
Most importantly, differently from \cite{DiFrancescoFagioliRosiniRussoKRM}, we allow the solutions to have non-classical shocks, and prove that these shocks obey the adapted entropy conditions of \cite{ElKhatibGoatinRosini,AmadoriGoatinRosini} at the turning curve. 

Our existence results are the consequence of the global construction of approximate solutions and of the convergence analysis for the many-particle approximation scheme. The convergence analysis is presented in two steps.

At a first step, in Theorem~\ref{t:exist} we propose a conditional convergence result under the assumption of the global variation control. It highlights the key arguments of passage to the limit. It takes into account the possible arise of non-classical shocks and leads to
existence of a 	$\mathbf{BV}$-regular solution, whenever 
a uniform $\mathbf{BV}$ bound for the sequence of approximate solutions can be obtained.
Let us underline that in practice, this delicate assumption seems to hold true for \lq\lq typical\rq\rq\ choices of initial data, see the numerical tests provided at the end of the paper, as well as those presented in \cite{HUANG2009127,MR3055243,MR3451862,MR3698447,MR3277564,MR3619091,MR3149318,MR3177723}.
Two unconditional applications, both excluding non-classical shocks, concern the case of two non-interacting crowds evacuating by the two exits. 
The first application result exploits a uniform bound for the turning curve, so that if the support of the initial datum is well separated from the origin, then no interaction occurs between the solution and the turning curve. This result has a limited outreach but it appears to be new; it also leads to well-posedness for specific data since our stability result is applicable in the setting of non-interacting crowds.
The second application deals with symmetric initial data. Note that existence (and uniqueness, which is immediate) for the symmetric case have been already stated both in \cite{AmadoriGoatinRosini} and \cite{DiFrancescoFagioliRosiniRussoKRM}; here we give an alternative argument.

\delaynewpage{1}
At a second step, we ``upgrade'' our convergence result. Based on the same construction we develop a less restrictive $\mathbf{BV}_{\text{loc}}$ compactness argument via local reduction to microscopic approximations of the LWR model. 
The resulting local $\mathbf{BV}$ control also requires a more delicate treatment of the passage to the limit; this is achieved via a reformulation relying on the idea that the test functions with the property of being constant in space in a vicinity of the turning curve form a dense subset of $\Cc\infty$.
With these additional tools, we assess the convergence of our many-particle scheme and achieve an unconditional existence result for general $\L\infty$ data, but without guarantee of $\mathbf{BV}$ regularity.
The numerical approximation scheme thus being rigorously assessed, we present some numerical simulations with the goal to numerically study the evacuation time as a function of the parameter $\alpha\geq0$.
We observe, at least in the case under consideration, a unique minimum and a discontinuous graph.
These aspects may deserve further investigations, starting from the development of a higher order numerical scheme.

The paper is organized as follows. 
In the next section we introduce the Hughes model.
In Section~\ref{s:main} we collect the main results of the paper. In Section~\ref{s:noboundary} we prove that the Hughes model is equivalent to the problem without boundary, that is in the whole one-dimensional space.
In Section~\ref{s:ContDep} we state and prove the stability result of Theorem~\ref{t:WS-stability} for $\textbf{BV}$-regular solutions without non-classical shocks.
In Section~\ref{s:mpa} we define a sharp many-particle analogue of the Hughes model, carefully describe its dynamics and rigorously assess its well-posedness.
In Section~\ref{s:prma} we give the proof of our first convergence result stated in Theorem~\ref{t:exist}.
We also provide applications of Theorem~\ref{t:exist} to the case of two non-interacting (separated) crowds. 
In Section~\ref{s:existBVloc}, we replace the compactness result used in the proof of Theorem~\ref{t:exist} by a local compactness argument, thus circumventing the issue of global variation control; then we improve the convergence analysis from Theorem~\ref{t:exist} and justify the general existence claim of Theorem~\ref{t:existBVloc}. 
In Section~\ref{s:numsche} we present and briefly discuss the fully discrete version of the many-particle approximation scheme, provide some numerical simulations  and make observations concerning the evacuation time.

\section{Main results}
\label{s:main}

As explained in the Introduction, the reformulation \eqref{e:cost0} of the eikonal equation contained in \eqref{e:modela} leads to the following definition of entropy solution to \eqref{e:model} based on entropy conditions for conservation laws with discontinuous flux functions. 
To make notations more concise, recall that  $f( \rho ) = \rho \, v( \rho )$ and introduce
\begin{align}
\label{e:Phi}
\Phi(t,x, \rho ,\kappa,\xi) &:= \sign\bigl(x-\xi(t)\bigr) \, \sign( \rho -\kappa) \, \bigl( f( \rho ) - f(\kappa) \bigr).
\end{align}

\begin{definition}[Entropy solution]\label{d:entro}	
A couple $(\rho,\xi) \in \Lloc1\left([0, \infty) \times \mathfrak{C};[0, \rho_{\max}]\right) \times \mathbf{Lip}([0, \infty);\mathfrak{C})$ is an entropy solution of the initial-boundary value problem \eqref{e:model} if it satisfies \eqref{e:cost0} for a.e.\ $t>0$, satisfies the entropy inequality
\begin{subequations}\label{e:entro}
\begin{align}\label{e:entro1}
0 \leq {}& \int_0^{ \infty} \int_{\mathfrak{C}} \bigl( | \rho -\kappa| \, \varphi_t + \Phi(t,x, \rho ,\kappa,\xi) \varphi_x \bigr) \, {\rm{d}} x \, {\rm{d}} t
\\&\label{e:entro3}
+2 \int_0^{ \infty} f(\kappa) \varphi\bigl(t,\xi(t)\bigr) \, {\rm{d}} t
\end{align}
\end{subequations}
for all $\kappa \in [0, \rho_{\max}]$ and all test function $\varphi \in \Cc \infty\left((0, \infty) \times \mathfrak{C};[0, \infty)\right)$;
moreover, upon choosing a suitable representative of $\rho$, there holds $\rho\in \C0([0,\infty);\L1(\mathfrak{C}))$ with the initial condition taken in the sense $\rho(0,\cdot)=\overline{\rho}$; and finally, the strong traces $\rho (\,\cdot,\pm 1^\mp)$ understood in the sense of \cite{Panov_traces2,MR1869441} verify the condition
\begin{equation}\label{e:entro2}
\hbox{for a.e.\ $t>0$,\;\; $\rho (t,\pm 1^\mp)\in [0, \rho_c]$.}
\end{equation}	
If, moreover, for all $T>0$ there holds $\rho \in \L \infty(0,T;\mathbf{BV}(\mathfrak{C}))$, then we say that $(\rho,\xi)$ is a $\mathbf{BV}$-regular solution.
\end{definition}

\delaynewpage{2}
In the above definition, the line \eqref{e:entro1} originates from the classical Kruzhkov definition \cite{Kruzhkov} of entropy solution of a Cauchy problem. The line \eqref{e:entro3} accounts for entropy admissibility condition at the discontinuity of the flux along the turning curve, see \cite{MR2086124,MR2195983,MR1770068,MR1961002,MR2024741,MR2209759,MR2300671,MR3425264} for analogous definitions. As a matter of fact, including \eqref{e:entro3} implicitly  prescribes the coupling, across the ``turning curve'' $\xi$, of the two crowds heading to the two exits, each of the two crowds being governed by the standard LWR equation  with unknown data at the interface $x=\xi(t)$.
The requirement \eqref{e:entro2} represents the classical BLN (Bardos et al.\ \cite{MR542510}) interpretation of the homogeneous Dirichlet  boundary condition \eqref{e:modelb}, see also \cite{AndrSbihi,MR1980978,MR2168427,MR1387428}.
We stress that the strict concavity of the flux $f \colon \rho \mapsto \rho \, v( \rho )$ (see \eqref{V1}) guarantees the existence at the boundary points $x=\pm 1$ of the strong one-sided traces of any $\L1\left([0, \infty) \times \mathfrak{C};[0, \rho_{\max}]\right)$ function $\rho$ verifying \eqref{e:entro}, see \cite{Panov_traces2,MR1869441}, thus giving sense to  \eqref{e:entro2}.
The initial condition could as well be formulated, analogously to the boundary ones, in the sense of initial trace \cite{Panov_traces1,MR1869441}; however, the stronger time-continuity property is automatic if $\rho$ verifies \eqref{e:entro1}, according to the following remark.
\begin{remark}\label{rem:L1-continuity}
Note that away from the turning curve $x=\xi(t)$, \eqref{e:entro} prescribes that $\rho$ is a local entropy solution of a standard homogeneous conservation law of LWR kind.
Therefore a measurable function $\rho$ taking values in $[0, \rho_{\max}]$ and fulfilling \eqref{e:entro} for all $\kappa \in [0, \rho_{\max}]$ and all test function $\varphi \in \Cc \infty\left((0, \infty) \times \mathfrak{C};[0, \infty)\right)$ can be normalized to fulfill $\rho\in \C0([0,\infty);\Lloc1(\mathfrak{C}))$. This follows by a straightforward calculation from the uniform boundedness of solutions and from the local time continuity result of \cite{CancesGallouet} applied on rectangles $(\tau_0,\tau_0+\delta)\times (-1,\xi(\tau_0)-\sigma)$ and $(\tau_0,\tau_0+\delta)\times (\xi(\tau_0)+\sigma,1)$, where $\tau_0>0$ is arbitrary, $\delta>0$ and $\sigma>0$ are small enough so that these rectangles do not cross the turning curve $\{(t,x)\,:\, x=\xi(t)\}$. We refer to \cite{Sylla} for details of the calculation. Further, because the interval $\mathfrak{C}$ is finite and $\rho$ is bounded, $\Lloc1$ can be replaced by $\L1$ in the resulting continuity property. 
\end{remark}

In Section~\ref{s:noboundary}, we will identify the problem addressed in Definition~\ref{d:entro} with the problem \eqref{e:cost0}, \eqref{eq:reformul}, \eqref{e:modelc} considered for $x\in \mathbb{R}$, meaning that the initial datum $\overline{\rho}(x)$ is extended by zero for $x \in \mathbb{R}\setminus\mathfrak{C}$ and the boundaries are suppressed, see Definition~\ref{d:entro-bis}.

We are now in the position to state our results.
First, we illustrate the importance of the $\mathbf{BV}$-regularity by carrying out a stability analysis for a restricted class of solutions; we show, in particular, the $\L1$-continuous dependence of $\rho$ on the initial datum $\overline{\rho}$ and on the cost parameter $\alpha$.
To this end, we introduce the following notion of well-separated solution (particular instances of such solutions are considered in Corollaries~\ref{t:exist1}, \ref{t:exist2}).
\begin{definition}[Well-separated solution]
\label{d:well-separated}
An entropy solution $(\rho,\xi)$ of the initial-boundary value problem \eqref{e:model} is called well-separated if the limits $\rho(t,\xi(t)^\pm)$ (well defined for a.e.\ $t>0$) are equal to zero.
\end{definition}
We recall that existence of strong one-sided traces of $\rho$ on the Lipschitz curve $x=\xi(t)$ is automatic: it follows from the local entropy inequalities and the strict non-linearity of the flux \cite{MR1869441}.
We then assess the following stability claim.
\begin{theorem}[Stability of $\mathbf{BV}$-regular well-separated solutions]
\label{t:WS-stability} 
For $i\in\{1,2\}$, let $( \rho_{i},\xi_{i})$ well-separated entropy solution of the initial-boundary value problem \eqref{e:model} corresponding to the initial datum $\overline{\rho}_{i}$ and the cost parameter $\alpha_{i}$. Assume moreover that $\rho_1$ is  $\mathbf{BV}$-regular.
For a.e.\ $t>0$ there exists a constant $C$ that depends, in a non-decreasing way, only on $t$ and $\max\{\alpha_{1},\alpha_{2}\}$, such that
\begin{equation}\label{eq:stability}
\left\| \rho_{1}(t,\cdot\,)- \rho_{2}(t,\cdot\,)\right\|_{\L1(\mathfrak{C})} \leq C\bigl(|\alpha_{1}-\alpha_{2}|+\left\|\overline{\rho}_{1}-\overline{\rho}_{2}\right\|_{\L1(\mathfrak{C})}\bigr).
\end{equation}
In particular, for any fixed cost parameter $\alpha$ and initial datum $\overline{\rho}$ there exists at most one $\mathbf{BV}$-regular well-separated solution.
\end{theorem}

\delaynewpage{2}
We emphasize that the $\mathbf{BV}$-regularity is a crucial ingredient of the stability argument; indeed, being understood that $\xi_{1}\not\equiv \xi_{2}$, we have to call upon quantitative stability of solutions of conservation laws $\rho_t + f(t,x, \rho )_x=0$ with respect to perturbations of the flux $f$. All such results (see \cite{BouchutPerthame,KarlsenRisebro,Mercier} for the best known ones) rely upon the $\mathbf{BV}$ in space estimate on at least one of the two solutions.
Moreover, the open-end reformulation of Definiton~\ref{d:entro-bis} is instrumental in our stability analysis. Finally, we stress that the well-separation assumption is not merely technical; dropping it would require new ideas for the sake of uniqueness analysis.

Turning to the issue of existence, we are able to bypass the well-separation assumption.
The starting point for existence is provided by the accurate construction of discrete solutions based upon the modification, introduced in Section~\ref{s:mpa}, of the many-particle scheme for Hughes' model considered in \cite{DiFrancescoFagioliRosiniRusso,MR3644595,DiFrancescoFagioliRosiniRussoKRM}. This discrete Hughes model is studied in detail for its own sake (see, in particular, Lemmas~\ref{l:bounxi} - \ref{l:maxpri}, Proposition~\ref{p:Opeth} and Theorem~\ref{th:discrete-WP}), and for the sake of preparing grounds for its asymptotic analysis as the number of particles goes to infinity.
Having in mind the importance of the $\mathbf{BV}$-regularity highlighted by the result and the method of proof of Theorem~\ref{t:WS-stability}, first we 
assess consistency and compactness of the many-particle approximation of the Hughes model provided uniform variation bounds are available through the construction. We arrive at the following conditional convergence and existence result.

\begin{theorem}[Conditional convergence for the ${\mathbf{BV}}$-case]\label{t:exist}
Consider the cost function \eqref{e:cost} and assume \eqref{V1}.
For any initial datum $\overline{\rho}$ in ${\mathbf{BV}}(\mathfrak{C};[0, \rho_{\max}])$ satisfying \eqref{I}, let
$\{( \rho^{n},\zeta^{n})\}_n$ be the approximate solutions constructed in Section~\ref{s:mpa}. 
Assume that for all $T>0$  there exists a constant $\hbox{\bf TV}=\hbox{\bf TV}(T)>0$ such that for any $t\in [0,T]$ and $n\in{\mathbb{N}}$ we have
\begin{equation}
\label{e:DavidMaximMicic}
\TV\left( \rho^{n}(t,\cdot\,)\right) \leq \hbox{\bf TV}.
\end{equation} 
Then, $\{( \rho^{n},\zeta^{n})\}_{n}$ converges, up to a subsequence,  in $\L1((0,T)\times\mathfrak{C})\times \C0([0,T])$ for all $T>0$, to a $\mathbf{BV}$-regular entropy solution $(\rho,\xi)$ of the initial-boundary value problem \eqref{e:model} in the sense of Definition~\ref{d:entro}. 
\end{theorem}

In Section~\ref{s:applications} we provide two applications; one of them corresponds to an already known existence result, but the convergence of the many-particle approximation is assessed for the first time.  
In these results, the control on the total variation of the approximate solutions is due to the absence of interaction of particles with the turning curve.
Note that the settings of the existence results of Corollaries~\ref{t:exist1}, \ref{t:exist2} fit the stability framework of Theorem~\ref{t:WS-stability}. 
In presence of particles' interactions with the turning curve (which corresponds to the presence of non-classical shocks in the limit solution), detailed insight into the analysis of variation bounds  can be found in \cite{MaxGra-preprint}; unfortunately it is not enough to assess the global space variation bound \eqref{e:DavidMaximMicic}. 

Then, we circumvent the use of  \eqref{e:DavidMaximMicic} in order to address solutions with non-classical shocks. We pursue an original line of investigation which relies upon a reduction to one-sided Lipschitz regularization properties of the standard LWR model throughly studied in \cite{DiFrancescoRosini,DiFrancescoFagioliRosini-BUMI}, see also \cite{HoldenRisebro1,HoldenRisebro2}. This approach
does not ensure the $\mathbf{BV}$-regularity of the constructed solutions, but guarantees existence for the one-dimensional Hughes' model \eqref{e:model} for general $\L\infty$ data via a less stringent $\mathbf{BV}_{\text{loc}}$ compactness argument and a refined convergence analysis in a vicinity of the turning curve. 
\begin{theorem}[Existence beyond the $\mathbf{BV}$ control]
Consider the cost function \eqref{e:cost} and assume \eqref{V1}, \eqref{Vbis}. 
For any initial datum $\overline{\rho}$ satisfying \eqref{I}, there exists an entropy solution to problem \eqref{e:model}. More precisely, the sequence of approximate solutions $\{( \rho^{n},\zeta^{n})\}_{n}$
constructed in Section~\ref{s:mpa} converges, up to a subsequence, in $\L1((0,T)\times \mathfrak{C})\times \C0([0,T])$ for all $T>0$, to an entropy solution $(\rho,\xi)$ of the initial-boundary value problem \eqref{e:model} in the sense of Definition~\ref{d:entro}.
\label{t:existBVloc}
\end{theorem}

\delaynewpage{2}
\section{Reduction to a problem without boundary}\label{s:noboundary}

This section is devoted to the justification of the fact that, in the setting of the one-dimensional Hughes model, the Dirichlet boundary conditions at exits \eqref{e:modelb} are mere open-end conditions corresponding to the problem without boundary in the whole one-dimensional space. The solutions to the latter are defined as follows, having in mind initial data supported in $\mathfrak{C}$.

\begin{definition}[Open-end formulation]
\label{d:entro-bis}
Consider a measurable initial datum $\overline{\rho} \colon \mathbb{R} \to [0, \rho_{\max}]$.
A couple $(\rho,\xi) \in \Lloc1\left([0, \infty) \times \mathbb{R};[0, \rho_{\max}]\right) \times \mathbf{Lip}([0, \infty);\mathfrak{C})$ is an entropy solution of the initial-value problem \eqref{e:cost0}, \eqref{eq:reformul}, \eqref{e:modelc} if it satisfies \eqref{e:cost0} for a.e.\ $t>0$, and satisfies the entropy inequality
\begin{align}
\label{e:entro-bis}
0 \leq {}& \int_0^{ \infty} \int_\mathbb{R} \bigl( | \rho -\kappa| \, \varphi_t + \Phi(t,x, \rho ,\kappa,\xi) \, \varphi_x \bigr) \, {\rm{d}} x \, {\rm{d}} t
+2 \int_0^{ \infty} f(\kappa) \, \varphi\bigl(t,\xi(t)\bigr) \, {\rm{d}} t
\end{align}
for all $\kappa \in [0, \rho_{\max}]$  and all test function $\varphi \in \Cc \infty\left((0, \infty) \times \mathbb{R};[0, \infty)\right)$, with $\Phi$ defined as in \eqref{e:Phi};
moreover, upon choosing a suitable representative of $\rho$, there holds $\rho\in \C0([0,\infty);\L1(\mathbb{R}))$ with the initial condition taken in the sense $\rho(0,\cdot)=\overline{\rho}$. 
\end{definition}
\begin{proposition}\label{p:refo}
Given $\overline{\rho} \colon \mathfrak{C} \to [0, \rho_{\max}]$, let us identify $\overline{\rho}$ with its extension to the whole of $\mathbb{R}$ by the value zero on $\mathbb{R}\setminus \mathfrak{C}$. 
Then the following problems are equivalent:
\begin{equation}
\label{e:IBVP}
\left\{\begin{array}{@{}l@{\qquad}l@{\ }l@{}}
\int_{-1}^{\xi(t)}c\bigl( \rho(t,y)\bigr) \, {\rm{d}} y = \int_{\xi(t)}^{1}c\bigl( \rho(t,y)\bigr) \, {\rm{d}} y, && t\geq 0,\\
\rho_{t}+\bigl(\sign\left(x-\xi(t)\right) \, \rho \, v( \rho )\bigr)_{x} = 0, &x\in \mathfrak{C},&t>0,\\
\rho (0,x) = \overline{\rho}(x), &x\in \mathfrak{C},\\
\rho(t,\pm 1)=0, &&t>0,
\end{array}\right.
\end{equation}
where the solution $(\rho,\xi)$ is understood in the sense of Definition~\ref{d:entro}, and
\begin{equation}
\label{e:IVP}
\left\{\begin{array}{@{}l@{\qquad}l@{\ }l@{}}
\int_{-1}^{\xi(t)}c\bigl( \rho(t,y)\bigr) \, {\rm{d}} y = \int_{\xi(t)}^{1}c\bigl( \rho(t,y)\bigr) \, {\rm{d}} y, && t\geq 0,\\
\rho_{t}+\bigl(\sign\left(x-\xi(t)\right) \, \rho \, v( \rho )\bigr)_{x} = 0, &x\in \mathbb{R},&t>0,\\
\rho (0,x) = \overline{\rho}(x), &x\in \mathbb{R},
\end{array}\right.
\end{equation}
where the solution $(\rho,\xi)$ is understood in the sense of Definition~\ref{d:entro-bis}.

More precisely, given $\overline{\rho}$, any solution to  Problem~\eqref{e:IBVP} can be seen as the restriction to $x\in \mathfrak{C}$ of a solution to Problem~\eqref{e:IVP}; reciprocally, 
any solution to Problem~\eqref{e:IVP} can be seen as an extension to $x\in \mathbb{R}$ of a solution to Problem~\eqref{e:IBVP}.
\end{proposition}	
\begin{proof}
Let us first prove that the component $\rho$ of the solution to Problem~\eqref{e:IBVP}  (defined for $x \in \mathfrak{C}:=(-1,1)$, verifying the entropy inequality \eqref{e:entro}) can be extended to $x \in \mathbb{R}$ in a way compatible with the entropy inequality \eqref{e:entro-bis}. 
Observe that, due to the shape assumption on $f:\rho\mapsto \rho \, v(\rho)$ contained in \eqref{V1}, the BLN interpretation \cite{MR542510} of the boundary condition $\rho(t,1^-)=0$ reads: 
\begin{equation*}
\text{for a.e.\ $t>0$,\;\; $\rho(t,1^-) \in [0, \rho_{\rm c}]$}.
\end{equation*}
We recall that the trace $\rho(t,1^-)$ exists in the strong $\Lloc1$ sense, see \cite{Panov_traces2}.	
Denote $\Omega_0:=(0,\infty)\times\mathfrak{C}$ the domain of definition of $\rho$ and consider, in the quarter-plane $\Omega_+:=(0,\infty)\times(1,\infty)$, the Cauchy-Dirichlet problem with the boundary datum at $x=1^+$ given by the trace $\rho(t,1^-)$ and with zero initial condition. 
The existence of a unique entropy solution  to this problem is standard (see, e.g., \cite{AndrSbihi}); let us denote it by $\rho_+$. Moreover, the maximum principle implies that, for a.e. $t>0$, the trace $\rho_+(t,1^+)$ belongs to $[0, \rho_{\rm c}]$ because both the initial and the boundary data take their values in $[0, \rho_{\rm c}]$. 
Using again the BLN interpretation of the boundary condition (or using the machinery of \cite{AndrSbihi}), we find that one must have $\rho_+(t,1^+)= \rho(t,1^-)$ for a.e.\ $t>0$. 
Further, observe that the two entropy solutions $\rho$ and $\rho_+$ defined in the adjacent subdomains $\Omega_0$ and $\Omega_+$ can be pieced together continuously across the boundary $\Gamma_+ :=(0,\infty)\times\{1\}$ shared by the two subdomains: the resulting function is an entropy solution in $\Omega_0 \cup \Gamma_+ \cup \Omega_+ = (0,\infty)\times(-1,\infty)$. 
In the same way, we extend $\rho$ to $x<-1$ by setting it equal to the entropy solution $\rho_-$ of the analogous Cauchy-Dirichlet problem set in the quarter-plane $\Omega_-:=(0,\infty)\times(-\infty,-1)$, with the boundary datum given by the trace $\rho(t,-1^+)$ and zero initial condition. The resulting extension of $\rho$ fulfills \eqref{e:entro-bis} and thus it solves Problem~\eqref{e:IVP}, because \eqref{e:cost0} is in common in Problems~\eqref{e:IBVP} and~\eqref{e:IVP}.

Reciprocally, consider a solution to Problem~\eqref{e:IVP} corresponding to data which are zero on $\mathbb{R}\setminus \mathfrak{C}$. Note that \eqref{e:entro} is immediate by restricting the support of the test function. It remains to justify \eqref{e:entro2}.
Note that $\xi$ takes values in $\mathfrak{C}$ and that, on each side from $\xi$, $\rho$ solves a scalar conservation law with strictly convex or strictly concave flux, so that the theory of generalized characteristics \cite{Dafermos} applies. 
Considering minimal backward characteristics issued from a point $(t,1)$, $t>0$, we see that either $\rho(t,1^-)=0$ or $f'( \rho(t,1^-))>0$, so that  $\rho(t,1^-)\in [0, \rho_c]$ in all cases. 
Similarly, $\rho(t,-1^+)\in [0, \rho_c]$. 
Consequently, $\rho$ restricted to $x\in\mathfrak{C}$ fulfills \eqref{e:entro2}.
\qed
\end{proof}

\section{Stability for well-separated $\mathbf{BV}$-regular solutions}\label{s:ContDep}

In this section, we prove Theorem~\ref{t:WS-stability}.
The claim of the theorem follows by a straightforward combination of two distinct ingredients provided in Propositions~\ref{prop:rho-of-xi} and~\ref{prop:xidot}; we state both, before turning to their proofs.

\begin{proposition}\label{prop:rho-of-xi}
For $i\in\{1,2\}$, let $( \rho_{i},\xi_{i})$ be a well-separated solution of the initial-boundary value problem \eqref{e:model} corresponding to the initial datum $\overline{\rho}_{i}$ and the cost parameter $\alpha_{i}$. 
Assume, moreover, that $\rho_1$ is $\mathbf{BV}$-regular.
Then, for a.e.\ $t>0$ there holds
\begin{equation}
\begin{aligned}
&\left\| \rho_{1}(t,\cdot\,)- \rho_{2}(t,\cdot\,)\right\|_{\L1(\mathfrak{C})} 
\\\leq\ &\left\|\overline{\rho}_{1}-\overline{\rho}_{2}\right\|_{\L1(\mathfrak{C})}
+2 C_1(t) \left(|\xi_{1}(0)-\xi_{2}(0)|+ \int_0^t |\dot{\xi}_{1}(s)-\dot{\xi}_{2}(s)| \, {\rm{d}} s \right),
\end{aligned}
\label{eq:rho-estim}
\end{equation}
where $C_1(t) := \sup_{s \in[0,t]} \left\{\TV\left( \rho_{1}(s,\cdot\,)\right)\right\}$.
\end{proposition}	

\begin{proposition}\label{prop:xidot}
For $i\in\{1,2\}$, let $( \rho_{i},\xi_{i})$ be a well-separated solution of the initial-boundary value problem \eqref{e:model} corresponding to initial datum $\overline{\rho}_{i}$ and the cost parameter $\alpha_{i}$. 
Then there exist $\tau>0$, $K>0$ that depend only on $\max\{\alpha_{1},\alpha_{2}\}$, such that
\begin{equation}\label{eq:integral-dotxi-estim}
\int_0^\tau|\dot{\xi}_{1}(s)-\dot{\xi}_{2}(s)|\, {\rm{d}} s \leq K\Bigl(|\alpha_{1}-\alpha_{2}|+\left\|\overline{\rho}_{1}-\overline{\rho}_{2}\right\|_{\L1(\mathfrak{C})}\Bigr);
\end{equation}
moreover, we have
\begin{equation}\label{eq:xi(0)-estim}
|\xi_{1}(0)-\xi_{2}(0)| \leq K\Bigl(|\alpha_{1}-\alpha_{2}|
+\left\|\overline{\rho}_{1}-\overline{\rho}_{2}\right\|_{\L1(\mathfrak{C})}\Bigr).
\end{equation}
\end{proposition}	

\delaynewpage{3}
\begin{proof}[of Theorem~\ref{t:WS-stability}]
The result for $t \in [0,\tau]$ follows by simply plugging the bounds \eqref{eq:integral-dotxi-estim}, \eqref{eq:xi(0)-estim} of Proposition~\ref{prop:xidot} into estimate \eqref{eq:rho-estim} of Proposition~\ref{prop:rho-of-xi}. 
Recall that $\tau$ depends only on $\max\{\alpha_{1},\alpha_{2}\}$. 
Then we bootstrap the argument. 
Indeed, from the time continuity result of \cite{CancesGallouet} for local entropy solutions of scalar conservation laws, it is easily derived that $\rho_{1}$, $\rho_{2}$ are continuous in time with values in $\L1(\mathfrak{C})$ (cf.\ \cite{Towers-BV,Sylla} for details, in a similar situation). 
This permits to use the stop-and-restart procedure, taking $t=k\,\tau$ for initial time, $k \in\mathbb{N}$. 
Observe that due to the BLN interpretation of the boundary condition, for $t>k \, \tau$ we have that $( \rho_{1}|_{\mathfrak{C}}(t,\cdot),\xi_{1})$ and $( \rho_{2}|_{\mathfrak{C}}(t,\cdot),\xi_{2})$ depend solely on $\rho_{1}|_{\mathfrak{C}}(k \, \tau,\cdot)$ and $\rho_{2}|_{\mathfrak{C}}(k \, \tau,\cdot)$. Hence we can recursively plug the estimate of $\left\| \rho_{1}(k \, \tau,\cdot\,) - \rho_{2}(k \, \tau,\cdot\,) \right\|_{\L1(\mathfrak{C})}$ into the estimate 
\begin{align*}
\left\| \rho_{1}(t,\cdot\,)- \rho_{2}(t,\cdot\,)\right\|_{\L1(\mathfrak{C})} \leq K\left(|\alpha_{1}-\alpha_{2}| +\left\| \rho_{1}(k \, \tau,\cdot\,)- \rho_{2}(k \, \tau,\cdot\,)\right\|_{\L1(\mathfrak{C})}\right)
\\
\forall t \in[k \, \tau,(k+1) \, \tau]&.
\end{align*}
Thus we extend the control of $\left\| \rho_{1}(t,\cdot\,)- \rho_{2}(t,\cdot\,)\right\|_{\L1(\mathfrak{C})}$ of the form \eqref{eq:stability} to any time $t>0$, with the constant $C$ that grows with $t$ like $K^{m+1}$, being $m:=\hbox{Ent}(t/\tau)$.\qed
\end{proof}	 

\begin{proof}[of Proposition~\ref{prop:rho-of-xi}]
The proof is analogous to the arguments that can be found in \cite{AndrLagoTakaSeguin14,DelleMonacheGoatin17,Sylla}, regarding continuous dependence of solutions of discontinuous-flux conservation laws with respect to a moving interface. 
For application of the argument to the Hughes model at hand, we rely upon the reformulation ``without boundary'': by Proposition~\ref{p:refo}, we can consider that $\rho_{1}$, $\rho_{2}$ are defined for all $x \in\mathbb{R}$.

In the subdomains $\{(t,x) : \pm (x-\xi_{1}(t))>0\}$ the function $\rho_{1}$ solves the conservation laws $\partial_t \rho_{1}+\partial_x(\pm f( \rho_{1}))=0$ in the standard Kruzhkov entropy sense. 
It is readily checked, by a change of variable in the entropy formulation (written with Lipschitz continuous test functions), that in the subdomains $\Theta_\pm:=\{(t,x) : \pm(x-\xi_{2}(t))>0\}$ the translated function $\rho_{3}(t,x):= \rho_{1}(t,x-\xi_{2}(t)+\xi_{1}(t))$ solves the scalar conservation law with translated flux
\begin{equation}\label{eq:auxil-equ}
\partial_t \rho_{3} + \partial_x\left( \pm f( \rho_{3}) + (\dot{\xi}_{2}(t) - \dot{\xi}_{1}(t)) \rho_{3} \right) = 0
\end{equation}
in the standard Kruzhkov entropy sense. 
Recall that $\rho_{2}$ solves the analogous problem with the fluxes $\pm f( \rho_{2})$ in the same subdomains $\Theta_\pm$. 
We then use the standard estimate of continuous dependence on the flux within the doubling of variables argument \cite{BouchutPerthame,KarlsenRisebro,Mercier}: for all smooth non-negative $\varphi$ compactly supported in $\Theta_- \cup \Theta_+$, we have
\begin{align}\nonumber
&- \int_0^{\infty} \int_\mathbb{R} \Bigl(
\begin{aligned}[t]
\varphi_t \left| \rho_{3}(t,x)- \rho_{2} (t,x)\right|
+ \varphi_x \, \sign\bigl(x-\xi_{2}(t)\bigr) \, \sign\bigl( \rho_{3}(t,x)- \rho_{2}(t,x)\bigr)\\
\times\bigl(f\bigl( \rho_{3}(t,x)\bigr)- f\bigl( \rho_{2}(t,x)\bigr)\bigr)\Bigr) {\rm{d}} x\, {\rm{d}} t
\end{aligned}\\
\leq\ & \int_\mathbb{R} \varphi(0,x) \left|\overline{\rho}_{3}(x)-\overline{\rho}_{2}(x)\right|\, {\rm{d}}x+ \int_0^{\infty} \left\|\varphi(t,\cdot\,)\right\|_\infty \, |\dot{\xi}_{2}(t)-\dot{\xi}_{1}(t)| \, \TV\bigl(\rho_{3}(t,\cdot\,)\bigr)\, {\rm{d}} t,
\label{eq:ContDepLocal}
\end{align} 	
where $\overline{\rho}_{3}(x):=\overline{\rho}_{1}(x-\xi_{2}(0)+\xi_{1}(0))$ is the initial condition for the solution $\rho_{3}$ of \eqref{eq:auxil-equ}.

Now, note that for well-separated solutions, it is straightforward to drop the assumption that $\varphi$ is zero in a neighborhood of the curve $\{(t,x) : x=\xi_{2}(t)\}$. 
Indeed, take general $\varphi \in \Cc{ \infty}([0, \infty) \times \mathbb{R})$ and consider truncated test functions $\varphi_m(t,x) := \varphi(t,x) \, \eta(m \, (x-\xi_{2}^{m}(t))) \in \Cc{ \infty}(\Theta_- \cup \Theta_+) $. 
Here $\xi_{2}^{m}$ is a smooth approximation of $\xi_{2}$ such that $\dot{\xi}_{2}$ is uniformly bounded and $\left\|\xi_{2}-\xi_{2}^{m}\right\|_\infty \leq 1/m$ (such approximation is constructed by convolution); $\eta \in\mathbf{C}^ \infty(\mathbb{R})$ is even, $\eta'(z) \geq 0$ for $z>0$, $\eta(z)$ equals $1$ for $z\geq 1$, and $\eta\equiv 0$ in a neighborhood of $z=0$. 
Upon substituting $\varphi_m$ into \eqref{eq:ContDepLocal}, the integrals of the terms
\begin{gather*}
m \, \eta'\bigl(m (x-\xi_{2}^{m}(t))\bigr) \, \dot{\xi}_{2}^{m}(t) \left| \rho_{3}- \rho_{2}\right| \varphi,\\
m \, \eta'\bigl(m (x-\xi_{2}^{m}(t))\bigr) \sign\bigl( \rho_{3}(t,x)- \rho_{2}(t,x)\bigr) \, \bigl(f\bigl( \rho_{3}(t,x)\bigr) - f\bigl( \rho_{2}(t,x)\bigr)\bigr) \, \varphi
\end{gather*}
vanish as $m\to \infty$ due to the assumption that $\rho_{1}$, $\rho_{2}$ are well separated, since it means that the traces of $\rho_{3}$, $\rho_{2}$ on the curve $x=\xi_{2}(t)$ are zero. Indeed, one can perform the change of variables $y := m\,(x-\xi_{2}(t))$ in the integrals of these terms. 
The dominated convergence can be applied since $\rho_{2,3}(t,\xi_{2}(t)+\frac{y}{m})\to 0$ pointwise, as $m\to \infty$, while the support of $y\mapsto \eta'\bigl(y+m\,(\xi_{2}(t)-\xi_{2}^{m}(t))\bigr)$ lies within the fixed interval $[-2,2]$, by the choice of $\eta$ and of $\xi_{2}^{m}$.

\delaynewpage{2}
Then, as in the standard Kruzhkov $\L1$-contraction argument, we can let $\varphi$ converge to the indicator function of $[0,t) \times \mathbb{R}$ and infer
\begin{equation}\label{eq:ContDepGlobal}
\int_\mathbb{R} | \rho_{3}(t,x)- \rho_{2} (t,x)|\, {\rm{d}} x 
\leq \int_\mathbb{R} |\overline{\rho}_{3}(x)-\overline{\rho}_{2}(x)|\, {\rm{d}} x+ C_{1}(t) \int_0^{t} |\dot{\xi}_{1}(s)-\dot{\xi}_{2}(s)|\, {\rm{d}} s,
\end{equation} 
where $C_1(t) := \sup_{s \in[0,t]} \left\{\TV\left( \rho_{1}(s,\cdot\,)\right)\right\}$ and
we also used the fact that, by construction, $\TV( \rho_{3}(s,\cdot\,))=\TV( \rho_{1}(s,\cdot\,))$.
From the definition of $\rho_{3}$, we also infer
\begin{align*}
\int_\mathbb{R} |\overline{\rho}_{1}(x)-\overline{\rho}_{3}(x)|\, {\rm{d}} x 
& \leq C_{1}(0) \, |\xi_{1}(0)-\xi_{2}(0)|
\leq C_{1}(t) \, |\xi_{1}(0)-\xi_{2}(0)| ,\\
\int_\mathbb{R} | \rho_{1}(t,x)- \rho_{3} (t,x)|\, {\rm{d}} x 
& \leq C_{1}(t) \, |\xi_{1}(t)-\xi_{2}(t)| \\
& \leq C_{1}(t) \left(|\xi_{1}(0)-\xi_{2}(0)|+ \int_0^t |\dot{\xi}_{1}(s)-\dot{\xi}_{2}(s)|\, {\rm{d}} s\right).
\end{align*}
Assembling these bounds with \eqref{eq:ContDepGlobal} via the triangle inequality, we infer the claim of the proposition.\qed
\end{proof}

\begin{proof}[of Proposition~\ref{prop:xidot}]
We start by assessing \eqref{eq:xi(0)-estim}. 
It follows from \eqref{e:cost0} that 
\[\int_{-1}^{\xi_{2}(0)} c_2(\overline{\rho}_{2}(x))\, {\rm{d}} x - \int_{-1}^{\xi_{1}(0)} c_{1}(\overline{\rho}_{1}(x))\, {\rm{d}} x 
= \int^1_{\xi_{2}(0)} c_2(\overline{\rho}_{2}(x))\, {\rm{d}} x - \int^1_{\xi_{1}(0)} c_{1}(\overline{\rho}_{1}(x))\, {\rm{d}} x,\]
which can be rewritten as	
\[2 \int_{\xi_{1}(0)}^{\xi_{2}(0)} c_2(\overline{\rho}_{2}(x))\, {\rm{d}} x 
= \int^1_{-1} \sign(x-\xi_{1}(0)) \bigl(c_2\left(\overline{\rho}_{2}(x)\right)-c_{1}\left(\overline{\rho}_{1}(x)\right) \bigr)\, {\rm{d}} x.\]
The choice of $c_i=1+\alpha_i \, \rho$, $i\in\{1,2\}$, now yields
\[2 \, |\xi_{2}(0)-\xi_{1}(0)| \leq \int_{-1}^1 |\alpha_{2} \, \overline{\rho}_{2}(x) - \alpha_{1} \, \overline{\rho}_{1}(x)|\, {\rm{d}} x,\]
which readily leads to \eqref{eq:xi(0)-estim} having in mind that $\left\|\overline{\rho}_{2}\right\|_{\L1(\mathfrak C)} \leq 2 \rho_{\max}$.

\smallskip
Now, let us admit for a while the explicit expression for $\dot{\xi}_{1}$, which can be obtained by a formal calculation. 
Keeping in mind that $c_{1}( \rho_{1}(t,\xi_{1}(t)^\pm))=c_{1}(0)=1$ due to the assumption that $\rho_{1}$ is a well-separated solution, substituting $\partial_x(\pm f( \rho_{1}))$ in the place of $\partial_t \rho_{1}$ for $\pm (x-\xi_{1}(t))>0$, we exhibit the formula
\begin{align}\label{eq:dotxi}
&\dot{\xi}_{1}(t)
= -\frac{\alpha_1}{2} \int_{-1}^1 f( \rho_{1}(t,x))_x \, {\rm{d}} x
= \frac{\alpha_1}{2} \Bigl(f\bigl( \rho_{1}(t,-1^+)\bigr) - f\bigl( \rho_{1}(t,1^-)\bigr)\Bigr)&
&
\end{align}
for a.e.\ $t>0$. The rigorous assessment of \eqref{eq:dotxi} is postponed to the end of the proof.

We now exploit the expression \eqref{eq:dotxi} for $\dot{\xi}_{1}$ (and the analogous expression for $\dot{\xi}_{2}$) in order to reach to \eqref{eq:integral-dotxi-estim}, for appropriately defined $\tau>0$ and $K>0$. Set
\begin{align*}
&\Lambda:=\max\{ \|f'\|_\infty, \|\dot{\xi}_{1}\|_\infty, \|\dot{\xi}_{2}\|_\infty\},\\
&\begin{aligned}
&x_*:=\min\{\xi_{1}(0),\xi_{2}(0)\},&
&x^*:=\max\{\xi_{1}(0),\xi_{2}(0)\},\\
&\tau^*:=\frac{1-x^*}{\Lambda},&
&\tau_*:=\frac{x_*+1}{\Lambda},&
&\tau:=\min\left\{\tau^*,\tau_*\right\}.
\end{aligned}
\end{align*}
Because the cost $c_{1}$ takes values in $[1,\max\{\alpha_{1},\alpha_{2}\} \rho_{\max}]$, it is easily seen from \eqref{e:cost0} that $\xi_{1}(0)$, $\xi_{2}(0)$ belong to $[-1+\delta,1-\delta]$ for some $\delta>0$ that only depends on $\max\{\alpha_{1},\alpha_{2}\}$; moreover, in view of \eqref{eq:dotxi}, $\|\dot{\xi}_{1}\|_\infty$, $\|\dot{\xi}_{2}\|_\infty$ are bounded by a constant times $\max\{\alpha_{1},\alpha_{2}\}$. Therefore $\tau>0$ depends only on $\max\{\alpha_{1},\alpha_{2}\}$.

Now, let $\mathcal{T}^*$ be the interior of the triangle with vertices $(x^*,0)$, $(1,0)$ and $(1,\tau^*)$.
By the definition of $x^*$ and of $\Lambda$, both $\rho_{1}$ and $\rho_{2}$ verify in $\mathcal{T}^*$ the same homogeneous scalar conservation law with flux $f$ that is non-affine on any interval.
In this situation, strong traces of $\rho_{1}$, $\rho_{2}$ as $t\to 0^+$ (the initial trace) and as $x\to 1^-$ (the boundary trace) exist, see \cite{Panov_traces1,Panov_traces2}. 
It follows that, first, the so-called Kato inequality in $\mathcal{T}^*$ is fulfilled:
\begin{align*}
- \iint_{\mathcal{T}^*} \Bigl(| \rho_{1}- \rho_{2}|\varphi_t+ \sign( \rho_{1}- \rho_{2})\bigl(f( \rho_{1})-f( \rho_{2})\bigr)\varphi_x \Bigr)\, {\rm{d}} x \, {\rm{d}} t \leq 0&&\forall \varphi \in \Cc\infty(\mathcal{T}^*).
\end{align*}
Second, it follows that one can proceed as in the classical setting of Kruzhkov \cite{Kruzhkov}, approximating the characteristic function of $\mathcal{T}^*$ by a sequence of $\varphi \in \mathbf{C}^ \infty_0(\mathcal{T}^*)$; note that, like in \cite{Kruzhkov}, we have chosen the slope $\Lambda$ of the oblique part of the boundary of $\mathcal{T}^*$ larger than $\|f'\|_\infty$. We find 
\[\int_{0}^{\tau^*} \sign\bigl( \rho_{1}(t,1^-)- \rho_{2}(t,1^-)\bigr) \left( f( \rho_{1}(t,1^-))-f( \rho_{2}(t,1^-))\right)\, {\rm{d}} t
\leq
\int_{x^*}^1 \left|\overline{\rho}_{1}(x)-\overline{\rho}_{2}(x)\right|\, {\rm{d}} x.\]
Finally, recalling the BLN interpretation \cite{MR542510} of the Dirichlet boundary condition (see also \cite{AndrSbihi}, where the boundary condition is interpreted in terms of monotone subgraphs of the graph of $f$), we point out that
\[\sign\bigl( \rho_{1}(t,1^-)- \rho_{2}(t,1^-)\bigr) \left(f( \rho_{1}(t,1^-))-f( \rho_{2}(t,1^-))\right)
= \left|f( \rho_{1}(t,1^-))-f( \rho_{2}(t,1^-))\right|.\]
To sum up, we find
\[\int_{0}^{\tau^*} \left|f( \rho_{1}(t,1^-))-f( \rho_{2}(t,1^-))\right|\, {\rm{d}} t \leq \left\|\overline{\rho}_{1}-\overline{\rho}_{2}\right\|_{\L1([x^*,1])}.\]
Further, the same inequality holds with $1^-$, $\tau^*$ and $[x^*,1]$ replaced by $-1^+$, $\tau_*$ and $[-1,x_*]$, respectively.

Recalling \eqref{eq:dotxi}, which we also write for $\dot{\xi}_{2}$ with $\alpha_{2}=(\alpha_{2}-\alpha_{1})+\alpha_{1}$, we finally deduce \eqref{eq:integral-dotxi-estim} under the precise form
\[\int_0^\tau|\dot{\xi}_{1}(s)-\dot{\xi}_{2}(s)|\, {\rm{d}} s \leq \tau\left\|f\right\|_\infty|\alpha_{1}-\alpha_{2}|+
\frac{\alpha_{1}}{2} \left\|\overline{\rho}_{1}-\overline{\rho}_{2}\right\|_{\L1(\mathfrak{C})}.\]

To conclude the proof, we now turn to the justification of \eqref{eq:dotxi}. Recall that by definition, $\xi_{1}$ is Lipschitz continuous, therefore its derivative $\dot{\xi}_{1}$ is defined a.e., and it is enough to establish
\begin{align}\label{eq:xidot-weak}
&2 \int_0^{\infty} \theta(t) \, \dot{\xi}_{1}(t)\, {\rm{d}} t 
= \alpha_1 \int_0^{\infty} \theta(t) \bigl(f\bigl( \rho_{1}(t,-1^+)\bigr) - f\bigl( \rho_{1}(t,1^-)\bigr)\bigr) {\rm{d}} t
&\forall \theta \in \Cc{ \infty}((0,\infty)).
\end{align}
As a starting point, let us multiply \eqref{e:cost0} by $\dot\theta$ and integrate in time. This leads to
\begin{align}\nonumber
0&= \int_0^{\infty} \dot\theta(t) \int_{\mathfrak{C}} \sign\left(x-\xi_{1}(t)\right) \, c_{1}\left( \rho_{1}(t,x)\right)\, {\rm{d}} x\, {\rm{d}} t
\\\nonumber
&= \int_0^{\infty} \int_{\mathfrak{C}} \dot\theta(t) \, \sign\left(x-\xi_{1}(t)\right)\, {\rm{d}} x \, {\rm{d}} t
+ \alpha_{1} \int_0^{\infty} \int_{\mathfrak{C}} \dot\theta(t) \, \sign\left(x-\xi_{1}(t)\right) \, \rho_{1}(t,x)\, {\rm{d}} x \, {\rm{d}} t\\\label{eq:I1+Irho}
&=: I_1+\alpha_{1} I_ \rho . 
\end{align}
We then consider
a sequence $\{\xi_{1}^m\}_m$ of $\mathbf{C}^ \infty$ approximations of $\xi_{1}$, which converge uniformly on $[0,\infty)$ while keeping uniformly bounded derivatives $\dot{\xi}_{1}^m$ converging pointwise to $\dot{\xi}_{1}$ (such approximations can be obtained by regularizing $\dot{\xi}_{1}$ by convolution). 
Without loss of generality, we can assume $\left\|\xi_{1}^m-\xi_{1}\right\|_\infty \leq 1/m$. 
Further, we approximate $\sign(\cdot)$ by functions $\eta_m(z):=\eta(mz)$, where $\eta \in\C\infty(\mathbb{R})$ is non-decreasing and $\eta(z)=\sign(z)$ for $|z|\geq 1$.
Note that $\eta_m(x-\xi_{1}^m(t))\to \sign(x-\xi_{1}(t))$ for all $(t,x)$ such that $x\neq \xi_{1}(t)$.

The dominated convergence readily yields $I_1=\lim_{m\to\infty} I_1^m$, where $I_1^m$ is defined by replacing $\sign(x-\xi_{1}(t))$ by $\eta_m(x-\xi_{1}^m(t))$ in the definition of $I_1$. Now we integrate by parts in $t$ in the expression of $I_1^m$. For sufficiently large $m$, using the fact that $\xi_{1}$ takes values in $[-1+\delta,1-\delta]$ with some $\delta>0$, as stated here above, we find
\begin{equation}\label{eq:I1}
I_1=\lim_{m\to \infty} I_1^m
=\lim_{m\to \infty} \int_0^{\infty} \theta(t) \, \dot{\xi}_{1}^m(t) \int_{\mathfrak{C}} \eta_m'(x-\xi_{1}^m(t))\, {\rm{d}} x\, {\rm{d}} t
=2 \int_0^{\infty} \theta(t) \, \dot{\xi}_{1}(t)\, {\rm{d}} t.
\end{equation}
In order to calculate $I_ \rho $, let us point out that the weak formulation contained in \eqref{e:entro} implies
\begin{equation}\label{eq:weak-formul}
\int_0^{ \infty} \int_{\mathfrak{C}} \bigl( \rho_{1}\, \varphi_t + \sign\left(x-\xi_{1}(t)\right) f( \rho_{1})\, \varphi_x \bigr) \, {\rm{d}} x \, {\rm{d}} t=0
\end{equation}
for all $\varphi \in \Cc \infty\left((0,\infty) \times \mathfrak{C}\right)$.
We then consider in \eqref{eq:weak-formul} 
test functions $\varphi_m(t,x)=\theta(t) \, \psi(x) \, \eta_m(x-\xi_{1}^m(t))$ with $\psi \in \Cc{\infty}(\mathfrak{C})$ and $\theta$, $\eta_m$, $\xi_{1}^m$ defined above. 
The choices we made for $\eta_m$ and $\xi_{1}^m$ ensure that the factor $|\eta'(m(x-\xi_{1}^m(t)))|$  has its support included in the set $\{(t,x) : |x-\xi_{1}(t)| \leq \frac 2m\}$; note that this factor is $\L \infty$ bounded uniformly in $m$.
It follows that, as $m\to \infty$, the integrals of
\[\theta(t) \, \psi(x) \, m \, \eta'\bigl(m\left(x-\xi_{1}^m(t)\right)\bigr) \, \dot{\xi}_{1}^m(t)\, \rho_{1}, \quad
\theta(t)\,\psi(x)\,m\,\eta'\bigl(m\left(x-\xi_{1}^m(t)\right)\bigr)\,f( \rho_{1})\]
vanish due to the assumption of zero traces of $\rho_{1}$ on the curve $x=\xi_{1}(t)$.
This is assessed via dominated convergence argument in transformed variables $(y,t)$, $y:=m(x-\xi_{1}(t))$, having in mind the above remark on the support of the $\eta'$ factor and the zero trace assumption meaning that $\rho_{1}(t,\xi_{1}(t)+\frac ym)\to 0$ pointwise, as $m\to \infty$.

As $\sign(x-\xi_{1}(t)) \, \eta\bigl(m(x-\xi_{1}(t))\bigr)\to 1$ a.e., with another application of the dominated convergence theorem we infer 
\begin{equation*}
\int_0^{\infty} \int_{\mathfrak{C}} \left( \dot\theta(t) \, \psi(x) \, \sign\left(x-\xi_{1}(t)\right) \, \rho_{1}\, + \theta(t)\psi'(x) f( \rho_{1})\, \right) \, {\rm{d}} x \, {\rm{d}} t=0.
\end{equation*}
Finally, to reach to $I_ \rho $ we let $\psi$ converge to the characteristic function of $\mathfrak{C}$; since strong boundary traces $\rho_{1}(\cdot,\pm1^\mp)$ exist (see \cite{Panov_traces2}), we get
\begin{equation}\label{eq:Irho}
I_ \rho = \int_0^{\infty} \theta(t) \left(f( \rho_{1}(t,1^-)) - f( \rho_{1}(t,-1^+))\right) \, {\rm{d}} t.
\end{equation}
Assembling \eqref{eq:I1} and \eqref{eq:Irho} within \eqref{eq:I1+Irho}, we reach to \eqref{eq:xidot-weak} and conclude the proof.\qed
\end{proof}

\section{A sharp model for many-particle Hughes dynamics}
\label{s:mpa}

In previous works \cite{DiFrancescoFagioliRosiniRusso,MR3644595,DiFrancescoFagioliRosiniRussoKRM} on the subject, many-particle approximations of Hughes' model were introduced and simulated. Here we propose  an improved many-particle Hughes model based upon a new definition of the approximate turning curve, that we denote by $\zeta^n$. This definition leads to a many-particle dynamics where the instants of particles' interactions with the approximate turning curve are sharply captured. This leads, in turn, to a rigorous construction of the unique global in time solution to the many-particle system.

Assume we are given $L>0$, $n\in \mathbb{N}$ and  $-1\leq \overline{x}_0<\overline{x}_1<\dots<\overline{x}_n\leq 1$ satisfying 
\[
\forall i \in \llbracket 0 , n-1 \rrbracket, \  \overline{x}_{i+1}-\overline{x}_{i} \geq \frac{\ell}{\rho_{\max}},
\]	
where we set $\ell:=L/n$.
We also set $R_{\max}:= \max_{i \in \llbracket 0 , n-1 \rrbracket} \frac{\ell}{\overline{x}_{i+1}-\overline{x}_{i}}\in (0,\rho_{\max}]$.

The time evolution in the whole of $\mathbb{R}$ of the particle system $x_{0}(t),\ldots,x_{n}(t)$ is described by the follow-the-leader system
\begin{equation}
\label{e:FTL}
\left\{\begin{array}{@{}l@{\qquad}l@{\quad}l@{}}
\dot{x}_{i}(t) = -v\left(R_{i-\frac{1}{2}}(t)\right) &\hbox{ if } x_{i}(t) < \zeta^{n}(t),& i \in \llbracket 0 , n \rrbracket,
\\
\dot{x}_{i}(t) = v\left(R_{i+\frac{1}{2}}(t)\right) &\hbox{ if } x_{i}(t) \geq \zeta^{n}(t),& i \in \llbracket 0 , n \rrbracket,
\\
x_{i}(0) = \overline{x}_{i}, && i \in \llbracket 0 , n \rrbracket.
\end{array}\right.
\end{equation}
Here and after
\begin{equation}
\label{e:Plini}
R_{i+\frac{1}{2}}(t) := 
\frac{\ell}{x_{i+1}(t)-x_{i}(t)},
\qquad i \in \llbracket -1 , n \rrbracket,
\end{equation}
where
\begin{align}
\label{e:IBuiltTheSky}
&x_{-1}(t):=- \infty,&
&x_{n+1}(t):= \infty.
\end{align}
Notice that by \eqref{e:Plini} and \eqref{e:IBuiltTheSky} we have $R_{-\frac{1}{2}} = 0$ and $R_{n+\frac{1}{2}} = 0$, therefore $v(R_{-\frac{1}{2}}) = v_{\max}$ and $v(R_{n+\frac{1}{2}}) = v_{\max}$.
The ODE system \eqref{e:FTL} with notations \eqref{e:Plini}, \eqref{e:IBuiltTheSky} needs to be closed by providing the dynamics of the approximate turning point $\zeta^{n}(t) \in \mathbb{R}$. The latter is implicitly uniquely determined by 
\begin{equation}
\label{e:turning}
Z_-\left(t,\zeta^{n}(t)\right) = Z_+\left(t,\zeta^{n}(t)\right),
\end{equation}
where $Z_\pm \colon [0,\infty) \times \mathbb{R} \to \mathbb{R}$ are defined by
\begin{align}
Z_-(t,x) &:= 
\begin{cases}\displaystyle x +1 + \alpha \int_{x_{I_-}(t)}^{x} \rho^{n}(t,y) \, {\rm{d}} y
&\begin{minipage}[t]{.4\linewidth}
if $\exists\, I_- \in \llbracket 0 , n \rrbracket$ such that\\$x_{I_--1}(t) \leq -1 < x_{I_-}(t) < x$,
\end{minipage}
\\[7pt] \displaystyle
x + 1 &\hbox{otherwise},
\end{cases}
\label{e:Xim}
\\
Z_+(t,x) &:= 
\begin{cases}\displaystyle 1 - x + \alpha \int_{x}^{x_{I_+}(t)} \rho^{n}(t,y) \, {\rm{d}} y
&\begin{minipage}[t]{.4\linewidth}
if $\exists\, I_+ \in \llbracket 0 , n \rrbracket$ such that\\$x < x_{I_+}(t) < 1 \leq x_{I_++1}(t)$,
\end{minipage}
\\[7pt] \displaystyle
1 - x &\hbox{otherwise},
\end{cases}
\label{e:Xip}
\end{align}
with $\rho^{n} \colon (0, \infty) \times \mathbb{R} \to [0, \rho_{\max}]$ being the approximate density
\begin{equation}
\label{e:disdens}
\rho^{n}(t,x) := \sum_{i = 0}^{n-1}R_{i+\frac{1}{2}}(t) \, \mathbbm{1}_{\left[x_{i}(t),x_{i+1}(t)\right)}(x).
\end{equation} 

Notice that by \eqref{e:Plini} and \eqref{e:disdens} we have
$ \int_{\mathbb{R}} \rho^{n}(t,y) \, {\rm{d}} y = L$, $t\geq0$.

\smallskip
We underline that $\zeta^{n}(t)$ is well defined by the monotonicity of $Z_\pm(t,\cdot\,)$.
We stress that in general it may happen that condition $x_{0}(t) < \zeta^{n}(t) \leq x_{n}(t)$ is not satisfied, as the following example proves.
\begin{example}
\label{ex:simple}
In the case $\alpha=0$, we have that $\zeta^{n}\equiv0$.
If $[\overline{x}_{\min} , \overline{x}_{\max}] \subset [-1,0) \cup (0,1]$, then condition $x_{0}(t) < \zeta^{n}(t) \leq x_{n}(t)$ is not satisfied for any time $t\geq0$.
\end{example}

\begin{remark}\label{rem:choice-zeta}
One can see that the definitions \eqref{e:Xim}, \eqref{e:Xip} take into account only those particles that are situated inside $\mathfrak{C}$ (those ranging from $x_{I_-}(t)$ to $x_{I_+}(t)$). Bearing in mind the idea of ``thinking particles'' behind the continuum Hughes' model, we find a natural interpretation of the above definition of the turning curve in the many-particle Hughes' model we propose here. Namely, the pedestrians that have reached the doors do not represent ``obstacles to evacuation'' any more, as if there were no particles at all in front of $x_{I_\pm}(t)$. Therefore the particles falling out of $\mathfrak{C}$ are disregarded in the evaluation, done by the remaining pedestrians, of the evacuation costs $Z_{\pm}$ by the respective exits located at $x=\pm 1$.
Beside this clear modeling assumption, the above definition of the turning curve $\zeta^n$ has numerous analytical consequences that we uncover in this section, making it advantageous in comparison to the straightforward alternative  definition $\xi^n$ of the turning curve, see \eqref{e:xiturning}, \eqref{e:Zmp} in Section~\ref{s:turn}.  	
\end{remark}

To sum up, the many-particle approximation consists in the ODE system \eqref{e:FTL}--\eqref{e:IBuiltTheSky} which features discontinuities in the state variable $(x_0,\dots,x_n)$ driven by the variable $\zeta^n$ implicitly determined by relations \eqref{e:turning}--\eqref{e:disdens}. 
By a solution to \eqref{e:FTL}--\eqref{e:disdens} we  mean an $(n+2)$-tuple $\bigl((x_0,\dots,x_n),\zeta^n\bigr)$ of functions
defined on $[0,\tau)$ (for some $\tau\in (0,\infty]$) and having the following regularity:
\begin{itemize}
\item [(i)] $x_i$, $i\in \llbracket 0 , n \rrbracket$, and $\zeta^n$ are piecewise $\C1$ on $[0,\tau)$. More precisely, there exists $H_{\rm sw}\in \mathbb{N}$ and times $\{t_h\}_{h\in \llbracket 1 , H_{\rm sw} \rrbracket}$, $t_1<t_2<\dots<t_{H_{\rm sw}}<\tau$, such that, upon setting $t_0:=0$ and $t_{H_{\rm sw}+1}:=\tau$, the restriction of each of these functions to the time intervals $(t_h,t_{h+1})$ can be extended to a  $\C1([t_h,t_{h+1}))$ function.
\label{pageC1}
\item[(ii)]  $x_i$, $i\in \llbracket 0 , n \rrbracket$, are continuous on $[0,\tau)$, while their derivatives $\dot x_i$ and the function $\zeta^n$ are normalized by the left-continuity at the times $t_h$, $i\in \llbracket 0 , H_{\rm sw} \rrbracket$.
\end{itemize} 
Note that, because of the piecewise regularity and of the discontinuity of the right-hand side of the system \eqref{e:FTL}, standard ODE tools do not readily yield existence of a solution, not even locally in time ($\tau<\infty$). In what follows, we will conduct an \emph{a priori} analysis of solutions of \eqref{e:FTL}--\eqref{e:disdens}. This analysis will culminate in an effective construction of a
unique global in time ($\tau=\infty$) solution to \eqref{e:FTL}--\eqref{e:disdens} in the above indicated sense, see Theorem~\ref{th:discrete-WP}. It will also uncover several properties of the many-particle dynamics. Some of them will be instrumental for the analysis of stability, consistency and convergence, as $n\to\infty$, of the many-particle scheme for problem \eqref{e:model}; moreover, they shed light on the agents' behavior within the many-particle variant of the Hughes' model.

\smallskip
Example~\ref{ex:simple} shows that the case $\alpha=0$ is trivial.
For this reason, below we assume that
\[\alpha>0.\]
In this case the functions $Z_\pm$ can be represented as in \figurename~\ref{f:xi}.
Notice that $Z_{-}$ is a piecewise linear strictly increasing map, and $Z_{+}$ is a piecewise linear strictly decreasing map.
Moreover, it results
\begin{align*}
&Z_-(t,-1) = 0,&
&Z_-(t,1) \in [2, 2+\alpha \, M(t)],&
&\partial_x Z_-(t,x) \geq 1,
\\
&Z_+(t,-1) \in [2, 2+\alpha \, M(t)],&
&Z_+(t,1) = 0,&
&\partial_x Z_+(t,x) \leq -1,
\end{align*}
where
\[M(t) := \int_{-1}^{1} \rho^{n}(t,y) \, {\rm{d}} y\]
is the total mass in $\mathfrak{C}$ at time $t\geq0$.
The above considerations imply that
\begin{align*}
&x+1 \leq Z_-(t,x) \leq x+1+\alpha\,M(t),&
&-x+1 \leq Z_+(t,x) \leq -x+1+\alpha\,M(t).
\end{align*}
Observe that by definition $M(0)=L$.

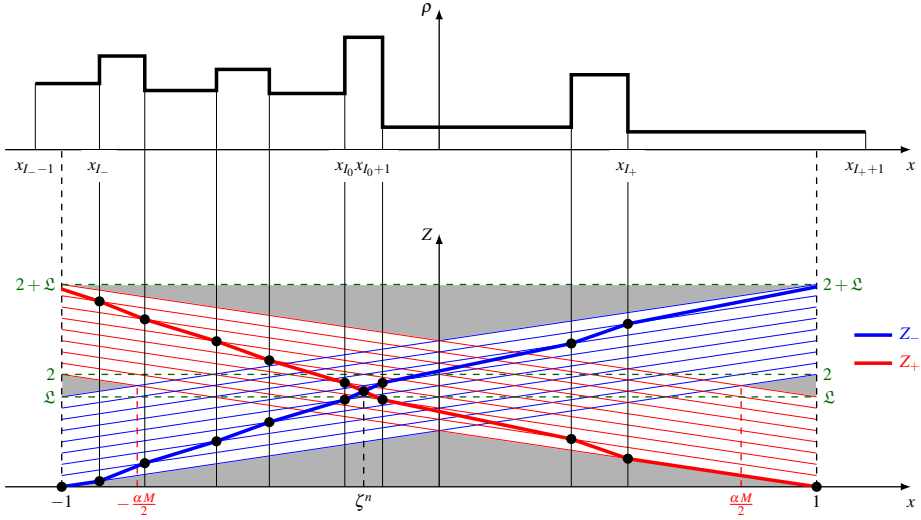
\begin{figure}[!htbp]
\resizebox{\linewidth}{!}{
\begin{tikzpicture}[x=60mm, y=9mm, semithick]

\def\lll{.2}
\def\a{-1.07}
\def\b{-.9}
\def\c{-.78}
\def\d{-.59}
\def\e{-.45}
\def\f{-.25}
\def\g{-.15}
\def\h{.35}
\def\i{.5}
\def\l{1.13}

\draw[white,fill=black!30!white] (-1,0) rectangle (1,3.6);
\draw[white,fill=white] (-1,1.6) -- (-1,0) -- (0,1) -- (1,0) -- (1,1.6) -- (.8,1.8) -- (1,2) -- (1,3.6) -- (0,2.6) -- (-1,3.6) -- (-1,2) -- (-.8,1.8) -- cycle;

\foreach \z in {0,1,...,8}{
\draw[blue,thin] (-1,{\z*\lll}) -- (1,{2+\z*\lll});
\draw[red,thin] (-1,{2+\z*\lll}) -- (1,{\z*\lll});
}

\draw[red,dashed] (-.8,0) node[below] {\strut$-\frac{\alpha \, M}{2}$} -- (-.8,1.8);
\draw[red,dashed] (.8,0) node[below] {\strut$\frac{\alpha \, M}{2}$} -- (.8,1.8);

\draw[blue,ultra thick] (1.1,2.7) -- ++(.1,0) node[right] {$Z_-$};
\draw[red,ultra thick] (1.1,2.2) -- ++(.1,0) node[right] {$Z_+$};

\draw[-latex] (0,0) -- (0,4.5)node[left] {$Z$};
\draw[-latex] (-1.15,0) -- (1.25,0) node[below] {\strut$x$};

\node[below] at (-1,0) {\strut $-1$};
\node[below] at (1,0) {\strut $1$};
\draw[dashed,green!40!black] (-1,3.6) node[left] {$2+\mathfrak{L}$} -- (1,3.6) node[right] {$2+\mathfrak{L}$};
\draw[dashed,green!40!black] (-1,2) node[left] {$2$} -- (1,2) node[right] {$2$};
\draw[dashed,green!40!black] (-1,1.6) node[left] {$\mathfrak{L}$} -- (1,1.6) node[right] {$\mathfrak{L}$};

\draw[dashed] (-1,0) -- (-1,6);
\draw[dashed] (1,0) -- (1,6);

\def\rab{(\lll/(\b-\a))}
\def\rbc{(\lll/(\c-\b))}
\def\rcd{(\lll/(\d-\c))}
\def\rde{(\lll/(\e-\d))}
\def\ref{(\lll/(\f-\e))}
\def\rfg{(\lll/(\g-\f))}
\def\rgh{(\lll/(\h-\g))}
\def\rhi{(\lll/(\i-\h))}
\def\ril{(\lll/(\l-\i))}

\draw[thin] (\b,0) -- (\b,{6+\rab});
\draw[thin] (\c,0) -- (\c,{6+\rbc});
\draw[thin] (\d,0) -- (\d,{6+\rcd});
\draw[thin] (\e,0) -- (\e,{6+\rde});
\draw[thin] (\f,0) -- (\f,{6+\ref});
\draw[thin] (\g,0) -- (\g,{6+\rfg});
\draw[thin] (\h,0) -- (\h,{6+\rgh});
\draw[thin] (\i,0) -- (\i,{6+\ril});

\begin{scope}[shift={(0,6)}]

\draw[thin] (\a,0) -- (\a,{\rab});
\draw[thin] (\l,0) -- (\l,{\ril});

\draw[-latex] (0,0) -- (0,2.5)node[left] {$\rho$};
\draw[-latex] (-1.15,0) -- (1.25,0) node[below] {\strut$x$};

\draw[ultra thick] (\a,{\rab}) -- (\b,{\rab}) -- (\b,{\rbc}) -- (\c,{\rbc}) -- (\c,{\rcd}) -- (\d,{\rcd}) -- (\d,{\rde}) -- (\e,{\rde}) -- (\e,{\ref}) -- (\f,{\ref}) -- (\f,{\rfg}) -- (\g,{\rfg}) -- (\g,{\rgh}) -- (\h,{\rgh}) -- (\h,{\rhi}) -- (\i,{\rhi}) -- (\i,{\ril}) -- (\l,{\ril});

\node[below] at (\a,0) {\strut $x_{I_--1}$};
\node[fill=white,below] at (\b,0) {\strut $x_{I_-}$};
\node[fill=white,below] at (\i,0) {\strut $x_{I_+}$};
\node[below] at (\l,0) {\strut $x_{I_++1}$};
\node[fill=white,below] at ({(\f+\g)/2},0) {\strut $x_{I_0}x_{I_0+1}$};

\end{scope}

\coordinate (BM) at (\b,{\b+1});
\coordinate (CM) at (\c,{\c+1+\lll});
\coordinate (DM) at (\d,{\d+1+2*\lll});
\coordinate (EM) at (\e,{\e+1+3*\lll});
\coordinate (FM) at (\f,{\f+1+4*\lll});
\coordinate (GM) at (\g,{\g+1+5*\lll});
\coordinate (HM) at (\h,{\h+1+6*\lll});
\coordinate (IM) at (\i,{\i+1+7*\lll});
\coordinate (LM) at (\l,{\l+1+8*\lll});
\coordinate (AP) at (\a,{-\a+1+8*\lll});
\coordinate (BP) at (\b,{-\b+1+7*\lll});
\coordinate (CP) at (\c,{-\c+1+6*\lll});
\coordinate (DP) at (\d,{-\d+1+5*\lll});
\coordinate (EP) at (\e,{-\e+1+4*\lll});
\coordinate (FP) at (\f,{-\f+1+3*\lll});
\coordinate (GP) at (\g,{-\g+1+2*\lll});
\coordinate (HP) at (\h,{-\h+1+\lll});
\coordinate (IP) at (\i,{-\i+1});

\begin{scope}
\clip(-1,0) rectangle (1,3.6);
\draw[blue,ultra thick] (-1,0) -- (BM) -- (CM) -- (DM) -- (EM) -- (FM) -- (GM) -- (HM) -- (IM) -- (LM);
\draw[red,ultra thick] (AP) -- (BP) -- (CP) -- (DP) -- (EP) -- (FP) -- (GP) -- (HP) -- (IP) -- (1,0);
\end{scope}

\draw[fill=black] (-1,0) circle (2pt);
\draw[fill=black] (BM) circle (2pt);
\draw[fill=black] (CM) circle (2pt);
\draw[fill=black] (DM) circle (2pt);
\draw[fill=black] (EM) circle (2pt);
\draw[fill=black] (FM) circle (2pt);
\draw[fill=black] (GM) circle (2pt);
\draw[fill=black] (HM) circle (2pt);
\draw[fill=black] (IM) circle (2pt);
\draw[fill=black] (BP) circle (2pt);
\draw[fill=black] (CP) circle (2pt);
\draw[fill=black] (DP) circle (2pt);
\draw[fill=black] (EP) circle (2pt);
\draw[fill=black] (FP) circle (2pt);
\draw[fill=black] (GP) circle (2pt);
\draw[fill=black] (HP) circle (2pt);
\draw[fill=black] (IP) circle (2pt);
\draw[fill=black] (1,0) circle (2pt);

\draw[dashed] (-.2,0) node[below] {\strut$\zeta^{n}$} -- (-.2,1.7);
\draw[fill=black] (-.2,1.7) circle (2pt);
\end{tikzpicture}}

\caption{Representations of $Z_\pm$ and $\zeta^{n}$ in the case $\mathfrak{L}:=\alpha\,(I_+-I_-+1) \, \ell < 2$ and $x_{I_--1} < -1 < x_{I_-} < \ldots < x_{I_+} < 1 < x_{I_++1}$. 
For convenience we omit the dependence on $t$.
Notice that, at least in the case under consideration, the value at $x_{i}$ of $Z_-$ (respectively, $Z_+$) corresponds to the intersection of the vertical line $x=x_{i}$ and the line $Z=x+1+\alpha(i-I_-) \, \ell$ (respectively, $Z=-x+1+\alpha(I_+-i) \, \ell$).
At last, the graphs of $Z_{\pm}$ on the whole of $\mathfrak{C}$ are then obtained by interpolating such points.}
\label{f:xi}
\end{figure}

For ease of notation, in the following we will drop the time and the $n$ dependencies whenever it is clear from the context.

\delaynewpage{2}
As a consequence of the next lemma we have that, once a particle leaves $\mathfrak{C}$, it cannot re-enter (it remains outside $\mathfrak{C}$). We will show that particle interactions with the turning curve (and the corresponding singularities of the piecewise $\C1$ solution to the system  \eqref{e:FTL}--\eqref{e:disdens}) occur only if exactly one particle leaves $\mathfrak{C}$; thus it becomes easy to count the interaction times.

\begin{lemma}
For any $t\in [0,\tau)$ we have that
\[\zeta^{n}(t) \in \mathfrak{C} \cap \left[-\frac{\alpha}{2}\,M(t),\frac{\alpha}{2}\,M(t)\right] \subseteq \mathfrak{C} \cap \left[-\frac{\alpha}{2}\,L,\frac{\alpha}{2}\,L\right].\]
\label{l:bounxi}
\end{lemma}

\begin{proof}
\begin{enumerate}[label={\bf Step~\Roman*},wide=0pt]

\item 
Assume by contradiction that $\zeta^n(t) \leq -1$.
Then by condition \eqref{e:turning} and definitions \eqref{e:Xim}, \eqref{e:Xip} we have
$0 \geq Z_-(t,\zeta^n(t)) = Z_+(t,\zeta^n(t)) \geq 2,$
but this gives a contradiction.
The case $\zeta^n(t)\geq1$ is analogous and is therefore omitted.

\item
We claim that:

\begin{itemize}

\item 
If $\zeta^n(t)<0$, then in the interval $\left(-1,\zeta^n(t)\right)$ there is at least one particle.

\item\label{Step2}
If $\zeta^n(t)>0$, then in the interval $(\zeta^n(t),1)$ there is at least one particle.
\end{itemize}
We prove the first claim; the second follows then from the symmetry of the model.
If by contradiction no particle is in $\mathfrak{C}$ at time $t$, then by \eqref{e:turning} we have $\zeta^n(t)=0$ and this contradicts the hypothesis $\zeta^n(t)<0$.
Furthermore, if by contradiction there exists $I_- \in \llbracket0,n\rrbracket$ such that
\[x_{I_--1}(t) \leq -1 < \zeta^n(t) \leq x_{I_-}(t) \leq x_{I_+}(t) < 1 \leq x_{I_++1}(t),\]
then by \eqref{e:turning} we have
\[
0 > 2\zeta^n(t) = \alpha \int_{\zeta^n(t)}^{x_{I_+}(t)} \rho(t,y) \, {\rm{d}} y \geq 0
\]
and this gives a contradiction.

\item
We prove now that
\[\zeta^n(t) \in \left[-\frac{\alpha}{2}\,M(t),\frac{\alpha}{2}\,M(t)\right].\]
If $\zeta^n(t)>0$, then by \ref{Step2}, \eqref{e:Xim}, \eqref{e:Xip} and \eqref{e:turning} we have
\[
\zeta^n(t) \leq \frac{\alpha}{2} \int_{\zeta^n(t)}^{x_{I_+}(t)} \rho(t,y) \, {\rm{d}} y \leq \frac{\alpha}{2}\,M(t).
\]
The case $\zeta^n(t)<0$ is analogous and the case $\zeta^n(t)=0$ is trivial.

\item
At last, we conclude the proof by observing that $M(t) \leq L$.\qed
\end{enumerate}
\end{proof} 

Next lemma highlights how the parameter $\alpha>0$ impacts on the approximate turning point $\zeta^{n}$.
\begin{lemma}
Fix $t\in [0,\tau)$.
Assume that in each of the intervals $(-1,\zeta^{n}(t)]$ and $[\zeta^{n}(t),1)$ there is at least one particle, namely, that there exist $I_-,I_+ \in \llbracket 0 , n \rrbracket$ such that
\[x_{I_--1}(t) \leq -1 <x_{I_-}(t) \leq \zeta^{n}(t) \leq x_{I_+}(t) < 1 \leq x_{I_++1}(t).\]
Then $\zeta^{n}$ belongs to the closed interval between $0$ and $\frac{1}{2}(x_{I_{\rm m}^-}+x_{I_{\rm m}^+})$, where
\[I_{\rm m}^-:=\left\lfloor\frac{I_++I_-}{2}\right\rfloor,
\qquad
I_{\rm m}^+:=\left\lceil\frac{I_++I_-}{2}\right\rceil.\]
More precisely, $\zeta^{n}$ is closer to $0$ for lower values of $\alpha$, whereas it is closer to $\frac{1}{2}(x_{I_{\rm m}^-}+x_{I_{\rm m}^+})$ for higher values of $\alpha$.
\end{lemma} 

\begin{proof} 
By \eqref{e:turning} we have
\begin{equation}
\label{e:Korn}
\zeta^n(t)+\alpha \int_{x_{I_-}(t)}^{\zeta^n(t)} \rho^{n}(t,y) \, {\rm{d}} y = \frac{\alpha}{2} \int_{x_{I_-}(t)}^{x_{I_+}(t)} \rho^{n}(t,y) \, {\rm{d}} y = \frac{\alpha\ell}{2}(I_+-I_-).
\end{equation}
By taking $\alpha=0$, from \eqref{e:Korn} we deduce that $\zeta^n\equiv0$.
By letting $\alpha$ go to infinity in \eqref{e:Korn}, we obtain
\[ \int_{x_{I_-}(t)}^{\zeta^n(t)} \rho^{n}(t,y) \, {\rm{d}} y = \int_{\zeta^n(t)}^{x_{I_+}(t)} \rho^{n}(t,y) \, {\rm{d}} y = \frac{1}{2} \int_{x_{I_-}(t)}^{x_{I_+}(t)} \rho^{n}(t,y) \, {\rm{d}} y,\]
or equivalently $\zeta^n\equiv\frac{1}{2}(x_{I_{\rm m}^-}+x_{I_{\rm m}^+})$.
It is now clear that if $\alpha \in(0, \infty)$, then $\zeta^n$ belongs to the interval between $0$ and $\frac{1}{2}(x_{I_{\rm m}^-}+x_{I_{\rm m}^+})$.\qed
\end{proof} 

Let $s_0:=0$ and denote by $s_{h}\in (0,\tau)$ the $h$-th positive time of exit of at least one particle from $\mathfrak{C}$.
Notice that at time $s_0=0$ at least one particle leaves $\mathfrak{C}$ if and only if $\overline{x}_{\min}=-1$ or $\overline{x}_{\max}=1$. 
By Lemma~\ref{l:bounxi} there are at most $n+2$ of such times $s_{h}$.
Denote by $s_{H_{\rm ex}}$ the last exit time and let $s_{H_{\rm ex}+1}:=\tau$.

\delaynewpage{2}
In the sequel, we say that the $i$-th particle changes direction at time $t>0$ if  $\dot x_i(t^-)$ and $\dot x_i(t^+)$ are of opposite sign.
\begin{lemma}\label{lem:switching}
No particle changes direction during each time interval $(s_{h},s_{h+1})$.
\end{lemma}

\begin{proof}
Let $S_{h} := (s_{h},s_{h+1})$. Let $I_-^h, I_0^h, I_+^h : S_{h}\to \llbracket 0 , n \rrbracket$ be defined by the property 
\begin{equation*}\label{eq:ordering-general}
x_{I_-^h-1} \leq -1 < x_{I_-^h} \leq  x_{I_0^h} < \zeta^{n} \leq x_{I_0^h+1} \!\leq x_{I_+^h} < 1 \leq x_{I_+^h+1} \; \text{in $S_{h}$},
\end{equation*}
(this is the typical case, we omit the simpler cases where there is no particle on one side from the turning curve, in order to not overload the proof). By definition no particle leaves $\mathfrak{C}$ during the time interval $S_{h}$, therefore each of the functions $I_\pm^h$ keeps a constant value in $S_{h}$. Our goal is to prove that no particle changes direction during the time inteval $S_{h}$, which means according to \eqref{e:FTL} that also $I_0^h$ keeps a constant value in $S_{h}$.

First, we prove that $\zeta^n$ is continuous on $S_{h}$.
From the definition \eqref{e:turning}--\eqref{e:disdens} of $\zeta^n$ and the invariance of $I_\pm^h$,  with calculations analogous to those of the proof of \eqref{eq:xi(0)-estim} in Proposition~\ref{prop:xidot}, we find that for $s,t\in S_{h}$ there holds
\begin{equation}\label{eq:zetan-cont}
2 \, |\zeta^n(t)-\zeta^n(s)| \leq \alpha \int_{I_-^h}^{I_+^h} |\rho^n(t,x) - \rho^n(s,x)|\, {\rm{d}} x.
\end{equation}
Recall that $x_i$,  $i\in \llbracket 0 , n \rrbracket$, are continuous on $(0,\tau)$ and $x_{i+1}-x_i\neq 0$ according to the notion of solution we have fixed for \eqref{e:FTL}--\eqref{e:IBuiltTheSky}. Therefore $R_{i+\frac 12}$, $i\in \llbracket -1 , n \rrbracket$, are continuous as well, hence $\rho^n$ in \eqref{e:disdens} belongs to $\C0([0,\tau);\L1(\mathbb{R}))$. Now \eqref{eq:zetan-cont} implies the continuity of $\zeta^n$.

Now, consider any point $t$ in $S_{h}$ which is a point of differentiability of $x_0,\dots,x_n$ and of $\zeta^n$; moreover, assume that at time $t$ the turning curve lies strictly between two particles, which reads 
\begin{equation}\label{eq:ordering}
x_{I}(t) \,<\, \zeta^{n}(t) < x_{I+1}(t)
\end{equation}	 
where we denoted $I:=I_0^h(t)$; the value $I$ is fixed until the end of the paragraph. 
By the continuity of $x_{I}$, $\zeta^{n}$ and $x_{I+1}$, the ordering \eqref{eq:ordering} of their values is preserved in some interval $(t-\delta,t+\delta)\subset S_{h}$. In this situation, arguing as in \cite[Lemma~2.1]{DiFrancescoFagioliRosiniRussoKRM} we compute
\begin{align}
\dot{\zeta}^{n}(t) &=
\frac{\alpha R_{I+\frac{1}{2}}(t)}{c\bigl(R_{I+\frac{1}{2}}(t)\bigr)} \cdot \frac{\dot{x}_{I}(t)\bigl(x_{I+1}(t)-\zeta^{n}(t)\bigr)+\dot{x}_{I+1}(t)\bigl(\zeta^{n}(t)-x_{I}(t)\bigr)}{x_{I+1}(t)-x_{I}(t)}.\label{eq:dotzetan}
\end{align}
By \eqref{e:cost}, the first factor on the right-hand side of the above equation belongs to the interval $(0,1)$; by \eqref{eq:ordering}, the second one belongs to the interval $[\dot{x}_{I}(t),\dot{x}_{I+1}(t)]$. 
We note in passing that the dynamics \eqref{e:FTL} implies, by a straightforward induction argument, the infinite speed of propagation, that is, for all $i\in \llbracket 0 , n \rrbracket$, $\dot x_i\neq 0$ on $(0,\tau)$ (see \cite{AndrRosini-SOTA} and Remark~\ref{rem:finite-time-evac} below).
Therefore we actually have $\dot{\zeta}^{n}(t) \in (\dot{x}_{I}(t),\dot{x}_{I+1}(t))$.
This implies that the ordering \eqref{eq:ordering} extends to all times in some right neighborhood $[t,t+\delta)$ of $t$, moreover, in this neighbourhood the gaps $\zeta^n-x_{I}$, $x_{I+1}-\zeta^n$ strictly increase with time.
Whence it is straightforward to deduce that the ordering \eqref{eq:ordering} extends to all times in $[t,s_{h+1})$ and moreover, the value $I=I_0^h(t)$ fits the definition of $I_0^h(s)$ for all $s\in [t,s_{h+1})$.

Thus, we have achieved the property that $I_0^h$ keeps a constant value on every interval of the form $[t,s_{h+1})$ with the above choice of $t\in S_{h}$.
To conclude the proof, it remains to observe that $t$ can be taken arbitrarily close to $s_h$. Indeed, otherwise we have for some $I$ that $\zeta^{n} \equiv x_{I+1}$ in some right vicinity $(s_h,s_h+\delta)$ of $s_h$. Instead of the strict ordering \eqref{eq:ordering}, we have  $x_{I}(t) \,<\, \zeta^{n}(t) = x_{I+1}(t)$ for $t\in (s_h,s_h+\delta)$. Keeping in mind the definition of the scheme, we see that also in this case the expression \eqref{eq:dotzetan} holds at any time $t\in (s_h,s_h+\delta)$ of differentiability of the functions involved into the calculation. This leads to a contradiction because the gap $ x_{I+1}-\zeta^n$ is found to be strictly increasing. This ends the proof.
\qed
\end{proof}

\begin{remark}
Note that the above proof ensures in passing the following properties of solutions of \eqref{e:FTL}--\eqref{e:disdens} in the sense indicated on page~\pageref{pageC1}:
\begin{itemize}
\item the equality $\zeta^{n}(t) = x_{I_0^h+1}(t)$ can hold only at isolated points $t=s_h$, elsewhere the strict ordering \eqref{eq:ordering} holds;
\item \eqref{e:FTL} holds at such points due to the normalization of $\zeta^n$ by the right continuity and by the piecewise $\C1$ regularity of $x_i$, $i\in \llbracket 0 , n \rrbracket$.
\end{itemize}
\end{remark}

As a consequence of Lemma~\ref{lem:switching} and the above remark, there exists a finite number of times where a particle switches its direction; only at these times, the solution of \eqref{e:FTL}--\eqref{e:disdens} can feature discontinuities of $\zeta^n$ and non-differentiability of some of the particle trajectories $x_i$, $i \in\llbracket 0 , n \rrbracket$. Let $t_0:=0$ and denote by $t_{h}>0$ the $h$-th positive time of switching direction of at least one particle.
Denote by $t_{H_{\rm sw}}$ the last time a particle changes direction and let $t_{H_{\rm sw}+1}:=\infty$.
Lemma~\ref{lem:switching} tells us that $\{t_0,t_1,\dots,t_{H_{\rm sw}}\}\subseteq \{s_0,s_1,\dots,s_{H_{\rm ex}}\}$.

\smallskip
Our next goal is to demonstrate that a locally defined solution would not cease to exist because the values $R_{i+\frac{1}{2}}$ in \eqref{e:Plini} fall out of the interval $[0,\rho_{\max}]$ on which $v$ is defined.

\begin{lemma}[Discrete maximum principle]
\label{l:maxpri}
Assume \eqref{V1} and \eqref{C}. 
For any $i \in\llbracket 0 , n-1 \rrbracket$, it holds that
\begin{equation}
\label{19}
\frac{\ell}{R_{\max}} \leq x_{i+1}(t)-x_{i}(t) \leq 2(v_{\max} \, t+1),
\qquad t\in [0,\tau).
\end{equation}
\end{lemma}
\begin{proof} 
The upper bound in \eqref{19} follows from the estimates $-1\leq x_{\min} < x_{\max} \leq 1$ and $|\dot{x}_{i}(t)| \leq v_{\max}$.
We prove now the lower bound in \eqref{19} that we denote by $\eqref{19}_1$.
Estimate $\eqref{19}_1$ holds at time $t=0$ because by \eqref{I} and \eqref{eq:Periphery} we have
\begin{equation}
\ell = \int_{\overline{x}_{i}}^{\overline{x}_{i+1}}\overline{\rho}(x) \, {\rm{d}} x \leq (\overline{x}_{i+1}-\overline{x}_{i}) \, R_{\max},
\qquad i \in\llbracket 0 , n-1 \rrbracket.
\label{e:Marillion}
\end{equation}
Now we prove that if $\eqref{19}_1$ holds at time $t=t_h$ then it also holds for all $t\in (t_h,t_{h+1}]$ (on $(t_h,t_{h+1})$ if $t_{h+1}=\tau$); clearly, the claim of the lemma follows by induction. By definition, no particle changes direction in $(t_h,t_{h+1})$, so that $x_{I_0^h}< \zeta^n\leq x_{I_0^h+1}$ on $[t_h,t_{h+1})$ for some fixed value $I^h_0$. 
We can  apply \cite[Lemma~1]{DiFrancescoRosini} separately to the particles in $x_i$, $i<I_0^h$, and to those with $i\geq I_0^h+1$. Observe in addition that for $t\in [t_h,t_{h+1})$ the distance $x_{I_0^h+1}-x_{I_0^h}$ doesn't decrease because according to \eqref{e:FTL} their movement is repulsive with respect to the turning curve. By the inductive assumption,  it follows that $\eqref{19}_1$ holds for  $t\in (t_h,t_{h+1})$. Then by the continuity of $x_i$ it also holds at $t=t_{h+1}$. This concludes the proof.
\qed
\end{proof} 

Notice that at time $s_{h}$ at most one particle crosses $x=-1$ or $x=1$, moreover,  the above lemma ensures that for all $h\in \llbracket 0 , H_{\rm ex} \rrbracket$ one has $s_{h+1}-s_h\geq \frac{\ell}{R_{\max}v_{\max}}$.

In the following proposition we collect some basic properties of the many-particle Hughes model, they describe more precisely particles exiting or switching direction and the singularities of $\zeta^n$. Their proof comes directly from some geometrical considerations on the graphs of $Z_{-}$ and $Z_{+}$.

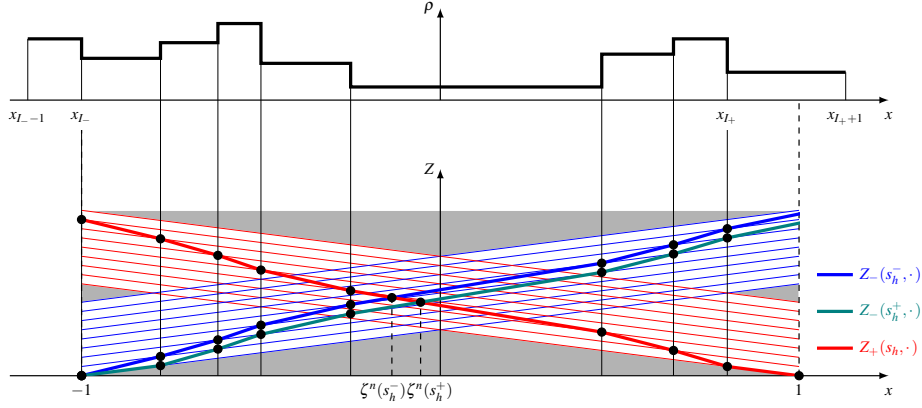
\begin{figure}[!htbp]
\resizebox{\linewidth}{!}{
\begin{tikzpicture}[x=62mm, y=8mm, semithick]

\def\lll{.2}
\def\a{-1.15}
\def\b{-1}
\def\c{-.78}
\def\d{-.62}
\def\e{-.5}
\def\f{-.25}
\def\g{.45}
\def\h{.65}
\def\i{.8}
\def\l{1.13}

\draw[white,fill=black!30!white] (-1,0) rectangle (1,3.6);
\draw[white,fill=white] (-1,1.6) -- (-1,0) -- (0,1) -- (1,0) -- (1,1.6) -- (.8,1.8) -- (1,2) -- (1,3.6) -- (0,2.6) -- (-1,3.6) -- (-1,2) -- (-.8,1.8) -- cycle;

\foreach \z in {0,1,...,8}{
\draw[blue,thin] (-1,{\z*\lll}) -- (1,{2+\z*\lll});
\draw[red,thin] (-1,{2+\z*\lll}) -- (1,{\z*\lll});
}

\draw[blue,ultra thick] (1.05,2.2) -- ++(.1,0) node[right] {$Z_-(s_{h}^-,\cdot\,)$};
\draw[teal,ultra thick] (1.05,1.4) -- ++(.1,0) node[right] {$Z_-(s_{h}^+,\cdot\,)$};
\draw[red,ultra thick] (1.05,.6) -- ++(.1,0) node[right] {$Z_+(s_{h},\cdot\,)$};

\draw[-latex] (0,0) -- (0,4.5)node[left] {$Z$};
\draw[-latex] (-1.2,0) -- (1.25,0) node[below] {\strut$x$};

\node[below] at (-1,0) {\strut $-1$};
\node[below] at (1,0) {\strut $1$};

\draw[dashed] (-1,0) -- (-1,6);
\draw[dashed] (1,0) -- (1,6);

\def\rab{(\lll/(\b-\a))}
\def\rbc{(\lll/(\c-\b))}
\def\rcd{(\lll/(\d-\c))}
\def\rde{(\lll/(\e-\d))}
\def\ref{(\lll/(\f-\e))}
\def\rfg{(\lll/(\g-\f))}
\def\rgh{(\lll/(\h-\g))}
\def\rhi{(\lll/(\i-\h))}
\def\ril{(\lll/(\l-\i))}

\draw[thin] (\b,0) -- (\b,{6+\rab});
\draw[thin] (\c,0) -- (\c,{6+\rbc});
\draw[thin] (\d,0) -- (\d,{6+\rcd});
\draw[thin] (\e,0) -- (\e,{6+\rde});
\draw[thin] (\f,0) -- (\f,{6+\ref});
\draw[thin] (\g,0) -- (\g,{6+\rfg});
\draw[thin] (\h,0) -- (\h,{6+\rgh});
\draw[thin] (\i,0) -- (\i,{6+\ril});

\begin{scope}[shift={(0,6)}]

\draw[thin] (\a,0) -- (\a,{\rab});
\draw[thin] (\l,0) -- (\l,{\ril});

\draw[ultra thick] (\a,{\rab}) -- (\b,{\rab}) -- (\b,{\rbc}) -- (\c,{\rbc}) -- (\c,{\rcd}) -- (\d,{\rcd}) -- (\d,{\rde}) -- (\e,{\rde}) -- (\e,{\ref}) -- (\f,{\ref}) -- (\f,{\rfg}) -- (\g,{\rfg}) -- (\g,{\rgh}) -- (\h,{\rgh}) -- (\h,{\rhi}) -- (\i,{\rhi}) -- (\i,{\ril}) -- (\l,{\ril});

\node[below] at (\a,0) {\strut $x_{I_--1}$};
\node[fill=white,below] at (\b,0) {\strut $x_{I_-}$};
\node[fill=white,below] at (\i,0) {\strut $x_{I_+}$};
\node[below] at (\l,0) {\strut $x_{I_++1}$};

\draw[-latex] (0,0) -- (0,2)node[left] {$\rho$};
\draw[-latex] (-1.2,0) -- (1.25,0) node[below] {\strut$x$};

\end{scope}

\coordinate (BM) at (\b,{\b+1});
\coordinate (CM) at (\c,{\c+1+\lll});
\coordinate (DM) at (\d,{\d+1+2*\lll});
\coordinate (EM) at (\e,{\e+1+3*\lll});
\coordinate (FM) at (\f,{\f+1+4*\lll});
\coordinate (GM) at (\g,{\g+1+5*\lll});
\coordinate (HM) at (\h,{\h+1+6*\lll});
\coordinate (IM) at (\i,{\i+1+7*\lll});
\coordinate (LM) at (\l,{\l+1+8*\lll});
\coordinate (BMp) at (\b,{\b+1});
\coordinate (CMp) at (\c,{\c+1});
\coordinate (DMp) at (\d,{\d+1+\lll});
\coordinate (EMp) at (\e,{\e+1+2*\lll});
\coordinate (FMp) at (\f,{\f+1+3*\lll});
\coordinate (GMp) at (\g,{\g+1+4*\lll});
\coordinate (HMp) at (\h,{\h+1+5*\lll});
\coordinate (IMp) at (\i,{\i+1+6*\lll});
\coordinate (LMp) at (\l,{\l+1+7*\lll});
\coordinate (AP) at (\a,{-\a+1+8*\lll});
\coordinate (BP) at (\b,{-\b+1+7*\lll});
\coordinate (CP) at (\c,{-\c+1+6*\lll});
\coordinate (DP) at (\d,{-\d+1+5*\lll});
\coordinate (EP) at (\e,{-\e+1+4*\lll});
\coordinate (FP) at (\f,{-\f+1+3*\lll});
\coordinate (GP) at (\g,{-\g+1+2*\lll});
\coordinate (HP) at (\h,{-\h+1+\lll});
\coordinate (IP) at (\i,{-\i+1});

\begin{scope}
\clip(-1,0) rectangle (1,3.6);
\draw[blue,ultra thick] (-1,0) -- (BM) -- (CM) -- (DM) -- (EM) -- (FM) -- (GM) -- (HM) -- (IM) -- (LM);
\draw[red,ultra thick] (AP) -- (BP) -- (CP) -- (DP) -- (EP) -- (FP) -- (GP) -- (HP) -- (IP) -- (1,0);
\draw[teal,ultra thick] (-1,0) -- (BMp) -- (CMp) -- (DMp) -- (EMp) -- (FMp) -- (GMp) -- (HMp) -- (IMp) -- (LMp);
\end{scope}

\draw[fill=black] (-1,0) circle (2pt);
\draw[fill=black] (BM) circle (2pt);
\draw[fill=black] (CM) circle (2pt);
\draw[fill=black] (DM) circle (2pt);
\draw[fill=black] (EM) circle (2pt);
\draw[fill=black] (FM) circle (2pt);
\draw[fill=black] (GM) circle (2pt);
\draw[fill=black] (HM) circle (2pt);
\draw[fill=black] (IM) circle (2pt);
\draw[fill=black] (BMp) circle (2pt);
\draw[fill=black] (CMp) circle (2pt);
\draw[fill=black] (DMp) circle (2pt);
\draw[fill=black] (EMp) circle (2pt);
\draw[fill=black] (FMp) circle (2pt);
\draw[fill=black] (GMp) circle (2pt);
\draw[fill=black] (HMp) circle (2pt);
\draw[fill=black] (IMp) circle (2pt);
\draw[fill=black] (BP) circle (2pt);
\draw[fill=black] (CP) circle (2pt);
\draw[fill=black] (DP) circle (2pt);
\draw[fill=black] (EP) circle (2pt);
\draw[fill=black] (FP) circle (2pt);
\draw[fill=black] (GP) circle (2pt);
\draw[fill=black] (HP) circle (2pt);
\draw[fill=black] (IP) circle (2pt);
\draw[fill=black] (1,0) circle (2pt);

\draw[fill=black] (-.135,1.7) circle (2pt);
\draw[fill=black] (-.055,1.6) circle (2pt);
\draw[dashed] (-.135,0) -- (-.135,1.7);
\draw[dashed] (-.055,0) -- (-.055,1.6);
\node[below] at (-0.095,0) {\strut$\zeta^{n}(s_{h}^-)\zeta^{n}(s_{h}^+)$};
\end{tikzpicture}}

\caption{Representations of $Z_\pm(s_{h}^\pm,\cdot\,)$ and $\zeta^{n}(s_{h}^\pm)$ in the case the particle $x_{I_-}$ crosses $x=-1$ at time $s_{h}$ and no particle changes direction.}
\label{f:xi2}
\end{figure}

\begin{proposition}
\label{p:Opeth}
We have the following:
\begin{enumerate}[label= {\bf(\arabic*)},leftmargin=*]
\item 
The approximate turning curve has a discontinuity jump if and only if a single particle leaves $\mathfrak{C}$.
More precisely, if a single particle leaves $\mathfrak{C}$ crossing $x=-1$ (respectively, $x=1$), then the approximate turning curve has a positive (respectively, negative) discontinuity jump.
\item 
At any time a single particle leaves $\mathfrak{C}$, then at most one particle changes
direction.
More precisely, if a single particle leaves $\mathfrak{C}$ crossing $x=-1$ (respectively, $x=1$) and a particle changes direction, then it changes from positive to negative (respectively, from negative to positive). Moreover, no
particle changes direction whenever two particles leave $\mathfrak{C}$ at the same time.
\end{enumerate}
\end{proposition}
\begin{proof} 
Consider an exit time $s_{h}$.
Assume that a single particle leaves $\mathfrak{C}$ crossing $x=-1$ at time $s_{h}$.
By \eqref{e:Xim} and \eqref{e:Xip}, see \figurename~\ref{f:xi2}, we have
\begin{equation}
\label{e:jump}
\begin{aligned}
Z_-(s_{h}^+,x) &= \max\{Z_-(s_{h}^-,x)-\alpha \, \ell,x+1\} < Z_-(s_{h}^-,x),
\\
Z_+(s_{h}^+,x) &= Z_+(s_{h}^-,x).
\end{aligned}
\end{equation}
As a consequence, the approximate turning curve has an increasing discontinuity.
In general, this discontinuity of $\zeta^n$ does not imply that a particle changes direction, see again \figurename~\ref{f:xi2}.
Indeed, this occurs if and only if a particle belongs to the interval $\left[\zeta^n(s_{h}^-),\zeta^n(s_{h}^+)\right)$, see \figurename~\ref{f:xi3}.
We prove that at time $s_{h}$ at most one particle changes direction by showing that the interval $\left[\zeta^n(s_{h}^-),\zeta^n(s_{h}^+)\right)$ contains at most one particle.
Indeed, in the case there is at least one particle in both $(-1,\zeta^n(s_{h}^\pm))$ and $[\zeta^n(s_{h}^\pm),1)$, by \eqref{e:turning} and \eqref{e:jump} we have
\begin{align*}
&Z_-(s_{h}^-,\zeta^n(s_{h}^-)) - Z_-(s_{h}^+,\zeta^n(s_{h}^+)) = Z_+(s_{h}^-,\zeta^n(s_{h}^-)) - Z_+(s_{h}^+,\zeta^n(s_{h}^+))
\\\Longleftrightarrow\ &
0 < \zeta^n(s_{h}^+) - \zeta^n(s_{h}^-) = \frac{\alpha}{2} \left(\ell - 2 \int_{\zeta^n(s_{h}^-)}^{\zeta^n(s_{h}^+)} \rho^{n}(t,y) \, {\rm{d}} y\right),
\end{align*}
and therefore
\[\int_{\zeta^n(s_{h}^-)}^{\zeta^n(s_{h}^+)} \rho^{n}(t,y) \, {\rm{d}} y < \frac{\ell}{2}.\]
The remaining cases can be treated analogously and are therefore omitted.

The case of a single particle that leaves $\mathfrak{C}$ crossing $x=1$ at time $s_{h}$ is analogous and is therefore omitted.
On the other hand, if two particles leave $\mathfrak{C}$ at the same time $s_{h}$, then by Lemma~\ref{l:maxpri} they don't cross the same exit, and therefore
\begin{align*}
Z_-(s_{h}^+,x) &= \max\{Z_-(s_{h}^-,x)-\alpha \, \ell,x+1\} < Z_-(s_{h}^-,x),
\\
Z_+(s_{h}^+,x) &= \max\{Z_+(s_{h}^-,x)-\alpha \, \ell,1-x\} < Z_+(s_{h}^-,x).
\end{align*}
As a consequence the approximate turning curve is continuous across time $s_{h}$.
\qed
\end{proof}

\begin{figure}[!htbp]
\resizebox{!}{!}{
\begin{tikzpicture}[x=65mm, y=5mm, semithick]

\def\lll{1}
\def\a{-.6}
\def\b{-.2}
\def\c{.4}

\begin{scope}
\clip(\a,0) rectangle (\c,4.9);
\foreach \z in {1,2,3,4}{
\draw[black,thin] (-1,{\z*\lll}) -- (1,{2+\z*\lll});
\draw[red,thin] (-1,{2+\z*\lll}) -- (1,{\z*\lll});
}
\draw[blue,ultra thick] (\a,{\a+1+2*\lll}) -- (\b,{\b+1+3*\lll}) -- (\c,{\c+1+4*\lll});
\draw[teal,ultra thick] (\a,{\a+1+\lll}) -- (\b,{\b+1+2*\lll}) -- (\c,{\c+1+3*\lll});
\draw[red,ultra thick] (\a,{-\a+1+3*\lll}) -- (\b,{-\b+1+2*\lll}) -- (\c,{-\c+1+1*\lll});
\end{scope}

\draw[blue,ultra thick] (.5,5) -- ++(.1,0) node[right] {$Z_-(s_{h}^-,\cdot\,)$};
\draw[teal,ultra thick] (.5,4.2) -- ++(.1,0) node[right] {$Z_-(s_{h}^+,\cdot\,)$};
\draw[red,ultra thick] (.5,3.4) -- ++(.1,0) node[right] {$Z_+(s_{h},\cdot\,)$};

\draw[thin] (\a,0) node[below] {\strut$x_{I_0-1}$} -- (\a,5) node[above] {\strut$x_{I_0-1}$};
\draw[thin] (\b,0) -- (\b,5) node[above] {\strut$x_{I_0}$};
\draw[thin] (\c,0) node[below] {\strut$x_{I_0+1}$} -- (\c,5) node[above] {\strut$x_{I_0+1}$};

\draw[fill=black] (-.285,3.5) circle (2pt);
\draw[dashed] (-.285,0) node[below] {\strut$\zeta^{n}(s_{h}^-)$} -- (-.285,3.5);
\draw[fill=black] (-.125,3) circle (2pt);
\draw[dashed] (-.125,0) node[below] {\strut$\zeta^{n}(s_{h}^+)$} -- (-.125,3);

\draw[fill=black] (-.5,3.5) circle (2pt);
\draw[dashed] (-.5,0) node[below] {\strut$\mathsf{L}$} -- (-.5,3.5);
\draw[fill=black] (0,3) circle (2pt);
\draw[dashed] (0,0) node[below] {\strut$\mathsf{U}$} -- (0,3);

\node[black,left] at (\a,{\a+1+3*\lll}) {$x+1+\alpha(I_0-I_-) \, \ell$};
\node[black,left] at (\a,{\a+1+2*\lll}) {$x+1+\alpha(I_0-1-I_-) \, \ell$};
\node[red,right] at (\c,{-\c+1+2*\lll}) {$-x+1+\alpha(I_+-I_0) \, \ell$};

\draw[-latex] (-1,0) -- (.8,0) node[below] {\strut$x$};

\draw[fill=black] (\a,{\a+1+2*\lll}) circle (2pt);
\draw[fill=black] (\a,{\a+1+\lll}) circle (2pt);
\draw[fill=black] (\b,{\b+1+3*\lll}) circle (2pt);
\draw[fill=black] (\b,{\b+1+2*\lll}) circle (2pt);
\draw[fill=black] (\c,{\c+1+3*\lll}) circle (2pt);
\draw[fill=black] (\a,{-\a+1+3*\lll}) circle (2pt);
\draw[fill=black] (\b,{-\b+1+2*\lll}) circle (2pt);
\draw[fill=black] (\c,{-\c+1+1*\lll}) circle (2pt);
\end{tikzpicture}}

\caption{Representations of $Z_\pm(s_{h}^\pm,\cdot\,)$ and $\zeta^{n}(s_{h}^\pm)$ in the case at time $s_{h}$ the particle $x_{I_-}$ crosses $x=-1$, the particle $x_{I_+}$ does not cross $x=1$ and the particle $x_{I_0}$ changes direction from positive to negative.
Above we denoted $\mathsf{L} :=\alpha \, \ell \, \left(\frac{I_++I_-}{2}-I_0\right)$ and $\mathsf{U}:=\alpha \, \ell \, \left(\frac{I_++I_-}{2}-I_0+\frac{1}{2}\right)$.}
\label{f:xi3}
\end{figure}
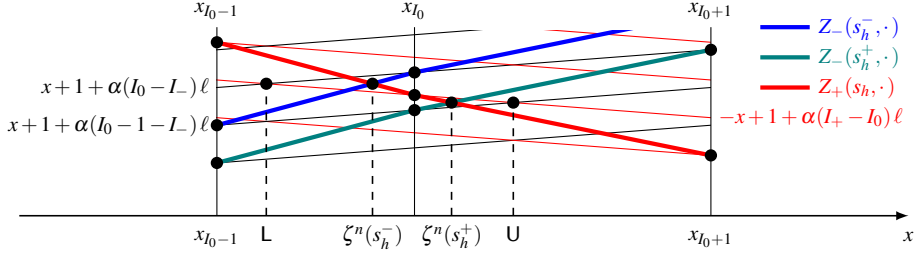

\begin{remark}\label{rem:finite-time-evac}
Let us point out that in the many-particle Hughes dynamics, whatever be the initial particle distribution, all the particles leave the corridor in a finite time.
To assess this claim, first observe that
the many-particle approximation has the ``infinite speed of propagation'' property \cite{AndrRosini-SOTA} meaning that even if the initial density can attain the saturation value $\rho_{\max}$ corresponding to the null velocity, all particles start to move instantaneously. Indeed, the leftmost and the rightmost particles move with velocities $\mp v_{\max}\neq 0$; then, by induction, it follows that all the following particles cannot stay at rest. For all $\tau>0$, the strict monotonicity of $v$ and the fact that $\rho^n(\tau,\cdot)$ takes only a finite number of values imply that  $\|\rho^n(\tau,\cdot)\|_\infty<\rho_{\max}$. Then the maximum principle of Lemma~\ref{l:maxpri} applied starting from the particle distribution at the initial time $\tau$ ensures that for all $t\geq \tau$, $|\dot x_i(t)|\geq \delta$ for some $\delta=\delta(\tau)>0$. Consequently,
the only possibility to have an infinite evacuation time is that at least one particle changes direction infinitely many times. However, we already know from Lemma~\ref{lem:switching} that
switching happens at most $n$ times.
\end{remark}	

Now we are ready to prove the main result of this section, which states the global in time solvability of the many-particle Hughes model.

\begin{theorem}\label{th:discrete-WP}
System \eqref{e:FTL}-\eqref{e:IBuiltTheSky} coupled to \eqref{e:turning}-\eqref{e:disdens} admits a unique global piecewise $\C1$ solution in the sense detailed on page~\pageref{pageC1}.
\end{theorem}

\begin{proof}
We construct a solution for \eqref{e:FTL} using a recursive procedure which ends after a finite number of steps. 
We first compute the initial position of the approximate turning curve $\zeta^n(0)$ by applying \eqref{e:turning}-\eqref{e:Xip} to the initial particle distribution.
Denote by $I_0^0$ the index such that $\overline{x}_{I_0^0}<\zeta^n(0)\leq \overline{x}_{I_0^0+1}$. Then, we let the particles $x_i$, $i\in \llbracket 0, n\rrbracket$, evolve according to
\begin{equation}\label{e:evolution}
\left\{\begin{array}{@{}l@{\qquad}l@{\quad}l@{}}
\dot{x}_{i}(t) = -v\left(R_{i-\frac{1}{2}}(t)\right) &\hbox{ if }  i \in \llbracket 0 , I_0^0 \rrbracket,
\\
\dot{x}_{i}(t) = v\left(R_{i+\frac{1}{2}}(t)\right) &\hbox{ if } i \in \llbracket I_0^0+1 , n \rrbracket,
\end{array}\right.
\end{equation}
until the maximal time $\tau_1$ which is either the time of explosion of the solution or the first time $s_1$ where a particle (or a couple of particles) exits $\mathfrak{C}$. Since \eqref{e:evolution} is a standard ODE system with locally Lipschitz right-hand side, the local solution exists, moreover, it is unique. Starting from $(x_0,\dots,x_n)$ on $[0,\tau_1)$ we can  define $\zeta^n$ on the same interval using the implicit relations \eqref{e:turning}-\eqref{e:disdens}, whose solvability has been checked.
Now, let us notice that on $[0,\tau_1)$ the solution $(x_0,\dots,x_n)$ of \eqref{e:evolution} is also a solution to the original problem \eqref{e:FTL}--\eqref{e:disdens}. This is due	
to the calculation \eqref{eq:dotzetan} of the proof of Lemma~\ref{lem:switching} which ensures that \eqref{eq:ordering} holds for $t$ between zero and the first exit time $s_1$  (it also ensures that $\zeta^n\in \C1([0,\tau_1)$). Then by Lemma~\ref{l:maxpri} we also know that the solution cannot cease to exist in finite time, therefore $\tau_1=s_1$. To sum up, we have constructed a solution to the scheme \eqref{e:FTL}--\eqref{e:disdens} on the interval $[0,s_1)$.
We can extend $x_0,\dots,x_n$ to $t=s_1$ because these functions are Lipschitz continuous with Lipschitz constant $v_{\max}$. Notice that the resulting particle paths are $\C1$ on $[0,s_1]$ and also $\zeta^n$ is a $\C1$ function on $[0,s_1)$.

For a generic time interval $(s_{h},s_{h+1}]$ we can proceed iteratively.
We compute the position $\zeta^n(s_{h})$ again by \eqref{e:turning}-\eqref{e:Xip}, we denote by $I_0^h$ the index such that $x_{I_0^h}(s_{h})<\zeta^n(s_{h})\leq x_{I_0^h+1}(s_{h})$.
Observe that the definition of $\zeta^n$ ensures its left-continuity at $t=s_h$, indeed, we have $I_\pm(s_h)=I_\pm(s_h^+)$ in \eqref{e:Xim}, \eqref{e:Xip} because  particles exiting $\mathfrak{C}$ at time $s_h$ do not re-enter $\mathfrak{C}$ at later times.
We let the particles evolve according to \eqref{e:evolution} with $I_0^h$ replacing $I_0^0$. The outcome is the definition of solution to the scheme \eqref{e:FTL}--\eqref{e:disdens} on $[s_h,s_{h+1})$; notice that it has the regularity required on page~\ref{pageC1}. 

To complete the existence proof it remains to show that the above procedure leads to a global solution for \eqref{e:FTL}.
This comes from the fact that we have a finite number of exit times $s_{h}$.
Hence, the above construction leads to a solution of \eqref{e:FTL}--\eqref{e:disdens} with $\tau=\infty$. Moreover, according to the analysis of Lemma~\ref{lem:switching}, any piecewise $\C1$ solution of \eqref{e:FTL}--\eqref{e:disdens} is actually a solution of the iterative scheme we used for the existence proof; because the latter is unique by construction, uniqueness for the scheme \eqref{e:FTL}--\eqref{e:disdens} follows. This ends the proof.\qed
\end{proof}

\delaynewpage{3}
Let us now recall that the algorithm produces the approximate density $\rho^n$ in \eqref{e:disdens}. We now provide a uniform in $n$ time continuity property of solutions $t\mapsto \rho^n(t,\cdot)\in \L1(\mathbb{R})$ constructed in Theorem~\ref{th:discrete-WP}, which we'll interpret later on as a weak $\L1$ continuity. An application of the discrete maximum principle stated in Lemma~\ref{l:maxpri} is the following 
\begin{lemma}
\label{l:dept}
For any $a<b$ and $0\leq s<t$, the approximate density $\{\rho^n\}_{n}$  satisfies
\begin{equation}
\label{e:ab}
\biggl|\int_{a}^{b}\bigl(\rho^n(t,x)-\rho^n(s,x)\bigr) \,  {\rm{d}} x\biggr|\leq C \, (t-s), 
\end{equation}
with $C:=6v_{\max}R_{\max}>0$.
\end{lemma}
\begin{proof}
We first restrict our computation to the case no particle leaves the interval $[a,b]$ for $\tau\in [s,t]$. Let us introduce the indices $0\leq I_a \leq I_b \leq n$ such that
\[x_{I_a-1}(\tau) \leq a < x_{I_a}(\tau)  \leq x_{I_b}(\tau) < b \leq x_{I_b+1}(\tau) \quad \text{for } \tau\in[s,t].\]
We distinguish two subcases:
\begin{itemize}
\item If $I_a=0$ and $I_b=n$, then \eqref{e:ab} is trivial because $\|\rho^n(\tau,\cdot\,)\|_{\L1(a,b)}=L$ for $\tau\in\{s,t\}$.
\item If $I_a\geq 1$ or $I_b\leq n-1$, then we have
\[\int_{a}^{b}\rho^n(\tau,x)\, {\rm{d}} x = \ell \, \bigl(\gamma_a(\tau)+ I_b-I_a+\gamma_b(\tau)\bigr)
\qquad\hbox{for } \tau\in\{s,t\},\]
with
\begin{align*}
\gamma_a(\tau)&:=\begin{cases} \frac{x_{I_a}(\tau)-a}{x_{I_a}(\tau)-x_{I_a-1}(\tau)} &\text{if } I_a\geq 1,
\\
0 &\text{if } I_a=0,
\end{cases}
&
\gamma_b(\tau)&:=\begin{cases} \frac{b-x_{I_b}(\tau)}{x_{I_b+1}(\tau)-x_{I_b}(\tau)} &\text{if } I_b\leq n-1,
\\
0 &\text{if } I_b=n,
\end{cases}
\end{align*}
and therefore
\[\biggl|\int_{a}^{b}(\rho^n(t,x)-\rho^n(s,x)) \,{\rm{d}} x\biggr|\leq \ell \, \bigl(\bigl|\gamma_a(t)-\gamma_a(s)\bigr|+\bigl|\gamma_b(t)-\gamma_b(s)\bigr|\bigr).\]

Let us estimate only $|\gamma_a(t)-\gamma_a(s)|$. 
Denoting $\eta(\tau):= x_{I_a}(\tau)-x_{I_a-1}(\tau)$ and $\delta(\tau):= x_{I_a}(\tau)-a$, we have
\begin{align*}
|\gamma_a(t)-\gamma_a(s)|=\frac{|\delta(t)\eta(s)-\delta(s)\eta(t)|}{\eta(t)\eta(s)}\leq{}& \frac{\eta(s)|\delta(t)-\delta(s)|+\delta(s)|\eta(t)-\eta(s)|}{\eta(t)\eta(s)}
\\
\leq{}&  \frac{3 v_{\max}R_{\max}(t-s)}{\ell},
\end{align*}
since
\begin{align*}
0\leq \delta(s)&\leq \eta(s), \qquad \eta(t) \geq \frac{\ell}{R_{\max}},
\\
|\delta(t)-\delta(s)|&=|x_{I_a}(t)-x_{I_a}(s)|\leq v_{\max} \, (t-s),
\\
|\eta(t)-\eta(s)|&\leq|x_{I_a}(t)-x_{I_a}(s)|+|x_{I_a-1}(t)-x_{I_a-1}(s)|\leq 2\, v_{\max} \, (t-s).
\end{align*}
An analogous estimate holds also for $|\gamma_b(t)-\gamma_b(s)|$.
Therefore we get \eqref{e:ab} with $C:=6v_{\max}R_{\max}>0$.
\end{itemize}
To conclude the proof it is sufficient to notice that for a fixed $n\in\mathbb{N}$, each particle can change direction only a finite number of times, thus there exists a finite number of times $\tau_i$, $i\in \llbracket 1, k \rrbracket$, such that a particle crosses $a$ or $b$. 
Denote $\tau_0:=s$ and $\tau_{k+1}:=t$.
The above argument ensures that \eqref{e:ab} holds true on each $(\tau_i,\tau_{i+1})$, $i\in \llbracket 0,k \rrbracket$. 
As a consequence
\begin{align*}
\biggl|\int_{a}^{b}\bigl(\rho^n(t,x)-\rho^n(s,x)\bigr)  \,{\rm{d}} x\biggr| 
&\leq
\sum_{i=0}^{k}\biggl|\int_{a}^{b}\bigl(\rho^n(\tau_{i+1},x)-\rho^n(\tau_i,x)\bigr) \, {\rm{d}} x\biggr|
\\&\leq
C \sum_{i=0}^{k}(\tau_{i+1}-\tau_i) = C \, (t-s),
\end{align*}
where $C:=6v_{\max}R_{\max}>0$, hence the proof is completed.\qed
\end{proof}

\section{Conditional convergence of the many-particle approximation (proof of Theorem~\ref{t:exist})}
\label{s:prma}

In order to define the many-particle approximation of model \eqref{e:model} we only need to prescribe initial data for \eqref{e:FTL} starting from a given initial density $\overline{\rho}$  satisfying \eqref{I}.
Consider the convex hull of the support of $\overline{\rho}$
\[[\overline{x}_{\min} , \overline{x}_{\max}] = \mathrm{Conv}\left( \mathrm{supp}(\overline{\rho}) \right) \subseteq [-1,1].\]
We extend $\overline{\rho}$ by zero outside $\mathfrak{C}$ and, with a slight abuse of notation, we denote it by $\overline{\rho}$.
Fix $n \in{\mathbb{N}}$ sufficiently large.
We split the interval $[\overline{x}_{\min},\overline{x}_{\max}]$ into $n$ sub-intervals having equal mass $\ell := L/n$.
To perform this task, we set $\overline{x}_{0} := \overline{x}_{\min}$ and define recursively
\begin{equation}
\overline{x}_{i} := \sup \left\{ x > \overline{x}_{i-1} : \int_{\overline{x}_{i-1}}^{x}\overline{\rho}(x) \, {\rm{d}} x < \ell\right\}, \qquad i \in \llbracket 1 , n \rrbracket.
\label{e:Slipknot}
\end{equation}
Notice that by definition
\begin{align}
\label{eq:Periphery}
&\overline{x}_{n} = \overline{x}_{\max},&
&\overline{x}_{i-1} < \overline{x}_{i},&
& \int_{\overline{x}_{i-1}}^{\overline{x}_{i}}\overline{\rho}(x) \, {\rm{d}} x = \ell,&
&i \in \llbracket 1 , n \rrbracket.
\end{align}

Theorem~\ref{th:discrete-WP} provides us with a sequence $\{(\rho^n,\zeta^n)\}_{n\in\mathbb{N}}$ of solutions to the many-particle Hughes model \eqref{e:FTL}--\eqref{e:disdens} for the discretized initial data \eqref{e:Slipknot}. In this section we prove the conditional convergence statement of Theorem~\ref{t:exist}.

We establish the compactness, in an appropriate sense, of the sequence of approximate solutions $\{(\rho^n,\zeta^n)\}_n$; then we make explicit the consistency of the scheme and perform the passage to the limit, as $n\to\infty$.

\subsection{Compactness of the sequence of approximate solutions}

As a first step, we obtain a uniform time continuity estimate for $\rho^{n}$ with respect to the 1-Wasserstein distance, see Proposition~\ref{p:JakubZytecki} in Section~\ref{s:cont}.
This estimate guarantees, together with \eqref{e:DavidMaximMicic}, the strong $\L1$-compactness with respect to both space and time by exploiting a generalized version of Aubin-Lions lemma, see Theorem~\ref{t:AL} in Section~\ref{s:comp}.
The compactness of the approximate turning curve $\{\zeta^{n}\}_n$ is proved in Section~\ref{s:turn} by applying Arzel{\`a}-Ascoli Theorem.

\subsubsection{Time continuity}
\label{s:cont}

We prove now a uniform Lipschitz estimate with respect to time in the $1$-Wasserstein distance. 
We set
\[\mathcal{M}_L := \left\{ \mu \hbox{ Radon measure on $\mathbb{R}$ with compact support} : \mu \geq0,\ \mu(\mathbb{R}) = L\right\}.
\]
The pseudo-inverse distribution function associated to $\mu \in \mathcal{M}_L$ is
\[X_{\mu}(z):= \inf\left\{x \in \mathbb{R} : \mu((- \infty,x])>z\right\}, \qquad z \in [0,L].\]
The (rescaled) $1$-Wasserstein distance between $\mu,\nu \in \mathcal{M}_L$ is defined as the $\L1$-distance of $X_{\mu}$ and $X_{\nu}$, that is
\begin{equation}
\label{17}
W_{1}(\mu,\nu):=\left\|X_{\mu}-X_{\nu}\right\|_{\L1([0,L];\mathbb{R})}.
\end{equation}
It is easily seen from \eqref{e:disdens} that $\rho^{n}(t,\cdot\,) \in \mathcal{M}_L$ for all $t\geq0$ and
\begin{align}
\label{18}
&X_{ \rho^{n}(t,\cdot\,)}(z) = x_{i}(t)+(z-i \, \ell)R_{i+\frac{1}{2}}(t)^{-1},&
&z \in (i \, \ell, (i+1) \, \ell),&
&i \in\llbracket 0 , n-1 \rrbracket.
\end{align}

\begin{proposition}
\label{p:JakubZytecki}
For any $n \in {\mathbb{N}}$ we have that
\[W_{1}\left( \rho^{n}(t,\cdot\,), \rho^{n}(s,\cdot\,)\right) \leq 2L \, v_{\max} \, |t-s|, \qquad t,s> 0.\]
\end{proposition}
\begin{proof} 
Fix $s,t>0$ with $s<t$.
From \eqref{17} and \eqref{18} it follows that
\begin{align*}
W_{1}\left( \rho^{n}(t,\cdot\,), \rho^{n}(s,\cdot\,)\right)
=\ &
\sum_{i=0}^{n-1} \int_{i \, \ell}^{(i+1) \, \ell}\left|x_{i}(t)-x_{i}(s)+(z-i \, \ell)\left(R_{i+\frac{1}{2}}(t)^{-1} - R_{i+\frac{1}{2}}(s)^{-1}\right)\right|{\rm{d}} z 
\\\leq\ & A(s,t)+B(s,t),
\end{align*}
with
\begin{align*}
A(s,t)&:=\sum_{i=0}^{n-1} \int_{i \, \ell}^{(i+1) \, \ell}|x_{i}(t)-x_{i}(s)| \, {\rm{d}} z,
\\
B(s,t)&:=\sum_{i=0}^{n-1} \int_{i \, \ell}^{(i+1) \, \ell} (z-i \, \ell) \, \left|R_{i+\frac{1}{2}}(t)^{-1}-R_{i+\frac{1}{2}}(s)^{-1}\right| \, {\rm{d}} z.
\end{align*}
We estimate now $A(s,t)$ and $B(s,t)$. Starting from $A(s,t)$, we have
\[|x_{i}(t)-x_{i}(s)| \leq v_{\max} \, (t-s)\]
and therefore
\[A(s,t) \leq \sum_{i=0}^{n-1}\ell \, v_{\max} \, (t-s) \leq L \, v_{\max} \, (t-s).\]
Turning to $B(s,t)$, we have
\[\left| R_{i+\frac{1}{2}}(t)^{-1}-R_{i+\frac{1}{2}}(s)^{-1} \right| =
\left| \frac{x_{i+1}(t)-x_{i+1}(s)}{\ell}
-
\frac{x_{i}(t)-x_{i}(s)}{\ell} \right|
\leq \frac{2v_{\max}}{\ell} \, (t-s)
\]
and therefore
\[B(s,t) \leq 
\frac{2v_{\max}}{\ell} \, (t-s) \sum_{i=0}^{n-1} \int_{i \, \ell}^{(i+1) \, \ell} (z-i \, \ell) \, {\rm{d}} z =
v_{\max} \, (t-s) \sum_{i=0}^{n-1} \ell
=L \, v_{\max} \, (t-s).\]
Adding gives $W_{1}( \rho^{n}(t,\cdot\,), \rho^{n}\left(s,\cdot\,)\right) \leq 2 L\, v_{\max} \, (t-s)$ and this concludes the proof.\qed
\end{proof} 

\subsubsection{Compactness for the approximate density}
\label{s:comp}

We deduce from the following generalized version of Aubin-Lions Lemma, see \cite{MR2005609,DiFrancescoRosini,DiFrancescoFagioliRosini-BUMI}, the $\Lloc1$-compactness of the approximations $\{ \rho^{n}\}_{n}$.
\begin{theorem}
\label{t:AL}
Take $T,L>0$ and a bounded open interval $I\subset \mathbb{R}$, possibly depending on $T$.
Let $\{\rho^{n}\}_{n}$ be a sequence in $\L{ \infty}((0,T);\L{1}(\mathbb{R}))$ such that $\left\|\rho^{n}(t,\cdot\,)\right\|_{\L1(\mathbb{R})} = L$ for all $n \in{\mathbb{N}}$ and $t \in[0,T]$.
Assume that:
\begin{enumerate}[label={\bf (\Alph*)}, leftmargin=*]
\item\label{i:A}
$\supp \{\rho^{n}(t,\cdot\,)\} \subseteq I$ for all $n \in \mathbb{N}$ and $t \in [0,T]$.
\item\label{i:B}
$\displaystyle \sup_{n} \int_{0}^{T} \TV\left(\rho^{n}(t,\cdot\,);I\right) \, {\rm{d}} t < \infty$.
\item\label{i:C}
There exists a constant $c$ independent on $n$ such that $W_{1}\left(\rho^{n}(t,\cdot\,),\rho^{n}(s,\cdot\,)\right) \leq c \, |t-s|$ for a.e.  $s,t \in (0,T)$.
\end{enumerate}
Then $\{\rho^{n}\}_{n}$ is strongly relatively compact in $\L{1}( (0,T)  \times \mathbb{R} )$.
\end{theorem}

\noindent
Indeed, for any fixed $T>0$, condition~\ref{i:A} is satisfied because
\[\supp \{ \rho^{n}(t,\cdot\,)\} \subseteq I:=[-1-v_{\max}T,1+v_{\max}T]\]
for all $n \in \mathbb{N}$ and $t \in [0,T]$.
Condition~\ref{i:B} holds true by the assumption \eqref{e:DavidMaximMicic} and because by construction $\left\| \rho^{n}(t,\cdot\,)\right\|_{\L1(\mathbb{R})} = L$.
At last, condition~\ref{i:C} holds true by Proposition~\ref{p:JakubZytecki}. 
As a result, from Theorem~\ref{t:AL} follows that $\{ \rho^{n}\}_{n}$ converges (up to a subsequence) almost everywhere and in $\L1$ on $(0,T) \times \mathbb{R}$ to a certain function we denote by $\rho$. Using the diagonal procedure, we define $\rho \in \Lloc1\left([0, \infty);\L1(\mathbb{R};[0, \rho_{\max}])\right)$ as the a.e. limit of (a subsequence of) $\{\rho^n\}_n$.

\subsubsection{Compactness for the turning curve}
\label{s:turn}

Let $x=\xi(t)$ be the turning curve implicitly defined by \eqref{e:cost0} and corresponding to the function $\rho$ obtained in Section~\ref{s:comp}.
We want to prove that $\{\zeta^{n}\}_n$ implicitly defined by \eqref{e:turning} converges to $\xi$.
In poor words, to do so we first introduce the turning curve $x=\xi^{n}(t)$ corresponding to $\rho^{n}$ and implicitly defined by \eqref{e:cost0} with $\rho$ replaced by $\rho^{n}$.
Then we prove that $\{\xi^{n}\}_n$ converges (up to a subsequence) to $\xi$ and that $\{\xi^{n}-\zeta^{n}\}_n$ converges to zero. In passing, we establish the uniform Lipschitz bound on $\{\xi^n\}_n$ and therefore, we prove that $\xi\in \mathbf{Lip}([0,\infty);\mathfrak{C})$.

By \eqref{e:cost0} we have that $\xi^{n}(t)$ is implicitly defined by
\begin{equation}
\label{e:xiturning}
\Xi_-\left(t,\xi^{n}(t)\right) = \Xi_+\left(t,\xi^{n}(t)\right),
\end{equation}
where $\Xi_\pm \colon [0,\infty) \times \mathbb{R} \to \mathbb{R}$ are defined by
\begin{equation}
\label{e:Zmp}
\begin{aligned}
\Xi_-(t,x) &:= 
\begin{cases}\displaystyle
\int_{-1}^{x} c\left( \rho^{n}(t,y)\right) \, {\rm{d}} y = x + 1 + \alpha \int_{-1}^{x} \rho^{n}(t,y) \, {\rm{d}} y
&\hbox{if }x\geq-1,
\\
x+1&\hbox{otherwise},
\end{cases}
\\
\Xi_+(t,x) &:=
\begin{cases}\displaystyle
\int_{x}^{1} c\left( \rho^{n}(t,y)\right) \, {\rm{d}} y= 1 - x + \alpha \int_{x}^{1} \rho^{n}(t,y) \, {\rm{d}} y
&\hbox{if }x \leq 1,
\\
1-x&\hbox{otherwise}.
\end{cases}
\end{aligned}
\end{equation}
We underline
that $\xi^{n}(t)$ is well defined by the strict monotonicity and continuity of $\Xi_-(t,\cdot\,)$ and $\Xi_+(t,\cdot\,)$.
In the next two lemmas, we prove further properties of $\xi^{n}$ that will be exploited in the following Proposition~\ref{p:DanWieten} to get the compactness for $\{\xi^{n}\}_n$.

\begin{lemma}
\label{l:1}
$\xi^{n}(t) \in \mathfrak{C}$ for all $t\geq0$.
\end{lemma}

\noindent
\begin{proof} 
As in \textbf{Step~I} of the proof of Lemma~\ref{l:bounxi}, assume by contradiction that $\xi^{n}(t) \leq -1$. 
Then by condition \eqref{e:xiturning} and definitions in \eqref{e:Zmp} we have
\[0\geq \Xi_-\left(t,\xi^{n}(t)\right)=\Xi_+\left(t,\xi^{n}(t)\right) \geq 2\]
and this gives a contradiction. 
Analogously, we can show that $\xi^{n}(t)< 1$; hence $\xi^{n}(t) \in \mathfrak{C}$.\qed
\end{proof} 

\begin{lemma}
The sequence $\{\xi^{n}\}_{n}$ is uniformly Lipschitz continuous in $[0,\infty)$.
\label{l:2}
\end{lemma}

\noindent
\begin{proof}
Consider $t\geq 0$ and $h>0$; we are intended to prove that
\begin{equation}\label{e:PaulMcCartney}
|\xi^n(t+h)-\xi^n(t)|\leq C\,h,
\end{equation}
with $C$ that depends only on $\rho_{\max}$, $v_{\max}$ and $\alpha$.
Subtracting identities \eqref{e:xiturning} written for times $t+h$ and $t$, rearranging the terms by separating $[-1,1]$ into intervals with endpoints $-1$, $\xi^n(t)$, $\xi^n(t+h)$ and $1$, and using the definition \eqref{e:cost} of $c$, we infer
\begin{align*}
&\int_{\xi^n(t)}^{\xi^n(t+h)} \Bigl(c(\rho^n(t,x))+c(\rho^n(t+h,x)) \Bigr) {\rm{d}} x  \\
=\ &
-\alpha \int_{-1}^{\xi^n(t)}  \Bigl(\rho^n(t+h,x)-\rho^n(t,x) \Bigr) {\rm{d}} x
+ \alpha \int_{\xi^n(t+h)}^1  \Bigl(\rho^n(t+h,x)-\rho^n(t,x) \Bigr) {\rm{d}} x.
\end{align*} 
We then take the absolute values: using the fact that $c\geq 1$, we find
\begin{align*}
2\,|\xi^n(t+h)-\xi^n(t)| & \;\leq\;  \biggl| \int_{\xi^n(t)}^{\xi^n(t+h)} \Bigl(c(\rho^n(t,x))+c(\rho^n(t+h,x)) \Bigr) {\rm{d}} x \biggr|\\
& \leq \;2\alpha\, \sup_{-1\leq a<b\leq 1}\, \biggl| \int_a^b \Bigl(\rho^n(t+h,x)-\rho^n(t,x) \Bigr) {\rm{d}} x \biggr|.
\end{align*}
Whence we derive \eqref{e:PaulMcCartney} by applying Lemma~\ref{l:dept}.\qed
\end{proof}

\begin{proposition}
\label{p:DanWieten}
$\{\xi^{n}\}_{n}$ admits a subsequence which converges in $\Lloc1([0, \infty);[-1,1])$ and locally uniformly on $[0,\infty)$ to a function $\xi \in \mathbf{Lip}([0, \infty);[-1,1])$.
\end{proposition}

\noindent
\begin{proof}
For any fixed $T>0$, by Lemmas~\ref{l:1} and~\ref{l:2} we can apply Arzel{\`a}-Ascoli Theorem and get that $\{\xi^{n}\}_{n}$ admits a subsequence which converges uniformly in $[0,T]$ to a function $\xi \in \mathbf{Lip}([0,T];[-1,1])$ with Lipschitz constant
that does not depend on $T$.
Hence, by a diagonal procedure, we obtain $\xi \in \mathbf{Lip}([0, \infty);[-1,1])$, which results to be the limit in $\Lloc1([0, \infty);[-1,1])$ of a subsequence of $\{\xi^{n}\}_{n}$.
At last, we have convergence in $\L1$ on $[0, \infty)$ because after the evacuation time we have $\xi\equiv0$.\qed
\end{proof} 

With a slight abuse of notation, we denote the subsequence of $\{\xi^{n}\}_{n}$ converging to $\xi$ again by $\{\xi^{n}\}_{n}$.
\begin{lemma}\label{l:zeta-xi}
The sequence $\{\xi^{n}-\zeta^n\}_{n}$ tends to zero uniformly on $[0,\infty)$ as $n\to\infty$.
\end{lemma}
\begin{proof}
It is sufficient to observe that, by \eqref{e:turning} and \eqref{e:xiturning}, we have for all $t>0$
\[
|\xi^{n}(t)-\zeta^{n}(t)|
=
\frac{\alpha}{2} \, \left| \int_{x_{I_+}(t)}^{1} \rho^{n}(t,y) \, {\rm{d}} y - \int_{-1}^{x_{I_-}(t)} \rho^{n}(t,y) \, {\rm{d}} y \right| 
\leq \frac{\alpha}{2} \, \ell,
\]
which tends to zero as $n\to\infty$.\qed
\end{proof}

From Proposition~\ref{p:DanWieten} and Lemma~\ref{l:zeta-xi}, we readily deduce the final claim of this paragraph.
\begin{corollary}\label{p:Korn}
Both $\{\zeta^{n}\}_{n}$ and $\{\xi^{n}\}_{n}$ converge (up to a subsequence) to some limit $\xi$ in $\Lloc1([0, \infty);[-1,1])$ and locally uniformly on $[0,\infty)$.
\end{corollary}

\subsection{Consistency of the approximation}
\label{s:cons}

Towards proving Theorem~\ref{t:exist}, our next goal is to show that the limit $(\rho,\xi)$ is indeed an entropy solution for the initial-boundary value problem \eqref{e:model} in the sense of Definition~\ref{d:entro}.
We stress that by construction $\rho$ is defined in the whole of $[0, \infty) \times \mathbb{R}$ and not only on $[0, \infty) \times \mathfrak{C}$.
Moreover, according to Proposition~\ref{p:refo}, it is enough for us to prove that $(\rho,\xi)$ is an entropy solution for the initial value problem \eqref{e:IVP} in the sense of Definition~\ref{d:entro-bis}, where, with a slight abuse of notation, we denote by $\overline{\rho}$ the extension of $\overline{\rho}$ to $\mathbb{R}$ by zero outside $\mathfrak{C}$.

\subsubsection{The relation defining the turning curve}

First, we show that the relation (the analogue of \eqref{e:cost0}) prescribing $\xi^n$ as a function of $\rho^n$ is inherited at the limit $n\to\infty$. 
\begin{proposition}
\label{p:inir}
The limit $(\rho,\xi)$ of the subsequence $\{( \rho^{n},\xi^{n})\}_{n}$ satisfies \eqref{e:cost0} almost everywhere.
\end{proposition}

\begin{proof}
By \eqref{e:xiturning} and the triangular inequality we have that
\begin{align*}
\Delta(t):=\ &\biggl| \int_{-1}^{\xi(t)} c\bigl( \rho(t,y)\bigr) \, {\rm{d}} y - \int_{\xi(t)}^{1} c\bigl( \rho(t,y)\bigr) \, {\rm{d}} y \biggr|
\\
=\ & \biggl| 
\int_{-1}^{\xi(t)} c\bigl( \rho(t,y)\bigr) \, {\rm{d}} y 
- \int_{-1}^{\xi^{n}(t)} c\bigl( \rho(t,y)\bigr) \, {\rm{d}} y 
+ \int_{-1}^{\xi^{n}(t)} c\bigl( \rho(t,y)\bigr) \, {\rm{d}} y 
\\&
- \int_{-1}^{\xi^{n}(t)} c\left( \rho^{n}(t,y)\right) \, {\rm{d}} y 
+ \int_{\xi^{n}(t)}^{1} c\left( \rho^{n}(t,y)\right) \, {\rm{d}} y 
- \int_{\xi^{n}(t)}^{1} c\bigl( \rho(t,y)\bigr) \, {\rm{d}} y 
\\&
+ \int_{\xi^{n}(t)}^{1} c\bigl( \rho(t,y)\bigr) \, {\rm{d}} y 
- \int_{\xi(t)}^{1} c\bigl( \rho(t,y)\bigr) \, {\rm{d}} y \biggr|
=\biggl| 
2 \int_{\xi^{n}(t)}^{\xi(t)} c\bigl( \rho(t,y)\bigr) \, {\rm{d}} y 
\\&
+\alpha \int_{-1}^{\xi^{n}(t)} \bigl( \rho(t,y) - \rho^{n}(t,y) \bigr) \, {\rm{d}} y 
+\alpha \int_{\xi^{n}(t)}^{1} \bigl( \rho^{n}(t,y) - \rho(t,y) \bigr) \, {\rm{d}} y 
\biggr|
\\
\leq\ &
2|\xi^{n}(t)-\xi(t)| \, (1+\alpha\,R_{\max})
+\alpha \int_{\mathfrak{C}} | \rho^{n}(t,y) - \rho(t,y) | \, {\rm{d}} y .
\end{align*}
To conclude that $\Delta(t)=0$ for a.e. $t\geq 0$, it is then sufficient to recall that $\xi^{n}\to\xi$ uniformly in $[0,T]$ and, moreover, that the fact that $\rho^{n}\to \rho $ in $\L1((0,T) \times \mathfrak{C})$ implies $\rho^n(t,\cdot)\to \rho(t,\cdot)$ for a.e. $t$.\qed
\end{proof} 

We underline that satisfying \eqref{e:cost0} implies taking values in $\mathfrak{C}$, thus Propositions~\ref{p:DanWieten} and~\ref{p:inir} imply that actually $\xi$ belongs to $\mathbf{Lip}([0, \infty);\mathfrak{C})$.

\subsubsection{The entropy condition for the density}

It remains to prove that the limit $(\rho,\xi)$ is an entropy solution to the initial value problem \eqref{e:IVP} in the sense of Definition~\ref{d:entro-bis}. Because the existence of a strong initial trace of $\rho$ relies upon local entropy inequalities, we start by establishing the latter.

\begin{proposition}\label{p:entropy-local}
The limit $(\rho,\xi)$ of the subsequence $\{( \rho^{n},\xi^{n})\}_{n}$ verifies the entropy condition \eqref{e:entro-bis} of Definition~\ref{d:entro-bis}.
\end{proposition}

\begin{proof}
By the $\L1$-convergence of $\{ \rho^{n}\}_{n}$ to $\rho$, by Corollary~\ref{p:Korn} together with the equality 
\[\int_{\mathbb{R}} \left|\sign\left(x-\xi(t)\right)-\sign\left(x-\zeta^{n}(t)\right)\right| \, {\rm{d}} x = 2 |\xi(t)-\zeta^{n}(t)|,
\] 
and by the Lipschitzianity of $\varphi$ and $f$,
the right-hand side of \eqref{e:entro-bis} coincides with the limit, as $n\to\infty$, of 
\[(\spadesuit^n) := 
\int_0^{ \infty} \int_{\mathbb{R}} \left( | \rho^{n}-\kappa| \, \varphi_t + \Phi(t,x, \rho^{n},\kappa,\zeta^n) \varphi_x \right) \, {\rm{d}} x \, {\rm{d}} t + 2 \int_0^{ \infty} f(\kappa) \varphi(t,\zeta^{n}) \, {\rm{d}} t.\]	
Therefore	
it is sufficient to prove that, for every $\kappa \in [0,\rho_{\max}]$ and every non-negative test function $\varphi \in \Cc \infty((0,\infty)\times \mathbb{R})$, the limit as $n$ goes to infinity of $(\spadesuit^n)$ is non-negative.

Recall that $t_1,\ldots,t_{H_{\rm sw}}$ are the strictly positive times at which a particle changes direction with $0<t_{h}<t_{h+1}$, $t_0:=0$. With a slight abuse of notation, instead of taking $t_{H_{\rm sw}+1}=\infty$, let $t_{H_{\rm sw}+1}$ be such that $\varphi(t,\cdot\,)\equiv0$ for any $t\geq t_{H_{\rm sw}+1}$.
It is not restrictive to take $t_{H_{\rm sw}+1}>t_{H_{\rm sw}}$. 

By \eqref{e:Plini}, see \figurename~\ref{f:Step2a}, we have
\begin{equation}
\begin{aligned}
\dot{R}_{i+\frac{1}{2}}(t) = 
-R_{i+\frac{1}{2}}(t) \, \frac{\dot{x}_{i+1}(t)-\dot{x}_{i}(t)}{x_{i+1}(t)-x_{i}(t)} = 
-R_{i+\frac{1}{2}}(t)^2 \, \frac{\dot{x}_{i+1}(t)-\dot{x}_{i}(t)}{\ell}, 
\\
t \in (t_{h},t_{h+1}), \ i \in\llbracket 0 , n-1 \rrbracket.
\end{aligned}
\label{e:Rdot}
\end{equation}
Moreover, there exists an index $I_0=I_0^h$ such that
\[x_{I_0}(t) < \zeta^{n}(t) \leq x_{I_0+1}(t),
\qquad t \in (t_{h},t_{h+1}).
\]
Notice that for any $t \in (t_{h},t_{h+1})$ we have
\begin{align*}
&\dot{R}_{i+\frac{1}{2}}(t) = 
\frac{\frac{\rm{d}}{{\rm{d}} t}v\left(R_{i+\frac{1}{2}}(t)\right)}{v'\left(R_{i+\frac{1}{2}}(t)\right)} =
\left\{\begin{array}{@{}r@{\qquad\hbox{if }}l@{}}
-\ddot{x}_{i+1}(t)\big/v'\left(R_{i+\frac{1}{2}}(t)\right)&
i \in\llbracket 0 , I_0-1 \rrbracket,
\\[10pt]
\ddot{x}_{i}(t)\big/v'\left(R_{i+\frac{1}{2}}(t)\right)&
i \in\llbracket I_0+1,n-1 \rrbracket.
\end{array}\right.
\end{align*}

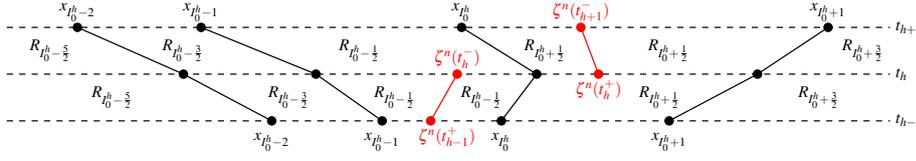
\begin{figure}[!htbp]
\resizebox{\linewidth}{!}{
\begin{tikzpicture}[x=15mm, y=8mm, semithick]

\coordinate (aa) at (-3,0);
\coordinate (ab) at (-1.5,0);
\coordinate (ac) at (.1,0);
\coordinate (acc) at (1.7,0);
\coordinate (ad) at (1,0);
\coordinate (ae) at (3.5,0);
\coordinate (ba) at (-4.2,1);
\coordinate (bb) at (-2.8,1);
\coordinate (bc) at (1.5,1);
\coordinate (bd) at (.15,1);
\coordinate (be) at (4.3,1);
\coordinate (ca) at (-2,-1);
\coordinate (cb) at (-.75,-1);
\coordinate (cc) at (-.2,-1);
\coordinate (cd) at (.6,-1);
\coordinate (ce) at (2.5,-1);

\draw[dashed] (-5,0) -- (5,0) node[right] {$t_{h}$};
\draw[dashed] (-5,1) -- (5,1) node[right] {$t_{h+1}$};
\draw[dashed] (-5,-1) -- (5,-1) node[right] {$t_{h-1}$};

\draw[fill=black] (aa) circle (2pt);
\draw[fill=black] (ab) circle (2pt);
\draw[red,fill=red] (ac) circle (2pt);
\draw[red,fill=red] (acc) circle (2pt);
\draw[fill=black] (ad) circle (2pt);
\draw[fill=black] (ae) circle (2pt);
\draw[fill=black] (ba) circle (2pt);
\draw[fill=black] (bb) circle (2pt);
\draw[red,fill=red] (bc) circle (2pt);
\draw[fill=black] (bd) circle (2pt);
\draw[fill=black] (be) circle (2pt);
\draw[fill=black] (ca) circle (2pt);
\draw[fill=black] (cb) circle (2pt);
\draw[red,fill=red] (cc) circle (2pt);
\draw[fill=black] (cd) circle (2pt);
\draw[fill=black] (ce) circle (2pt);
\draw (ca) node[below] {\strut$x_{I_0^h-2}$} -- (aa) -- (ba) node[above] {\strut$x_{I_0^h-2}$};
\draw (cb) node[below] {\strut$x_{I_0^h-1}$} -- (ab) -- (bb) node[above] {\strut$x_{I_0^h-1}$};
\draw[red] (cc) node[below] {\hspace{5.5mm}\strut$\zeta^{n}(t_{h-1}^+)$} -- (ac) node[above] {\strut$\zeta^{n}(t_{h}^-)$};
\draw[red] (acc) node[below] {\strut$\zeta^{n}(t_{h}^+)$} --(bc) node[above] {\strut$\zeta^{n}(t_{h+1}^-)$};
\draw (cd) node[below] {\strut$x_{I_0^h}$} -- (ad) -- (bd) node[above] {\strut$x_{I_0^h}$};
\draw (ce) node[below] {\strut$x_{I_0^h+1}$} -- (ae) -- (be) node[above] {\strut$x_{I_0^h+1}$};

\node at (-4.5,.5) {$R_{I_0^h-\frac{5}{2}}$};
\node at (-3,.5) {$R_{I_0^h-\frac{3}{2}}$};
\node at (-1,.5) {$R_{I_0^h-\frac{1}{2}}$};
\node at (1.1,.5) {$R_{I_0^h+\frac{1}{2}}$};
\node at (2.5,.5) {$R_{I_0^h+\frac{1}{2}}$};
\node at (4.7,.5) {$R_{I_0^h+\frac{3}{2}}$};

\node at (-3.8,-.5) {$R_{I_0^h-\frac{5}{2}}$};
\node at (-1.8,-.5) {$R_{I_0^h-\frac{3}{2}}$};
\node at (-.6,-.5) {$R_{I_0^h-\frac{1}{2}}$};
\node at (.38,-.5) {$R_{I_0^h-\frac{1}{2}}$};
\node at (2.4,-.5) {$R_{I_0^h+\frac{1}{2}}$};
\node at (4.2,-.5) {$R_{I_0^h+\frac{3}{2}}$};
\end{tikzpicture}}
\caption{Above, for convenience, the paths of both the particles and the turning curve are represented by just straight lines.
Notice that, at least in the case under consideration, $I_0^{h-1} = I_0^{h}-1$.}
\label{f:Step2a}
\end{figure}

With this notation, we have
\[(\spadesuit^n) = 
\sum_{h=0}^{s} \int_{t_{h}}^{t_{h+1}} \Biggl(
\begin{aligned}[t]
& \int_{\mathbb{R}} | \rho^{n}-\kappa| \, \varphi_t \, {\rm{d}} x 
- \int_{- \infty}^{\zeta^{n}} \!\sign( \rho^{n}-\kappa) \bigl(f( \rho^{n})-f(\kappa)\bigr) \varphi_x \, {\rm{d}} x 
\\&
+ \int_{\zeta^{n}}^{ \infty} \sign( \rho^{n}-\kappa) \bigl(f( \rho^{n})-f(\kappa)\bigr) \varphi_x \, {\rm{d}} x
+ 2f(\kappa) \varphi(t,\zeta^{n})
\Biggr) {\rm{d}} t.
\end{aligned}\]
For simplicity in the exposition, we consider a time interval $(t_{h},t_{h+1})$ for which there exists $I_0^h \in \llbracket 0 , n-1 \rrbracket$ such that
\[x_{I_0^h}(t) < \zeta^{n}(t) < x_{I_0^h+1}(t) \qquad \forall t \in (t_{h},t_{h+1}),\]
the remaining cases are analogous and are therefore omitted.
For any $t \in (t_{h},t_{h+1})$ we have by \eqref{e:disdens} that
\begin{align*}
&\int_{\mathbb{R}} | \rho^{n}-\kappa| \, \varphi_t \, {\rm{d}} x 
= 
\sum_{i=-1}^{n} \left( |R_{i+\frac{1}{2}}-\kappa| \int_{x_{i}}^{x_{i+1}} \varphi_t \, {\rm{d}} x \right)
\\
=\ &
\sum_{i=-1}^{n} \frac{{\rm{d}}}{{\rm{d}} t} \left(|R_{i+\frac{1}{2}}-\kappa| \int_{x_{i}}^{x_{i+1}} \varphi \, {\rm{d}} x\right)
- \sum_{i=0}^{n-1} \sign(R_{i+\frac{1}{2}}-\kappa) \, \dot{R}_{i+\frac{1}{2}} \, \left( \int_{x_{i}}^{x_{i+1}} \varphi \, {\rm{d}} x \right)
\\&
- \sum_{i=-1}^{n-1} |R_{i+\frac{1}{2}}-\kappa| \, \dot{x}_{i+1} \, \varphi(t,x_{i+1})
+ \sum_{i=0}^{n} |R_{i+\frac{1}{2}}-\kappa| \, \dot{x}_{i} \, \varphi(t,x_{i}).
\end{align*}
Analogously we have
\begin{align*}
&- \int_{- \infty}^{\zeta^{n}} \sign( \rho^{n}-\kappa) \bigl(f( \rho^{n})-f(\kappa)\bigr) \varphi_x \, {\rm{d}} x
\\={}&
- \sum_{i=-1}^{I_0^h-1} \sign(R_{i+\frac{1}{2}}-\kappa) \bigl(f(R_{i+\frac{1}{2}})-f(\kappa)\bigr) \bigl( \varphi(t,x_{i+1})-\varphi(t,x_{i})\bigr)
\\&
- \sign(R_{I_0^h+\frac{1}{2}}-\kappa) \bigl(f(R_{I_0^h+\frac{1}{2}})-f(\kappa)\bigr) \bigl( \varphi(t,\zeta^{n})-\varphi(t,x_{I_0^h})\bigr)
\end{align*}
and
\begin{align*}
& \int_{\zeta^{n}}^{ \infty} \sign(\rho^{n}-\kappa) \bigl(f( \rho^{n})-f(\kappa)\bigr) \varphi_x \, {\rm{d}} x
\\={}&
\sign(R_{I_0^h+\frac{1}{2}}-\kappa) \bigl(f(R_{I_0^h+\frac{1}{2}})-f(\kappa)\bigr) \bigl( \varphi(t,x_{I_0^h+1})-\varphi(t,\zeta^{n})\bigr) 
\\& 
+ \sum_{i=I_0^h+1}^{n} \sign(R_{i+\frac{1}{2}}-\kappa) \bigl(f(R_{i+\frac{1}{2}})-f(\kappa)\bigr) \bigl(\varphi(t,x_{i+1})-\varphi(t,x_{i})\bigr).
\end{align*}
Observe furthermore that, since $\varphi(0,\cdot\,)\equiv0$ and $\varphi(t,\cdot\,)\equiv0$ for any $t \geq t_{H_{\rm sw}+1}$, and moreover $R_{\frac{1}{2}},\ldots,R_{n-\frac{1}{2}}$ are $\C0$ in $(0, \infty)$ and $\C1$ in each $(t_{h},t_{h+1})$, we have
\[\sum_{h=0}^{s} \int_{t_{h}}^{t_{h+1}} \left(
\sum_{i=-1}^{n} \frac{{\rm{d}}}{{\rm{d}} t} \left(|R_{i+\frac{1}{2}}-\kappa| \int_{x_{i}}^{x_{i+1}} \varphi \, {\rm{d}} x\right)
\right) \, {\rm{d}} t = 0.\]
\delaynewpage{2}
Therefore, since \eqref{e:FTL} implies the following identities
\begin{align*}
R_{i+\frac{1}{2}} \, \dot{x}_{i+1} + f(R_{i+\frac{1}{2}}) &=
-R_{i+\frac{1}{2}} \, v(R_{i+\frac{1}{2}}) + f(R_{i+\frac{1}{2}}) = 0,&& i \leq I_0^h-1,
\\
R_{i+\frac{1}{2}} \, \dot{x}_{i} + f(R_{i+\frac{1}{2}}) &=
R_{i+\frac{1}{2}} \bigl( v(R_{i+\frac{1}{2}}) - v(R_{i-\frac{1}{2}}) \bigr),&& i \leq I_0^h-1,
\\
R_{I_0^h+\frac{1}{2}} \, \dot{x}_{I_0^h+1} - f(R_{I_0^h+\frac{1}{2}}) &=
R_{I_0^h+\frac{1}{2}} \bigl(v(R_{I_0^h+\frac{3}{2}}) - v(R_{I_0^h+\frac{1}{2}})\bigr),
\\
R_{I_0^h+\frac{1}{2}} \, \dot{x}_{I_0^h} + f(R_{I_0^h+\frac{1}{2}}) &=
R_{I_0^h+\frac{1}{2}} \bigl(-v(R_{I_0^h-\frac{1}{2}}) + v(R_{I_0^h+\frac{1}{2}})\bigr),
\\
R_{i+\frac{1}{2}} \, \dot{x}_{i+1} - f(R_{i+\frac{1}{2}}) &=
R_{i+\frac{1}{2}} \bigl( v(R_{i+\frac{3}{2}}) - v(R_{i+\frac{1}{2}}) \bigr),&& i \geq I_0^h+1,
\\
R_{i+\frac{1}{2}} \, \dot{x}_{i} - f(R_{i+\frac{1}{2}}) &=
R_{i+\frac{1}{2}} \, v(R_{i+\frac{1}{2}}) - f(R_{i+\frac{1}{2}}) = 0,&& i \geq I_0^h+1,
\end{align*}
we get $(\spadesuit^n)=\kappa(\spadesuit^n_1)+(\spadesuit^n_2)$, where
\begin{align}\nonumber
&( \spadesuit^n_1):= \sum_{h=0}^{s} \int_{t_{h}}^{t_{h+1}} \biggl\{
- \sum_{i=0}^{I_0^h-1}  \sign(R_{i+\frac{1}{2}}-\kappa) \, \Delta_v(\kappa,R_{i-\frac{1}{2}}) \, \varphi(t,x_{i})
\\&\nonumber
+\sum_{i=-1}^{I_0^h-1}    \sign(R_{i+\frac{1}{2}}-\kappa)\, \Delta_v(\kappa,R_{i+\frac{1}{2}}) \, \varphi(t,x_{i+1})
-  \sign(R_{I_0^h+\frac{1}{2}}-\kappa) \, \Delta_v(\kappa,R_{I_0^h-\frac{1}{2}}) \, \varphi(t,x_{I_0^h})
\\
&\nonumber + \sum_{i=I_0^h+1}^{n}  \sign(R_{i+\frac{1}{2}}-\kappa) \, \Delta_v(\kappa,R_{i+\frac{1}{2}}) \,  \varphi(t,x_{i})
- \sum_{i=I_0^h+1}^{n-1}  \sign(R_{i+\frac{1}{2}}-\kappa) \, \Delta_v(\kappa,R_{i+\frac{3}{2}}) \,  \varphi(t,x_{i+1}) 
\\
&-  \sign(R_{I_0^h+\frac{1}{2}}-\kappa) \, \Delta_v(\kappa,R_{I_0^h+\frac{3}{2}}) \, \varphi(t,x_{I_0^h+1})
+ 2 \bigl(1 + \sign(R_{I_0^h+\frac{1}{2}}-\kappa)\bigr)\, v(\kappa) \,  \varphi(t,\zeta^{n})\biggr\} \ {\rm d} t,
\label{e:spade_1}
\\\nonumber
&(\spadesuit^n_2) :=
\sum_{h=0}^{s} \int_{t_{h}}^{t_{h+1}} \biggl\{
- \sum_{i=0}^{I_0^h-1} \sign(R_{i+\frac{1}{2}}-\kappa) 
\biggl[
\dot{R}_{i+\frac{1}{2}} \int_{x_{i}}^{x_{i+1}} \varphi \,  {\rm d} x 
- R_{i+\frac{1}{2}} \, \Delta_v(R_{i+\frac{1}{2}},R_{i-\frac{1}{2}}) \, \varphi(t,x_{i}) 
\biggr]
\\&\nonumber
- \sum_{i=I_0^h+1}^{n-1}\!\! \sign(R_{i+\frac{1}{2}}\!-\!\kappa) 
\biggl[
\dot{R}_{i+\frac{1}{2}}   \int_{x_{i}}^{x_{i+1}} \varphi  \, {\rm d}x 
+ R_{i+\frac{1}{2}} \, \Delta_v(R_{i+\frac{3}{2}},R_{i+\frac{1}{2}}) \, \varphi(t,x_{i+1}) \biggr]
\\&\nonumber
-\sign(R_{I_0^h+\frac{1}{2}}-\kappa) \biggl[
\dot{R}_{I_0^h+\frac{1}{2}} \int_{x_{I_0^h}}^{x_{I_0^h+1}} \varphi \,  {\rm d}x
+ 2\, f(R_{I_0^h+\frac{1}{2}})  \, \varphi(t,\zeta^{n}) 
\\&
+ R_{I_0^h+\frac{1}{2}} \, \Delta_v(R_{I_0^h+\frac{3}{2}},R_{I_0^h+\frac{1}{2}}) \, \varphi(t,x_{I_0^h+1})
- R_{I_0^h+\frac{1}{2}} \, \Delta_v(R_{I_0^h+\frac{1}{2}},R_{I_0^h-\frac{1}{2}}) \, \varphi(t,x_{I_0^h})
\biggr]
\biggr\} \, {\rm d} t,
\label{e:spade_2}
\end{align}
with
\begin{equation}
\label{e:notations}
\Delta_v(\alpha,\beta) := v(\alpha) - v(\beta).
\end{equation}
We first prove that $(\spadesuit^n_2)\to 0$ as $n\to \infty$. Denote by $\mathcal{L}_v$ and $\mathcal{L}_\varphi$ the Lipschitz constants for $v$ and $\varphi$, respectively. By applying \eqref{e:Rdot} we get 
\begin{align}\nonumber
&\left| - \sum_{i=0}^{I_0^h-1} \dot{R}_{i+\frac{1}{2}} \int_{x_{i}}^{x_{i+1}} \varphi \, {\rm{d}} x
+ \sum_{i=0}^{I_0^h-1} R_{i+\frac{1}{2}}  \, \Delta_v(R_{i+\frac{1}{2}},R_{i-\frac{1}{2}}) \varphi(t,x_{i}) \right|
\\={}&\nonumber
\left| \sum_{i=0}^{I_0^h-1} R_{i+\frac{1}{2}}  \, \Delta_v(R_{i-\frac{1}{2}},R_{i+\frac{1}{2}}) \fint_{x_{i}}^{x_{i+1}} \bigl( \varphi(t,x) - \varphi(t,x_{i}) \bigr) {\rm{d}} x \right|
\\ \leq {}&
\mathcal{L}_v \, \mathcal{L}_\varphi  \sum_{i=0}^{I_0^h-1} |R_{i-\frac{1}{2}}-R_{i+\frac{1}{2}}| \, R_{i+\frac{1}{2}} \, \frac{x_{i+1}-x_{i}}{2}
= \frac{\mathcal{L}_v \, \mathcal{L}_\varphi}{2} \, \ell \sum_{i=0}^{I_0^h-1} |R_{i-\frac{1}{2}}-R_{i+\frac{1}{2}}| ,
\label{e:ChickCorea}
\\\nonumber
&\left| -\sum_{i=I_0^h+1}^{n-1}\dot{R}_{i+\frac{1}{2}} \int_{x_{i}}^{x_{i+1}} \varphi \, {\rm{d}} x 
- \sum_{i=I_0^h+1}^{n-1}  R_{i+\frac{1}{2}}  \, \Delta_v(R_{i+\frac{3}{2}},R_{i+\frac{1}{2}}) \varphi(t,x_{i+1}) \right|
\\={}&\nonumber
\left|- \sum_{i=I_0^h+1}^{n-1} R_{i+\frac{1}{2}}  \, \Delta_v(R_{i+\frac{3}{2}},R_{i+\frac{1}{2}}) \fint_{x_{i}}^{x_{i+1}} \bigl( \varphi(t,x_{i+1}) - \varphi(t,x) \bigr) {\rm{d}} x \right|
\\ \leq {}&
\mathcal{L}_v \, \mathcal{L}_\varphi  \sum_{i=I_0^h+1}^{n-1} |R_{i+\frac{3}{2}}-R_{i+\frac{1}{2}}| \, R_{i+\frac{1}{2}} \, \frac{x_{i+1}-x_{i}}{2} 
= \frac{\mathcal{L}_v \, \mathcal{L}_\varphi}{2} \,  \ell \sum_{i=I_0^h+1}^{n-1} |R_{i+\frac{3}{2}}-R_{i+\frac{1}{2}}| ,
\label{e:JohnColtrane}
\\\nonumber
&\Biggl|- \dot{R}_{I_0^h+\frac{1}{2}} \int_{x_{I_0^h}}^{x_{I_0^h+1}} \varphi \, {\rm{d}} x 
- R_{I_0^h+\frac{1}{2}}  \, \Delta_v(R_{I_0^h+\frac{3}{2}},R_{I_0^h+\frac{1}{2}}) \, \varphi(t,x_{I_0^h+1}) 
\\&\nonumber
- R_{I_0^h+\frac{1}{2}}  \, \Delta_v(R_{I_0^h-\frac{1}{2}},R_{I_0^h+\frac{1}{2}}) \, \varphi(t,x_{I_0^h})
-2 f(R_{I_0^h+\frac{1}{2}}) \varphi(t,\zeta^{n})\Biggr|
\\={}&\nonumber
\Biggl|R_{I_0^h+\frac{1}{2}}  \, \Delta_v(R_{I_0^h+\frac{1}{2}},R_{I_0^h+\frac{3}{2}})  \fint_{x_{I_0^h}}^{x_{I_0^h+1}} \bigl(\varphi(t,x_{I_0^h+1})-\varphi(t,x)\bigr) {\rm{d}} x 
\\&\nonumber
+ R_{I_0^h+\frac{1}{2}}  \, \Delta_v(R_{I_0^h-\frac{1}{2}},R_{I_0^h+\frac{1}{2}}) \fint_{x_{I_0^h}}^{x_{I_0^h+1}} \bigl(\varphi(t,x)-\varphi(t,x_{I_0^h})\bigr) {\rm{d}} x 
\\&\nonumber
+2f(R_{I_0^h+\frac{1}{2}}) \fint_{x_{I_0^h}}^{x_{I_0^h+1}} \bigl(\varphi(t,x)-\varphi(t,\zeta^{n})\bigr) {\rm{d}} x \Biggr|
\\ \leq {}&\nonumber
\frac{\mathcal{L}_v \, \mathcal{L}_{\varphi}}{2} \, \left[ |R_{I_0^h+\frac{1}{2}}-R_{I_0^h+\frac{3}{2}}| + |R_{I_0^h-\frac{1}{2}}-R_{I_0^h+\frac{1}{2}}| \right] \ell
+2 \mathcal{L}_\varphi \, v_{\max} \, R_{I_0^h+\frac{1}{2}} \fint_{x_{I_0^h}}^{x_{I_0^h+1}} |x-\zeta^{n}| \,{\rm{d}} x
\\ \leq {}&
\frac{\mathcal{L}_v \, \mathcal{L}_{\varphi}}{2} \, \left[ |R_{I_0^h+\frac{1}{2}}-R_{I_0^h+\frac{3}{2}}| + |R_{I_0^h-\frac{1}{2}}-R_{I_0^h+\frac{1}{2}}| \right] \ell
+2 \mathcal{L}_\varphi \, v_{\max} \, \ell.
\label{e:MilesDavis}
\end{align}
Therefore, we get
\begin{align*}
|(\spadesuit^n_2)| \leq {}& \mathcal{L}_\varphi \left( \frac{\mathcal{L}_{v}}{2} \int_0^{t_{s+1}} \TV\left( \rho^{n}(t,\cdot\,)\right) \, {\rm{d}} t + 2 v_{\max} \, t_{N+1} \right) \, \ell,
\end{align*}
where the right hand side converges to zero because $\ell$ do so and by \eqref{e:DavidMaximMicic}, hence $(\spadesuit^n_2)\to 0$ as $n\to \infty$ as we claimed. As a consequence, to conclude the proof it is sufficient to show that $( \spadesuit^n_1)\geq 0$.
Since
\begin{align}
\nonumber &- \sum_{i=0}^{I_0^h-1} \sign(R_{i+\frac{1}{2}}-\kappa)  \, \Delta_v(\kappa,R_{i-\frac{1}{2}}) \, \varphi(t,x_{i})
+\sum_{i=-1}^{I_0^h-1} \sign(R_{i+\frac{1}{2}}-\kappa)  \, \Delta_v(\kappa,R_{i+\frac{1}{2}}) \, \varphi(t,x_{i+1})
\\
\nonumber &
- \sign(R_{I_0^h+\frac{1}{2}}-\kappa)  \, \Delta_v(\kappa,R_{I_0^h-\frac{1}{2}}) \, \varphi(t,x_{I_0^h})
\\
\label{e:milan} ={}&
\sum_{i=0}^{I_0^h} \bigl(\sign(R_{i-\frac{1}{2}}-\kappa) - \sign(R_{i+\frac{1}{2}}-\kappa) \bigr)  \, \Delta_v(\kappa,R_{i-\frac{1}{2}}) \, \varphi(t,x_{i}) 
\geq 0,
\\[5pt]
\nonumber &\sum_{i=I_0^h+1}^{n} \sign(R_{i+\frac{1}{2}}-\kappa)  \, \Delta_v(\kappa,R_{i+\frac{1}{2}}) \, \varphi(t,x_{i})
- \sum_{i=I_0^h+1}^{n-1} \sign(R_{i+\frac{1}{2}}-\kappa)  \, \Delta_v(\kappa,R_{i+\frac{3}{2}}) \, \varphi(t,x_{i+1})
\\
\nonumber &
- \sign(R_{I_0^h+\frac{1}{2}}-\kappa)  \, \Delta_v(\kappa,R_{I_0^h+\frac{3}{2}}) \, \varphi(t,x_{I_0^h+1})
\\
\label{e:roma} ={}&
\sum_{i=I_0^h+1}^{n} \bigl(\sign(R_{i+\frac{1}{2}}-\kappa)-\sign(R_{i-\frac{1}{2}}-\kappa)\bigr)  \, \Delta_v(\kappa,R_{i+\frac{1}{2}}) \, \varphi(t,x_{i}) 
\geq 0,
\\[5pt]
\label{e:lazio}&2 \bigl(1 + \sign(R_{I_0^h+\frac{1}{2}}-\kappa)\bigr) v(\kappa)  \varphi(t,\zeta^{n})
\geq 0,
\end{align}
then we have $(\spadesuit^n_1)\geq 0$ and this completes the proof.\qed
\end{proof}

\subsubsection{The initial condition for the density}

As a last step, we show that the solution satisfies the initial condition.
We first prove that the initial datum is achieved in a weak sense.

\begin{lemma}
\label{l:ini}
The approximate density $\{\rho^n(0,\cdot\,)\}_{n}$ converges to $\overline{\rho}$ weakly in $\L1(\mathbb{R})$.
\end{lemma}
\begin{proof}
Observe that by construction, for all $n$ there holds $\|\rho^n(0,\cdot\,)\|_\infty\leq\rho_{\max}$, moreover, the support of 	$\rho^n(0,\cdot\,)$ is contained in $\overline{\mathfrak{C}}$. The Dunford-Pettis theorem applies, so that $\{\rho^n(0,\cdot\,)\}_n$ is weakly relatively compact in $\L1(\mathbb{R})$. Since the $\mathcal D'$ convergence implies the weak $\L1$ convergence by a standard argument (by contradiction, extracting a subsequence not converging weakly in $\L1$ to $\overline{\rho}$ and applying the compactness claim to this subsequence), it is enough to prove that $\{\rho^n(0,\cdot\,)\}_n$ converges to $\overline{\rho}$ in $\mathcal D'(\mathbb{R})$.

Fix $\varphi \in \Cc\infty(\mathbb{R})$.
From \eqref{eq:Periphery}, \eqref{e:Plini} and \eqref{e:disdens} it follows that
\begin{align*}
&\biggl| \int_{\mathbb{R}} \bigl( \overline{\rho}(x)-\rho^{n}(0,x) \bigr) \varphi(x) \, {\rm{d}}x \biggr|
\leq \sum_{i=0}^{n-1} \biggl| \int_{\overline{x}_i}^{\overline{x}_{i+1}} \bigl( \overline{\rho}(x) - R_{i+\frac{1}{2}}(0) \bigr)  \varphi(x) \, {\rm{d}}x \biggr|
\\
=\ & \sum_{i=0}^{n-1} \biggl| \int_{\overline{x}_i}^{\overline{x}_{i+1}} \biggl[ \overline{\rho}(x) - \fint_{\overline{x}_i}^{\overline{x}_{i+1}} \overline{\rho}(y) \, {\rm{d}}y \biggr]  \varphi(x) \, {\rm{d}}x \biggr|
= \sum_{i=0}^{n-1} \biggl| \int_{\overline{x}_i}^{\overline{x}_{i+1}} \overline{\rho}(x)  \biggl[ \varphi(x) - \fint_{\overline{x}_i}^{\overline{x}_{i+1}} \varphi(y) \, {\rm{d}}y \biggr] {\rm{d}}x \biggr|
\\
\leq\ & \mathcal{L}_{\varphi}  \sum_{i=0}^{n-1} (\bar{x}_{i+1}-\bar{x}_i) \int_{\overline{x}_i}^{\overline{x}_{i+1}} \overline{\rho}(x) \, {\rm{d}} x
= \mathcal{L}_{\varphi}  (\bar{x}_{\max}-\bar{x}_{\min}) \, \ell
\leq 2 \mathcal{L}_{\varphi} \, \ell,
\end{align*}
which vanishes as $n$ goes to infinity.
This concludes the proof.\qed
\end{proof}

\delaynewpage{}
Now we are in the position to prove the $\L1$-continuity near $t=0$. We have 
\begin{proposition}
\label{p:ini}
The limit $\rho$ belongs to $\C0([0,\infty);\L1(\mathbb{R}))$ and satisfies the initial condition $\rho(0,\cdot)=\overline \rho$ on $\mathbb{R}$.
\end{proposition}

\begin{proof} 
We argue as in Remark~\ref{rem:L1-continuity}, to get the continuity of $\rho$ in time with values in $\Lloc1(\mathbb{R})$. To observe that the local integrability can be replaced with the global one, observe that the supports of $\rho^n(t,\cdot)$ are contained, uniformly in $n$, in the interval $[-1-v_{\max}t, 1+v_{\max}t]$ because particles' maximal velocity is $v_{\max}$. Therefore $\rho(0,\cdot)$ is well defined in the sense of a strong $\L1(\mathbb{R})$ trace which necessarily coincides with its weak $\L1(\mathbb{R})$ trace; by Lemma~\ref{l:ini}, the latter equals $\overline{\rho}$.
\qed
\end{proof}

\subsection{Two applications to non-interacting crowds.}
\label{s:applications}

Applications of Theorem~\ref{t:exist} to existence for the Hughes model \eqref{e:model} are subject to $\mathbf{BV}$ control in space of sequences of approximate solutions generated by 
the many-particle Hughes scheme of Section~\ref{s:mpa}. These estimates are easy to obtain if the particles in the scheme do not switch, in other words, if the turning curve separates the particles into two non-interacting crowds. We examine two cases where this situation occurs.	

Our first application deals with densities that are ``well separated'' from the origin. 
Let $\mathcal{S}_1$ be the space of functions $\rho$ in $\L \infty(\mathfrak{C};[0, \rho_{\max}])$ satisfying \eqref{I} with $L < 2/\alpha$ and having support in $[-1,1] \setminus [-\frac{\alpha\,L}{2},\frac{\alpha\,L}{2}]$.
\begin{corollary}
Consider the cost function \eqref{e:cost} and assume \eqref{V1}.
For any initial datum $\overline{\rho}$ in $\mathcal{S}_1$, the sequence of approximate solutions $( \rho^{n},\zeta^{n})$ of the many-particle Hughes' model of Section~\ref{s:mpa} converges, up to a subsequence, in $\L1((0,T)\times\mathfrak{C})\times \C0([0,T])$ for all $T>0$, to the unique $\mathbf{BV}$-regular entropy solution $(\rho,\xi)$ to the initial-boundary value problem \eqref{e:model} in the sense of Definition~\ref{d:entro}, such that $\rho(t,\cdot\,) \in \mathcal{S}_1$ for all $t>0$.
\label{t:exist1}
\end{corollary}

As a second result, we provide the analogous result for the ``symmetric'' case already investigated in \cite{AmadoriGoatinRosini,DiFrancescoFagioliRosiniRusso}.
Let $\mathcal{S}_2$ be the space of functions $\rho$ in $\L \infty(\mathfrak{C};[0, \rho_{\max}])$ satisfying \eqref{I} that are even, i.e., $\rho (x)= \rho (-x)$ for a.e.~$x \in\mathfrak{C}$.

\begin{corollary}
The result of Corollary~\ref{t:exist1} holds true with $\mathcal{S}_1$ replaced by $\mathcal{S}_2$.	
\label{t:exist2}
\end{corollary}

\begin{proof}[of Corollaries~\ref{t:exist1}, \ref{t:exist2}]
To assess the claims, let us observe that no particle switches direction, that is, no particle interacts with the approximate turning curve $\zeta^{n}$, implicitly defined by \eqref{e:turning} (and therefore at the limit, we fall into the framework of entropy solutions without non-classical shocks, studied in \cite{DiFrancescoFagioliRosiniRussoKRM,AmadoriGoatinRosini}).
More precisely, with reference to Corollary~\ref{t:exist1}, in Lemma~\ref{l:bounxi} we see that the approximate turning point $\zeta^{n}$ takes values in $\mathfrak{C}\cap [-\frac{\alpha\,L}{2},\frac{\alpha\,L}{2}]$; as a consequence, if no particle is initially in such interval, then the approximate turning curve cannot reach any of them.
Moreover, in Proposition~\ref{p:Opeth} we have proved that if a particle changes direction, then exactly one particle leaves $\mathfrak{C}$; however, with reference to Corollary~\ref{t:exist2},  in the symmetric case this cannot occur because the many-particle Hughes model is also symmetric.

Consequently, the $\mathbf{BV}$ estimate \eqref{e:DavidMaximMicic} of Theorem~\ref{t:exist} holds with $\hbox{\bf TV} := \TV(\overline{\rho})+2R_{\max}$ due to the application of the results of \cite{DiFrancescoRosini,DiFrancescoFagioliRosini-BUMI} to the particles situated on each side from $\zeta^n$, with the term $2R_{\max}$ to control the variation of $\rho^n$ between the particles that surround $\zeta^n$.
Finally, in the well-separated case  it is easily seen that the limit $\rho$ of (any subsequence of) the many-particle Hughes scheme yields solutions that are not only $\mathbf{BV}$-regular but also well-separated in the sense of Definition~\ref{d:well-separated}; by Theorem~\ref{t:WS-stability}, such solution is unique and therefore one can bypass extraction of a subsequence. The same holds true in the symmetric case with $R_{\max}<\rho_{\max}$, but notice that the uniqueness of a solution in the symmetric case was proved in \cite{AmadoriGoatinRosini}. Therefore also in this case the whole sequence $\{(\rho^n,\zeta^n) \}_n$ generated by the many-particle Hughes' model converges to the unique solution of the continuous Hughes model \eqref{e:model}.
\qed
\end{proof}

\section{Existence via a local compactness argument}\label{s:existBVloc} 

In this section, we prove our second existence result in the framework of entropy solutions which authorizes the formation of non-classical shocks at the location of the turning curve.
The proof is a refinement of the one of Theorem~\ref{t:exist}.
In particular, we replace the uniform $\mathbf{BV}$ compactness and convergence arguments of Section~\ref{s:prma}, which rely upon condition \eqref{e:DavidMaximMicic}, by a localized compactness and convergence arguments. The latter ones rely upon a localized spatial $\mathbf{BV}$ bound and is based on the Oleinik-kind one-sided Lipschitz regularization known for scalar conservation laws with convex (or concave) flux and also for their follow-the-leader particle approximation, see \cite{DiFrancescoFagioliRosini-BUMI}.
Our result in Theorem~\ref{t:existBVloc} holds in case the velocity map satisfies also the additional assumption \eqref{Vbis}, while the only requirement for the initial datum $\overline{\rho}$ is to satisfy \eqref{I}. We argue by reduction, away from the turning curve, to the following known result concerning the many-particle approximation of the standard LWR model.

\begin{theorem}{(adapted from \cite{DiFrancescoFagioliRosini-BUMI})}~ \label{th:OleinikCompactness}
Assume \eqref{V1}, \eqref{Vbis}.
Fix $L>0$ and, for $n\in \mathbb{N}$, set $\ell:=L/n$ and consider a sequence of many-particle approximation corresponding to some  $\overline{\tilde{x}}_i$, ${i \in \llbracket 0 , n \rrbracket}$, verifying  
$$
-1\leq \overline{\tilde{x}}_{0}<\overline{\tilde{x}}_{n}\leq 1 \text{\; and}\quad  
\overline{\tilde{x}}_{i+1}-\overline{\tilde{x}}_{i}\geq \ell/\rho_{\max} \text{\;\;for $i\in \llbracket 0 , n-1 \rrbracket$, }
$$
and evolved according to the ODE system
\begin{equation}
\label{e:FTL-basic}
\left\{\begin{array}{@{}l@{\qquad}l@{\quad}l@{}}
\dot{\tilde{x}}_{n}(t) = v_{\max}, &\\
\dot{\tilde{x}}_{i}(t) = v\left(\frac{\ell}{\tilde{x}_{i+1}(t)-\tilde{x}_{i}(t)}\right) ,&  i \in \llbracket 0 , n-1 \rrbracket,
\\
\tilde{x}_{i}(0) = \overline{\tilde{x}}_{i}, & i \in \llbracket 0 , n \rrbracket.
\end{array}\right.
\end{equation}
Let $\tilde{\rho}^n$ be the corresponding approximate density, i.e., 
\begin{equation}
\label{tildedisdens}
\tilde{\rho}^{n}(t,y) := \sum_{i = 0}^{n-1}\frac{\ell}{\tilde{x}_{i+1}(t)-\tilde{x}_i(t)} \, \mathbbm{1}_{[\tilde{x}_{i}(t),\tilde{x}_{i+1}(t))}(y).
\end{equation}
Then, for any fixed $\varepsilon>0$ and  $T>0$, there exist two positive constants $C_1$ and $C_2$, depending only on $T$ and $v_{\max}$, such that
\begin{equation}\label{e:EricClapton}
\TV\left(\tilde{\rho}^n(t,\cdot\,);\mathbb{R}\right) \leq C_1+\frac{C_2}{\varepsilon} \quad \text{for all\; } t \in [\varepsilon,T].
\end{equation}
Moreover,  $\{\tilde{\rho}^n\}_n$ converges, up to a subsequence, almost everywhere and strongly in $\L1((0,T)\times \mathbb{R})$, to an entropy solution of $\rho_t+(\rho \, v(\rho))_x=0$.
\end{theorem}

\noindent
The proof, fully based upon the estimate of \cite[Proposition~3]{DiFrancescoFagioliRosini-BUMI}, follows the argumentation of \cite{DiFrancescoFagioliRosini-BUMI} with the only difference that the initial condition does not originate from the many-particle approximation of a given initial datum; note that we make no claim concerning the initial datum for the limit of the extracted subsequence (but this issue can be handled through the arguments of Lemma~\ref{l:ini}), nor about the uniqueness of the limit, nor about convergence of the whole sequence $\{\tilde{\rho}^n\}_n$. The conclusion of \cite[Proposition~3]{DiFrancescoFagioliRosini-BUMI} permits to control the variation of solution for $x\in [a,b]$ with $a,b$-dependent constants $C_1$, $C_2$; but in our case, the support of $\rho^n(t,\cdot)$ for $t\leq T$ is included in $[a_T,b_T]:=[-1-Tv_{\max}, 1+Tv_{\max}]$. Therefore the constants $C_1$, $C_2$ depend solely on $T$, $v_{\max}$.

\begin{remark}\label{rem:leftwardLWR}
A fully analogous result holds for the system of particles moving leftward, i.e., with $v$ replaced by $-v$ and $\tilde{x}_0$ playing the role of $\tilde{x}_n$.
\end{remark}

\subsection{Compactness for the approximate density and localized $\mathbf{BV}$ bounds}

To start with, observe that Proposition~\ref{p:DanWieten} and Lemmas~\ref{l:1}, \ref{l:zeta-xi} leading to Corollary~\ref{p:Korn} hold without any assumption on the total variation of the initial datum and of the approximate solution. Therefore, upon extraction of a subsequence (we will not relabel the extracted subsequences in what follows), we can define  $\xi$ as the limit, in the topology of the locally uniform convergence on $[0,\infty)$, of $\{\zeta^n\}_n$.	 We introduce
\begin{equation}
\label{e:O}
\mathcal{O} := \left\{(t,x) \in(0,\infty)\times\mathbb{R} : x\neq \xi(t)\right\}.
\end{equation}
Then, we are able to establish the localized $\mathbf{BV}$ bound and compactness for $\{\rho^n\}_n$, more precisely we have the following result.

\begin{proposition}
Let $(t_*,x_*) \in \mathcal{O}$. Then the following holds.

\noindent
{\bf (i)} There exists an open neighborhood $\mathcal{O}_{(t_*,x_*)}\subset \mathcal{O}$ of $(t_*,x_*)$ such that, up to extraction of a subsequence, $\{\rho^n\}_{n}$ converges a.e.\ on $\mathcal{O}_{(t_*,x_*)}$.

\noindent
{\bf(ii)} The limit $\rho$ of any such subsequence of $\{\rho^n\}_{n}$ is an entropy solution on $\mathcal{O}_{(t_*,x_*)}$ in the sense of Definition~\ref{d:entro-bis} with $\varphi \in \Cc\infty(\mathcal{O}_{(t_*,x_*)};[0,\infty))$.

\noindent Moreover, we have

\noindent {\bf (iii)} 
There exists $\beta>0$, depending only on $v_{\max}$ and $\xi$, such that, given any $\varepsilon>0$ and $T>0$,
for all $t\in[\varepsilon,T]$ and large enough $n$ it holds that the total variation of $\rho^n(t,\cdot\,)$ on $(-\infty,\xi(t)-\beta\varepsilon]\cup [\xi(t)+\beta\varepsilon,\infty)$ is bounded by a constant that depends on $T$ and $\varepsilon$, but is independent of $t\in [0,T]$ and $n$.
\label{prop:rectangle_}
\end{proposition}		

\begin{proof}
By assumption and definition \eqref{e:O} we have $x_*\neq\xi(t_*)$, hence by the symmetry of the problem it is sufficient to consider the case $x_*>\xi(t_*)$ (for handling the symmetric case, we should rely upon Remark~\ref{rem:leftwardLWR} for the LWR model with a non-positive velocity field).
We first notice that the graph of $\xi$ is closed and $(t_*,x_*)$ does not belong to this graph. 
Moreover, we have that $t_*>0$, hence there exists a constant $d = d(t_*,x_*)>0$ such that
\[[t_*-d,t_*+d] \times [x_*-d,\infty) \subset \mathcal{O}.\]
We then set 
\begin{equation}\label{eq:neighb_}
\mathcal{O}_{(t_*,x_*)}:=( \tau_0,\tau_1 ) \times(x_*-\varepsilon,\infty),
\end{equation}
see \figurename~\ref{f:rectangles}, where 
\begin{align*}
\tau_0 := t_*-\varepsilon, \quad \tau_1:= t_*+\varepsilon, \quad \varepsilon = \varepsilon(t_*,x_*) := \frac{d}{16\left(v_{\max}+\mathcal{L}_{\xi}+1\right)},
\end{align*}
with $\mathcal{L}_{\xi}$
being the Lipschitz constant for $\xi$ (recall that $\xi$ is the limit of uniformly Lipschitz curves $\xi^n$, see Lemma~\ref{l:2} and Proposition~\ref{p:DanWieten}).

The claims of the proposition will follow by applying Theorem~\ref{th:OleinikCompactness} to a well-chosen set of particles $\tilde{x}_i$, $i\in \llbracket 0, n \rrbracket$, corresponding to the standard LWR model defined for times $t \geq \tau_0$ and which gives rise to the density $\tilde{\rho}^n$ defined in \eqref{tildedisdens} and verifying the two properties
\begin{equation}
\label{e:GeorgeHarrison}
\left\|\tilde\rho^n-\rho^n\right\|_{\L1(\mathcal{O}_{(t_*,x_*)})} \to 0 \;\;\text{as $n\to\infty$},
\end{equation}
\begin{equation}
\label{e:RingoStarr}
\bigl| \TV\bigl(\tilde{\rho}^n(t,\cdot\,) ; [x_*-\varepsilon,\infty)\bigr)- \TV\bigl(\rho^n(t,\cdot\,) ; [x_*-\varepsilon,\infty)\bigr) \bigr|\leq \rho_{\max} \quad \forall t\in [\tau_0,\tau_1].
\end{equation}
Indeed, with \eqref{e:GeorgeHarrison}, if $\{\tilde\rho^n\}_n$ admits a subsequence which converges  a.e.\ on $\mathcal{O}_{(t_*,x_*)}$, than so does $\{\rho^n\}_n$, and the limit $\rho$ is an entropy solution of the standard LWR equation in $\mathcal{O}_{(t_*,x_*)}$, see Theorem~\ref{th:OleinikCompactness}: this will justify the claims {\bf (i)} and {\bf (ii)} of the proposition. 
As to property \eqref{e:RingoStarr}, it will be combined in the final part of the proof with the variation bound \eqref{e:EricClapton} for the set of particles $\tilde{x}_0,\dots, \tilde{x}_n$, which will permit to assess the claim {\bf (iii)}.

Therefore, let us show how to construct a proper set of particles $\tilde{x}_i$, $i\in \llbracket 0, n \rrbracket$, which satisfies conditions \eqref{e:GeorgeHarrison} and \eqref{e:RingoStarr}. By the convergence of $\{\zeta^n\}_{n}$ to $\xi$, we can assume that $|\xi(t)-\zeta^n(t)|\leq d/4$ for all $t \in [\tau_0,\tau_1]$, by restricting our attention to large enough values of $n$. 
Then, $\mathcal{O}_{(t_*,x_*)}$ lies entirely to the right from the curve $x=\xi(t)+d/4$, because we have chosen $\varepsilon < d/16 < 3d/4 = d-d/4$. 
Now, we define $J_0$ as the smallest value $i \in \llbracket 0 , n \rrbracket$ such that the path $x_{J_0}$ crosses $\mathcal{O}_{(t_*,x_*)}$. 
Two cases may occur depending on the behaviour of the particle $x_{J_0-1}$, see \figurename~\ref{f:rectangles}.

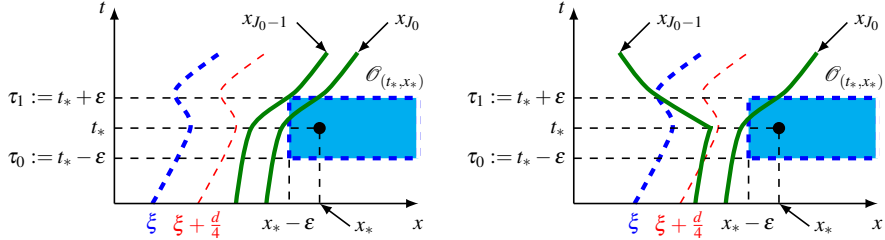
\begin{figure}[!htbp]
\centering\resizebox{!}{!}{
\begin{tikzpicture}[x=20mm, y=20mm, semithick]

\draw [blue, dashed, ultra thick] plot [smooth]	coordinates { (.25,0) (.5,.5) (.4,.75) (.7,1) };
\node[blue,below] at (.25,0) {\strut$\xi$};
\begin{scope}[shift={(.3,0)}]
\draw [red, dashed] plot [smooth]	coordinates { (.25,0) (.5,.5) (.4,.75) (.7,1) };
\node[red,below] at (.25,0) {\strut$\xi+\frac{d}{4}$};
\end{scope}

\draw[ultra thick,dashed,blue,fill=cyan] (1.15,.3) rectangle (2,.7);
\draw[ultra thick,white] (2,.3) -- (2,.7);
\draw[dashed] (0,.3) node[left] {\strut$\tau_0:=t_*-\varepsilon$} -- (1.15,.3) -- (1.15,0) node[below] {\strut$x_*-\varepsilon$};
\draw[dashed] (0,.5) node[left] {\strut$t_*$} -- (1.35,.5) -- (1.35,0);
\draw[latex-] (1.35,0) -- (1.55,-.15);
\node[right] at (1.5,-.15) {\strut\ $x_*$};
\draw[dashed] (0,.7) node[left] {\strut$\tau_1:=t_*+\varepsilon$} -- (1.15,.7);
\node[above left] at (2.1,.7) {\strut$\mathcal{O}_{(t_*,x_*)}$};
\draw[black, fill=black] (1.35,.5) circle (2pt);

\draw [green!50!black, ultra thick] plot [smooth]	coordinates { (1,0) (1.1,.5) (1.4,.75) (1.6,1) };
\draw[latex-] (1.6,1) -- (1.8,1.2) node[right] {$x_{J_0}$};
\begin{scope}[shift={(-.2,0)}]
\draw [green!50!black, ultra thick] plot [smooth]	coordinates { (1,0) (1.1,.5) (1.4,.75) (1.6,1) };
\draw[latex-] (1.6,1) -- (1.4,1.2) node[left] {$x_{J_0-1}$};
\end{scope}

\draw[-latex] (0,0) -- (2,0) node[below] {\strut$x$};
\draw[-latex] (0,0) -- (0,1.3) node[left] {$t$};

\end{tikzpicture}
\quad
\begin{tikzpicture}[x=20mm, y=20mm, semithick]

\begin{scope}[shift={(.15,0)}]
\draw [blue, dashed, ultra thick] plot [smooth]	coordinates { (.25,0) (.5,.5) (.4,.75) (.7,1) };
\node[blue,below] at (.25,0) {\strut$\xi$};
\begin{scope}[shift={(.3,0)}]
\draw [red, dashed] plot [smooth]	coordinates { (.25,0) (.5,.5) (.4,.75) (.7,1) };
\node[red,below] at (.25,0) {\strut$\xi+\frac{d}{4}$};
\end{scope}
\end{scope}

\draw[ultra thick,dashed,blue,fill=cyan] (1.15,.3) rectangle (2,.7);
\draw[ultra thick,white] (2,.3) -- (2,.7);
\draw[dashed] (0,.3) node[left] {\strut$\tau_0:=t_*-\varepsilon$} -- (1.15,.3) -- (1.15,0) node[below] {\strut$x_*-\varepsilon$};
\draw[dashed] (0,.5) node[left] {\strut$t_*$} -- (1.35,.5) -- (1.35,0);
\draw[latex-] (1.35,0) -- (1.55,-.15);
\node[right] at (1.5,-.15) {\strut\ $x_*$};
\draw[dashed] (0,.7) node[left] {\strut$\tau_1:=t_*+\varepsilon$} -- (1.15,.7);
\node[above left] at (2.1,.7) {\strut$\mathcal{O}_{(t_*,x_*)}$};
\draw[black, fill=black] (1.35,.5) circle (2pt);

\draw [green!50!black, ultra thick] plot [smooth]	coordinates { (1,0) (1.1,.5) (1.4,.75) (1.6,1) };
\draw[latex-] (1.6,1) -- (1.8,1.2) node[right] {$x_{J_0}$};
\begin{scope}[shift={(-.2,0)}]
\begin{scope}
\clip (.9,0) rectangle (1.111,.511);
\draw [green!50!black, ultra thick] plot [smooth]	coordinates { (1,0) (1.1,.5) (1.4,1)};
\end{scope}
\draw [green!50!black, ultra thick] plot [smooth]	coordinates { (1.1,.5) (.7,.75) (.5,1) };
\draw[latex-] (.5,1) -- (.7,1.2) node[right] {$x_{J_0-1}$};
\end{scope}

\draw[-latex] (0,0) -- (2,0) node[below] {\strut$x$};
\draw[-latex] (0,0) -- (0,1.3) node[left] {$t$};

\end{tikzpicture}}

\caption{With reference to the proof of Proposition~\ref{prop:rectangle_}, on the left the case described in Case~1, on the right the case described in Case~2.}
\label{f:rectangles}
\end{figure}

\begin{enumerate}[label={\bf{Case} $\pmb{\arabic*}$},wide=0pt]
\item[\bf{Case}~$\pmb{1}$]
Suppose that the path of $x_{J_0-1}$ for $t \in [\tau_0,\tau_1]$ lies to the right from the curve $x=\xi(t)+d/4$ or that $J_0=0$. 
In this case, for $t\geq \tau_0$ we consider the particle system $\tilde{x}_0,\ldots,\tilde{x}_n$ evolving according to the follow-the-leader system of the form \eqref{e:FTL-basic} but with shifted initial time:
\begin{equation*}
\begin{cases}
\dot{\tilde{x}}_n =v_{\max},
\\
\dot{\tilde{x}}_i (t)=v\biggl(\frac{\ell}{\tilde{x}_{i+1}(t)-\tilde{x}_i(t)}\biggr), & i \in \llbracket 0 , n-1 \rrbracket,
\\
\tilde{x}_i(\tau_0)= x_i(\tau_0), & i \in \llbracket 0 , n \rrbracket.
\end{cases}
\end{equation*}
We highlight that the evolution of $\tilde{x}_i$, $i \in \llbracket 0 , J_0-2 \rrbracket$, does not influence the evolution of particles $\tilde{x}_i$, $i \in \llbracket J_0-1 , n \rrbracket$.
Thus, for $i\geq J_0-1$, we have $\tilde{x}_i(t)\equiv x_i(t)$ for all $t\in [\tau_0,\tau_1]$. Moreover, by the definition of $J_0$, $\mathcal{O}_{(t_*,x_*)}$ lies entirely to the right from the path of $x_{J_0-1}$, therefore within $\mathcal{O}_{(t_*,x_*)}$ we have that the density $\tilde{\rho}^{n}$ defined in \eqref{tildedisdens} coincides with the density $\rho^n$ defined in \eqref{e:disdens} for the original particle system.
Hence, in this case it holds $\left\|\tilde\rho^n-\rho^n\right\|_{\L1(\mathcal{O}_{(t_*,x_*)})}=0$.
Moreover, for $t\in[\tau_0,\tau_1]$ the variations of $\tilde{\rho}^n(t,\cdot\,)$ and $\rho^n(t,\cdot\,)$ over $[x_*-\varepsilon,\infty)$ coincide, thus both \eqref{e:GeorgeHarrison} and \eqref{e:RingoStarr}  hold true.

\item[\bf{Case}~$\pmb{2}$]
Suppose that there exists $s \in [\tau_0,\tau_1]$ such that $x_{J_0-1}(s)\leq \xi(s)+d/4$. 
In this case, for $t\geq \tau_0$ we consider the follow-the-leader particle system involving particles $\tilde{x}_{0},\dots,\tilde{x}_n$ and given by the following analogue of \eqref{e:FTL-basic}
\begin{equation*}
\begin{cases}
\dot{\tilde{x}}_n =v_{\max},
\\
\dot{\tilde{x}}_i (t)=v\biggl(\frac{\ell}{\tilde{x}_{i+1}(t)-\tilde{x}_i(t)}\biggr), & i \in \llbracket 0 , n-1 \rrbracket,
\\
\tilde{x}_i(\tau_0)= \overline{\tilde{x}}_i, & i \in \llbracket 0 , J_0-1 \rrbracket,
\\
\tilde{x}_i(\tau_0)= x_i(\tau_0), & i \in \llbracket J_0 , n \rrbracket,
\end{cases}
\end{equation*}
where $\overline{\tilde{x}}_i$, $i\in \llbracket 0 , J_0-1 \rrbracket$, are chosen so that $\overline{\tilde{x}}_{i+1} - \tilde{\overline{x}}_i \geq \ell/\rho_{\max}$ and $x_{J_0}(\tau_0)-\tilde{\overline{x}}_{J_0-1} \geq 3 \, \varepsilon \, v_{\max}$.
Note that we also ensured that $\|\tilde{\rho}^n(\,\cdot,t)\|_{\L1(\mathbb{R})}=L$ for any $t\in [\tau_0,\tau_1]$, with $\tilde{\rho}^n$ defined as in \eqref{tildedisdens}. As in Case~1, the evolution of $\tilde{x}_i$ (now for $i \in \llbracket 0 , J_0-1 \rrbracket$) does not influence the evolution of particles $\tilde{x}_i$,  for $i \in \llbracket J_0, n \rrbracket$; for $i\geq J_0$, we have $\tilde{x}_i(t)\equiv x_i(t)$ for all $t\in [\tau_0,\tau_1]$, but note that the path of $x_{J_0-1}(t)$ differs from the path of $\tilde{x}_{J_0-1}(t)$. 
Since  $\mathcal{O}_{(t_*,x_*)}$ lies entirely to the right from the path of $x_{J_0-1}$, the density $\tilde\rho^n$ differs from $\rho^n$ only in 
$$\mathcal{O}_{(t_*,x_*)}\cap \left\{(t,x) : t\in (\tau_0,\tau_1),\, x_{J_0-1}(t)\leq x \leq x_{J_0}(t)\right\},$$ 

in particular they differ over $[x_*-\varepsilon,\infty)$ only by their leftmost states. Since all states take values in $[0,\rho_{\max}]$, \eqref{e:RingoStarr} holds true.
It remains to show that \eqref{e:GeorgeHarrison} holds true, i.e., the difference between $\tilde{\rho}^n$ and $\rho^n$ is negligible in the $\L1$ norm on $\mathcal{O}_{(t_*,x_*)}$. 
Focusing first on $\rho^n$, we notice that, for all $t \in[\tau_0,\tau_1]$, there holds 
\begin{align*}
x_{J_0-1}(t)-\xi(t)& = x_{J_0-1}(s)-\xi(s)+ \int_{s}^t \left(\dot{x}_{J_0-1}(\tau)-\dot{\xi}(\tau)\right)\, {\rm{d}} \tau\\
& \leq \frac{d}{4} + |t-s| \cdot \|\dot{x}_{J_0-1}-\dot{\xi}\|_\infty 
\leq \frac{d}{4}+ 2\varepsilon \left(v_{\max}+\mathcal{L}_{\xi}\right).
\end{align*}
On the other hand, notice that the path of $x_{J_0}$ crosses $\mathcal{O}_{(t_*,x_*)}$, which lies at distance $\geq 3d/4$ to the right from the curve $\xi$ for $t \in [\tau_0,\tau_1]$, hence an analogous argument yields 
$$
x_{J_0}(t)-\xi(t)\geq \frac{3d}{4}-2\varepsilon\left(v_{\max}+\mathcal{L}_{\xi}\right) \quad \text{for all } t \in[\tau_0,\tau_1]. 
$$
As a consequence, for all such $t$, by the choice of $\varepsilon$ it holds
$$ 
x_{J_0}(t)-x_{J_0-1}(t)\geq \frac{d}{2}-4\varepsilon\left(v_{\max}+\mathcal{L}_{\xi}\right) > \frac{d}{4},
$$	
hence, for $(t,x) \in \mathcal{O}_{(t_*,x_*)}$ such that $x<x_{J_0}(t)$, we have
\begin{equation}
\label{e:estrhon}
\rho^n(t,x)=\frac{\ell}{x_{J_0}(t)-x_{J_0-1}(t)} < \frac{4\ell}{d}.
\end{equation} 
Turning to $\tilde{\rho}^n$, by the choice of $\overline{\tilde{x}}_{J_0-1}$ we have that
\begin{align*}
\tilde{x}_{J_0}(t)-\tilde{x}_{J_0-1}(t)&= x_{J_0}(\tau_0)-\overline{\tilde{x}}_{J_0-1}+\int_{\tau_0}^t \left( \dot{\tilde{x}}_{J_0}(\tau) - \dot{\tilde{x}}_{J_0-1}(\tau) \right) \, {\rm{d}} \tau
\\&\geq 3 \, \varepsilon \, v_{\max}- 2 \, \varepsilon \, v_{\max} = \varepsilon \, v_{\max}.
\end{align*}
Hence, for $(t,x) \in \mathcal{O}_{(t_*,x_*)}$ such that $x<x_{J_0}(t)$, we have that
\begin{equation}
\label{e:estrhon2}
\tilde\rho^n(t,x)=\frac{\ell}{\tilde{x}_{J_0}(t)-\tilde {x}_{J_0-1}(t)}\leq \frac{\ell}{\varepsilon v_{\max}}.
\end{equation}
Furthermore, by construction $|\tilde{x}_{J_0}(t)-(x_*-\varepsilon)| \leq 1+v_{\max} \cdot 2\varepsilon - x_* + \varepsilon$.
This inequality, together with \eqref{e:estrhon} and \eqref{e:estrhon2}, ensures that \[\left\|\tilde\rho^n-\rho^n\right\|_{\L1(\mathcal{O}_{(t_*,x_*)})}
\leq 2 \, \varepsilon \, \left(1+2\,\varepsilon\,v_{\max} - x_* + \varepsilon\right) \, \left( \frac{4}{d} + \frac{1}{\varepsilon\,v_{\max}}\right) \, \ell,\]
where we recall that $d$ and $\varepsilon$ are independent on $n$, while $\ell\to 0$ as $n\to \infty$.
Hence \eqref{e:GeorgeHarrison} holds true and this concludes the proof of {\bf (i)} and {\bf (ii)}.
\end{enumerate}

It remains to prove {\bf (iii)}. Let us fix $T>0$ and define $\beta:=16\left(v_{\max}+\mathcal{L}_{\xi}+1\right)$. Let us arbitrarily choose $\varepsilon>0$, and then take $d:=\beta\varepsilon$. With this choice, one can fix any $t_*\in [\varepsilon,T]$ and take $x_*:=\xi(t_*) +d$ so that $[t_*-\varepsilon,t_*+\varepsilon]\times (x_*,\infty)\subset \mathcal{O}_{(t_*,x_*)}$.
The above construction of $\tilde{\rho}^n$ can be used in the so defined $\mathcal{O}_{(t_*,x_*)}$: in particular, the estimate \eqref{e:EricClapton} is valid for $\tilde \rho^n(t_*,\cdot\,)$ because $t_*-\tau_0=t_*-(t_*-\varepsilon)=\varepsilon$, while due to \eqref{e:RingoStarr} we infer the bound of the form  $\rho_{\max}+C_1+C_2/\varepsilon$ on the space variation of $\rho^n(t_*,\cdot\,)$ over $[x_*,\infty)$, with $C_1, C_2>0$ independent on $t_*$ and $n$. This concludes the proof of {\bf (iii)}.\qed
\end{proof} 

The previous proposition yields to the compactness of $\{\rho^n\}_{n}$ on $(0,\infty) \times \mathbb{R}$. Indeed, we have
\begin{proposition}
\label{p:compact}
Let $\xi$ be the limit, in the topology of the locally uniform convergence, of $\{\xi^n\}_{n}$. 
Then, up to a subsequence, $\{\rho^n\}_{n}$ converges a.e.\ on $(0,\infty) \times \mathbb{R}$. 
Moreover, the limit $\rho$ satisfies the entropy inequality~\eqref{e:entro-bis} on the open domain $\mathcal{O}$, i.e. with 
test functions $\varphi$ supported in $\mathcal{O}$.
\end{proposition} 

\begin{proof}
The proof of the statement relies on {\bf (i)} of Proposition~\ref{prop:rectangle_} and on a covering argument, which yields global compactness in the sense of the a.e.\ convergence. 
Indeed, $\mathcal{O}$ defined as in \eqref{e:O} can be represented as the countable union of neighborhoods of the form $\mathcal{O}_{(t_*,x_*)}$ as in \eqref{eq:neighb_}, where, up to a subsequence, $\{\rho^n\}_n$ converges  a.e.\ by Proposition~\ref{prop:rectangle_}.
As a consequence, $\mathcal{O}$ is covered by a countable number of sets on which, up to a subsequence, $\{\rho^n\}_n$ converges  a.e.\ and a diagonal extraction argument permits to conclude that $\{\rho^n\}_n$ converges to some limit $\rho$  a.e.\ on $\mathcal{O}$, whence it converges   a.e.\ on $\mathbb{R}\times(0,\infty)$. Finally, the limit $\rho$ is an entropy solution because by Proposition~\ref{prop:rectangle_} it is so on each $\mathcal{O}_{(t_*,x_*)}$.\qed
\end{proof}

\subsection{Consistency of the approximation}

Our final goal is to show that the limit $(\rho,\xi)$ is indeed an entropy solution for the initial-boundary value problem \eqref{e:model} in the sense of Definition~\ref{d:entro}. We first notice that condition \eqref{e:cost0} holds for a.e.\ $t>0$ due to Proposition~\ref{p:inir}.  We also recall that we do not need to take care of boundary conditions thanks to Proposition~\ref{p:refo}. Moreover, according to Proposition~\ref{p:compact}, local entropy inequalities in $\mathcal{O}$ are satisfied by $\rho$: this is enough to ensure that Proposition~\ref{p:ini} concerning the initial datum holds true. Therefore, in order to prove the consistency claim we only need to assess \eqref{e:entro-bis} with test functions supported in $(0,\infty)\times \mathbb{R}$.

\begin{proposition}
The limit $(\rho,\xi)$ of the subsequence $\{(\rho^{n},\xi^{n})\}_{n}$ verifies the entropy condition \eqref{e:entro-bis} of Definition~\ref{d:entro-bis} with the choice of test functions restricted to $\Cc \infty\left((0, \infty) \times \mathbb{R}; [0, \infty)\right)$.
\end{proposition}
\begin{proof}
We  revisit 
the proof of Proposition~\ref{p:entropy-local}  
in the context of the local $\mathbf{BV}$ spatial bound of Proposition~\ref{prop:rectangle_} in the place of the global $\mathbf{BV}$ in space bound \eqref{e:DavidMaximMicic}.
The idea is to adapt the choice of test functions using a density argument.
First, due to $\L\infty$ boundedness of $\rho^n$, we can work with ${\bf W^{1,\pmb \infty}}((0,\infty)\times \mathbb{R})$ test functions in the place of $\Cc\infty$ ones. Second, observe that any smooth compactly supported test function $\varphi$ is the limit, in the ${\bf W^{1,1}}((0,\infty)\times\mathbb{R})$ topology, of the sequence $\varphi_\sigma$ of ${\bf W^{1,\pmb \infty}}$ functions defined by
\begin{equation*}
\varphi_\sigma(t,x):=\begin{cases} \varphi(t,x+\sigma) & \text{if~} x\leq \xi(t)-\sigma,
\\
\varphi(t,\xi(t)) & \text{if~ } \xi(t)-\sigma< x< \xi(t)+\sigma,
\\
\varphi(t,x-\sigma) & \text{if~ } x\geq \xi(t)+\sigma.
\end{cases}
\end{equation*} 
Indeed,  $\varphi_\sigma$, $\partial_t\varphi_\sigma$ and $\partial_x\varphi_\sigma$ are uniformly bounded by the $\C1$ norm of $\varphi$ and they converge a.e. to $\varphi$, $\varphi_t$ and $\varphi_x$, respectively, as $\sigma\to 0$, due to the uniform continuity of $\varphi$, $\varphi_t$ and $\varphi_x$. In conclusion, if we prove \eqref{e:entro-bis} with test function $\varphi_\sigma$ for all sufficiently small $\sigma$, at the limit we obtain \eqref{e:entro-bis} with test function $\varphi$. By this argument, for the sequel we can assume without loss of generality that there exists $\sigma>0$ such that
\begin{equation}\label{eq:phisigma-prop}
\text{$\partial_x \varphi(t,x)=0$ for all $(t,x)$ with $t>0$, $x\in (\xi(t)-\sigma,\xi(t)+\sigma).$}
\end{equation}

We then proceed as for the proof of Proposition~\ref{p:entropy-local}. By the $\L1$-convergence of $\{\rho^{n}\}_{n}$ to $\rho$,  using Corollary~\ref{p:Korn} together with the equality \[\int_{\mathbb{R}} \bigl|\sign\bigl(x-\xi(t)\bigr)-\sign\bigl(x-\zeta^{n}(t)\bigr)\bigr| \, {\rm{d}} x = 2 |\xi(t)-\zeta^{n}(t)|,\]
it is sufficient to prove that, for a fixed $\sigma>0$ and for all $\kappa \in [0,\rho_{\max}]$, the limit as $n$ goes to infinity of
\[(\spadesuit^n) :=
\int_0^{ \infty} \int_{\mathbb{R}} \bigl( | \rho^{n}-\kappa| \, \partial_t\varphi + \Phi(t,x, \rho^{n},\kappa,\zeta^n) \, \partial_x\varphi \bigr) {\rm{d}} x \, {\rm{d}} t 
+ 2 \int_0^{ \infty} f(\kappa) \varphi(t,\zeta^{n}) \, {\rm{d}} t\]
is non-negative,  where the test function $\varphi\in {\bf W^{1,1}}((0,\infty)\times\mathbb{R})$ is supported in $[\tau,T]\times \mathbb{R}$ for some $0<\tau<T$, and $\varphi$ satisfies \eqref{eq:phisigma-prop}.

We first notice that Corollary \ref{p:Korn} implies that $|\xi(t)-\zeta^n(t)|<\sigma$ for all $t\in (0,T)$, upon  restricting our attention to large enough values of $n$. 
Further, we can choose $\varepsilon>0$ small enough such that $\varepsilon\leq \tau$ and $\beta\varepsilon\leq \sigma$, where $\beta:=16(v_{\max}+\mathcal{L}_{\xi}+1)$ is the same as in Proposition~\ref{prop:rectangle_}. As a consequence, for all $n$ we have
\begin{equation}\label{eq:TVeps}
\TV\bigl({\rho}^n(t,\cdot\,);\mathbb{R}\setminus (\xi(t)-\sigma,\xi(t)+\sigma)\bigr)\leq \textbf{TV}_{\varepsilon,T} \quad \text{for all $t\in [\varepsilon,T]$} 
\end{equation}	 
where $\textbf{TV}_{\varepsilon,T}$ is independent of $t\in [\epsilon,T]$ and $n$.

We are intended to demonstrate that the localized variation bound \eqref{eq:TVeps} combined with the property \eqref{eq:phisigma-prop} of the test function can replace the global variation bound \eqref{e:DavidMaximMicic} in the context of the proof of Proposition~\ref{p:entropy-local}.

Let us assume the existence of two indices $J_0, K_0\in \llbracket 0,n-1 \rrbracket$ such that $J_0\leq I_0^h< I_0^h+1\leq  K_0$ and $x_i(t)\in [\xi(t)-\sigma,\xi(t)+\sigma]$ for $i\in \llbracket J_0, K_0 \rrbracket$. Note that $J_0$ and $K_0$ vary with $n$ and with $t$. 
This case is not generic but the other cases, that is when there is less of two particles in the interval $[\xi(t)-\sigma,\xi(t)+\sigma]$, are similar and are therefore omitted.

With the same computation as in the proof of Proposition~\ref{p:entropy-local}, we can show that $(\spadesuit^n)=\kappa(\spadesuit^n_{1})+(\spadesuit^n_{2})$, where $(\spadesuit^n_{1})$ and $(\spadesuit^n_{2})$ are given respectively by \eqref{e:spade_1} and \eqref{e:spade_2} with $\varphi$ satisfying \eqref{eq:phisigma-prop}. 

We first prove that $(\spadesuit^n_{2})\to 0$ as $n\to \infty$. For this purpose, we write the analogue of \eqref{e:ChickCorea}, \eqref{e:JohnColtrane} and \eqref{e:MilesDavis}. 
With the same notation introduced in \eqref{e:notations}, by the above definition of the indices $J_0$ and $K_0$, which replace the index $I_0$ in the proof of Proposition~\ref{p:entropy-local}, and due to \eqref{eq:phisigma-prop}, the estimate \eqref{e:ChickCorea} becomes
\begin{align*}
&\left| - \sum_{i=0}^{I_0^h-1} \dot{R}_{i+\frac{1}{2}} \int_{x_{i}}^{x_{i+1}} \varphi(t,x) \, {\rm{d}} x
+ \sum_{i=0}^{I_0^h-1} R_{i+\frac{1}{2}}  \, \Delta_v(R_{i+\frac{1}{2}},R_{i-\frac{1}{2}}) \, \varphi(t,x_{i})  \right|
\\={}&
\Biggl| \sum_{i=0}^{J_0-2} R_{i+\frac{1}{2}}  \, \Delta_v(R_{i-\frac{1}{2}}-R_{i+\frac{1}{2}}) \, \fint_{x_{i}}^{x_{i+1}} \bigl( \varphi(t,x) - \varphi(t,x_{i}) \bigr) {\rm{d}} x 
\\
&+  R_{J_0-\frac{1}{2}} \frac{\Delta_v(R_{J_0-\frac{3}{2}},R_{J_0-\frac{1}{2}})}{x_{J_0}-x_{J_0-1}} \biggl(\int_{x_{J_0-1}}^{\xi-\sigma} \varphi(t,x) \, {\rm{d}} x+\int_{\xi-\sigma}^{x_{J_0}} \varphi(t,\xi) \, {\rm{d}} x\biggr)
\\
&-  R_{J_0-\frac{1}{2}}  \, \Delta_v(R_{J_0-\frac{3}{2}},R_{J_0-\frac{1}{2}}) \, \varphi(t,x_{J_0-1})
\Biggr|
\\ \leq {}&
\mathcal{L}_v \, \mathcal{L}_\varphi \left( \sum_{i=0}^{J_0-2} |R_{i-\frac{1}{2}}-R_{i+\frac{1}{2}}| \, R_{i+\frac{1}{2}} \fint_{x_{i}}^{x_{i+1}} ( x-x_{i} ) \, {\rm{d}} x \right)
\\&
+\Biggl|  R_{J_0-\frac{1}{2}}  \, \Delta_v(R_{J_0-\frac{3}{2}},R_{J_0-\frac{1}{2}}) \fint_{x_{J_0-1}}^{x_{J_0}} \bigl(\varphi(t,x)-\varphi(t,x_{J_0-1})\bigr) {\rm{d}} x \Biggr|
\\&
+ \Biggl| R_{J_0-\frac{1}{2}}  \frac{\Delta_v(R_{J_0-\frac{3}{2}},R_{J_0-\frac{1}{2}})}{x_{J_0}-x_{J_0-1}}\int_{\xi-\sigma}^{x_{J_0}} \bigl(\varphi(t,\xi)-\varphi(t,x)\bigr) {\rm{d}} x \Biggr|
\\
\leq {}&
\mathcal{L}_v \, \mathcal{L}_\varphi \left( \sum_{i=0}^{J_0-2} |R_{i-\frac{1}{2}}-R_{i+\frac{1}{2}}| \, R_{i+\frac{1}{2}} \, \frac{x_{i+1}-x_{i}}{2}\right)
+ \mathcal{L}_v \, \mathcal{L}_\varphi |R_{J_0-\frac{3}{2}}-R_{J_0-\frac{1}{2}}| \, R_{J_0-\frac{1}{2}} \, \frac{x_{J_0}-x_{J_0-1}}{2}
\\
&+ \mathcal{L}_v \, \mathcal{L}_\varphi |R_{J_0-\frac{3}{2}}-R_{J_0-\frac{1}{2}}|\, R_{J_0-\frac{1}{2}}\,  \frac{\bigl(x_{J_0}-(\xi-\sigma)\bigr)^2}{2(x_{J_0}-x_{J_0-1})}
\\
\leq{}&\frac{\mathcal{L}_v \, \mathcal{L}_\varphi}{2} \, \left( \sum_{i=0}^{J_0-2} |R_{i-\frac{1}{2}}-R_{i+\frac{1}{2}}| +2\, R_{\max}\right) \, \ell.
\end{align*}
\delaynewpage{3}
In the same way, \eqref{e:JohnColtrane} and \eqref{e:MilesDavis} become respectively
\begin{align*}
&\left| -\sum_{i=I_0^h+1}^{n-1} \dot{R}_{i+\frac{1}{2}} \int_{x_{i}}^{x_{i+1}} \varphi(t,x) \, {\rm{d}} x 
- \sum_{i=I_0^h+1}^{n-1}\!\!  R_{i+\frac{1}{2}} \Delta_v(R_{i+\frac{3}{2}},R_{i+\frac{1}{2}}) \bigr) \varphi(t,x_{i+1}) \right|
\\
\leq{}& \frac{\mathcal{L}_v \, \mathcal{L}_\varphi}{2} \, \left( \sum_{i=K_0+1}^{n-1} |R_{i+\frac{3}{2}}-R_{i+\frac{1}{2}}| +2 \, R_{\max}\right) \, \ell,
\\
&\Biggl|- \dot{R}_{I_0^h+\frac{1}{2}} \int_{x_{I_0^h}}^{x_{I_0^h+1}} \varphi (t,x) \, {\rm{d}} x 
- R_{I_0^h+\frac{1}{2}}  \, \Delta_v(R_{I_0^h+\frac{3}{2}},R_{I_0^h+\frac{1}{2}}) \, \varphi (t,x_{I_0^h+1}) 
\\&
- R_{I_0^h+\frac{1}{2}}  \, \Delta_v(R_{I_0^h-\frac{1}{2}},R_{I_0^h+\frac{1}{2}}) \, \varphi (t,x_{I_0^h})
-2 f(R_{I_0^h+\frac{1}{2}}) \varphi(t,\zeta^n)\Biggr|
\\={}&
\Biggl|R_{I_0^h+\frac{1}{2}} \bigl(v(R_{I_0^h+\frac{3}{2}})+v(R_{I_0^h-\frac{1}{2}})\bigr) \fint_{x_{I_0^h}}^{x_{I_0^h+1}} \varphi (t,x) \, {\rm{d}} x 
- R_{I_0^h+\frac{1}{2}}  \, \Delta_v(R_{I_0^h+\frac{3}{2}},R_{I_0^h+\frac{1}{2}}) \, \varphi (t,x_{I_0^h+1}) 
\\&
- R_{I_0^h+\frac{1}{2}}  \, \Delta_v(R_{I_0^h-\frac{1}{2}},R_{I_0^h+\frac{1}{2}}) \, \varphi (t,x_{I_0^h})
-2 f(R_{I_0^h+\frac{1}{2}}) \varphi(t,\zeta^n)\Biggr|
\\={}&
\Biggl|R_{I_0^h+\frac{1}{2}} \bigl(v(R_{I_0^h+\frac{3}{2}})+v(R_{I_0^h-\frac{1}{2}})\bigr) \varphi(t,\xi)
- R_{I_0^h+\frac{1}{2}}  \, \Delta_v(R_{I_0^h+\frac{3}{2}},R_{I_0^h+\frac{1}{2}}) \, \varphi(t,\xi)
\\&
- R_{I_0^h+\frac{1}{2}} \, \Delta_v(R_{I_0^h-\frac{1}{2}},R_{I_0^h+\frac{1}{2}}) \, \varphi(t,\xi)
-2 \, R_{I_0^h+\frac{1}{2}}\, v(R_{I_0^h+\frac{1}{2}}) \varphi(t,\xi)\Biggr|= 0.
\end{align*}
Therefore, by the definitions of $J_0$ and $K_0$ we get
\begin{align*}
|(\spadesuit^n_{2})|\!\leq\! \frac{\mathcal{L}_{v}\mathcal{L}_{\varphi}}{2} \biggl(&\!\int_{0}^{T} \!\!\!\!\!\TV\bigl(\rho^{n}(t,\cdot\,); \mathbb{R}\setminus (\xi(t)\!-\!\sigma, \xi(t)\!+\!\sigma) \bigr)   {\rm d} t + 2 R_{\max}   T \biggr)   \ell
\end{align*}
and hence, due to \eqref{eq:TVeps}, we have
\begin{equation*}
|(\spadesuit^n_{2})|\leq \frac{\mathcal{L}_{v}\, \mathcal{L}_{\varphi}}{2}  T\bigl({\bf TV_{\varepsilon,T}} + 2 R_{\max}  \bigr)   \ell,
\end{equation*}
where the right-hand side converges to zero as $n\to \infty$ because $\ell$ do so, while ${\bf TV_{\varepsilon}}$ is independent on $n$.
As a consequence, to conclude the proof it is sufficient to show that $(\spadesuit^n_{1})\geq 0$. As we did for $(\spadesuit^n_{2})$, we write the analogue of \eqref{e:milan}, \eqref{e:roma} and \eqref{e:lazio}. By the above definitions of the indices $J_0$ and $K_0$ and due to \eqref{eq:phisigma-prop}, we have that \eqref{e:milan}, \eqref{e:roma} and \eqref{e:lazio} become respectively
\begin{align*}
&- \sum_{i=0}^{I_0^h-1}  \sign(R_{i+\frac{1}{2}}-\kappa)  \, \Delta_v(\kappa,R_{i-\frac{1}{2}}) \, \varphi(t,x_{i})
+\sum_{i=-1}^{I_0^h-1} \sign(R_{i+\frac{1}{2}}-\kappa)  \, \Delta_v(\kappa,R_{i+\frac{1}{2}}) \, \varphi(t,x_{i+1})
\\&
-  \sign(R_{I_0^h+\frac{1}{2}}-\kappa) \, \Delta_v(\kappa,R_{I_0^h-\frac{1}{2}}) \, \varphi(t,x_{I_0^h})
\\={}&
\sum_{i=0}^{J_0-1} \bigl( \sign(R_{i-\frac{1}{2}}-\kappa) - \sign(R_{i+\frac{1}{2}}-\kappa) \bigr) \, \Delta_v(\kappa,R_{i-\frac{1}{2}}) \, \varphi(t,x_{i}) 
\geq 0,
\\[5pt]
&\sum_{i=I_0^h+1}^{n} \sign(R_{i+\frac{1}{2}}-\kappa) \, \Delta_v(\kappa,R_{i+\frac{1}{2}}) \, \varphi(t,x_{i})
- \sum_{i=I_0^h+1}^{n-1}  \sign(R_{i+\frac{1}{2}}-\kappa) \, \Delta_v(\kappa,R_{i+\frac{3}{2}}) \, \varphi(t,x_{i+1})
\\&
- \sign(R_{I_0^h+\frac{1}{2}}-\kappa) \, \Delta_v(\kappa,R_{I_0^h+\frac{3}{2}}) \, \varphi(t,x_{I_0^h+1})
\\={}&
\sum_{i=K_0+1}^{n} \bigl(\sign(R_{i+\frac{1}{2}}-\kappa)-\sign(R_{i-\frac{1}{2}}-\kappa)\bigl) \, \Delta_v(\kappa,R_{i+\frac{1}{2}}) \, \varphi(t,x_{i}) 
\geq 0,
\\[5pt]
&\hspace{-0.3cm}2 \bigl(1 + \sign(R_{I_0^h+\frac{1}{2}}\!-\!\kappa)\bigl) v(\kappa) \varphi(t,\zeta^{n})
=2 \bigl(1 + \sign(R_{I_0^h+\frac{1}{2}}\!-\!\kappa)\bigl) v(\kappa) \varphi(t,\xi)\geq 0,
\end{align*}
then we have $(\spadesuit^n_{1})\geq 0$ and hence the proof is completed.\qed 
\end{proof}

\section{A numerical case study}
\label{s:numsche}

\delaynewpage{1}
In this section we consider the simplest affine choice of $v$:
\begin{align*}
v( \rho ) &:= v_{\max} \left(1-\frac{\rho}{ \rho_{\max}}\right),&
f( \rho ) &:= \rho \,v( \rho ),&
v_+( \rho ) &:= [v( \rho )]_+.
\end{align*}
We propose the following fully discrete version of the many-particle Hughes model of Section~\ref{s:mpa}.
Let $\mathtt{T}$ be the target time horizon.
Fix $\Delta t>0$ sufficiently small and $n \in{\mathbb{N}}$ sufficiently large.
Introduce the following notation
\begin{align*}
t^h &:= h\,\Delta t,&
\mathtt{x}_i^h&:\approx x_{i}(t^h),&
&h \in \llbracket0,\mathtt{H}\rrbracket,&
&i \in\llbracket0,n\rrbracket,
\end{align*}
where $\mathtt{H} := \lceil\mathtt{T}/\Delta t\rceil$.
The approximation $\mathtt{x}_i^h$ of the position for the $i$-th particle at time $t^h$ is obtained by applying the following numerical scheme
\begin{align*}
\mathtt{x}_i^0 &:= \overline{x}_i,&\hspace{-30mm}i \in\llbracket0,n\rrbracket,
\\
\mathtt{x}_0^{h+1} &:=\mathtt{x}_0^{h} - v_{\max} \, \Delta t = \overline{x}_0 - v_{\max} \, t^{h+1},
\\
\mathtt{x}_i^{h+1} &:= \left\{\begin{array}{@{}l@{\quad}l@{}}
\mathtt{x}_i^{h} - v_+\Bigl(\frac{\ell}{\mathtt{x}_i^h-\mathtt{x}_{i-1}^h}\Bigr) \, \Delta t&\hbox{if }
\begin{aligned}[t]
\frac{2}{\alpha\,\ell} \, \mathtt{x}_i^{h} < \ 
&\# \{ j \in \llbracket0,n\rrbracket : \mathtt{x}_i^h < \mathtt{x}_j^{h} < 1\}
\\&- \# \{ j \in \llbracket0,n\rrbracket : -1 < \mathtt{x}_j^{h} < \mathtt{x}_i^h\},
\end{aligned}
\\
\mathtt{x}_i^{h} +v_+\Bigl(\frac{\ell}{\mathtt{x}_{i+1}^h-\mathtt{x}_{i}^h}\Bigr) \, \Delta t&\displaystyle
\hbox{otherwise},
\end{array}\right.
\\&&\hspace{-30mm}i \in\llbracket1,n-1\rrbracket,
\\
\mathtt{x}_n^{h+1} &:=\mathtt{x}_n^{h} + v_{\max} \, \Delta t = \overline{x}_n + v_{\max} \, t^{h+1},
\end{align*}
where $\overline{x}_i$ are defined in \eqref{e:Slipknot}.
We stress that the above scheme for the particles position $\mathtt{x}_i^h$ is decoupled from the scheme for the turning point $\xi^h$.
This choice allows for faster simulations, at least if we are interested only in the particles path, which is the case if, for instance, we are interested in computing the evacuation time.

In the following proposition we deduce a sort of CFL condition by requiring the order preservation of the particles. This stability result is the key step towards the convergence analysis of the scheme. However, in order not to overload the paper, we do not address the convergence, as $\Delta t\to 0$, of the scheme towards the unique solution of the model of Section~\ref{s:mpa}.

\begin{proposition}
If $\Delta t$ and $n$ are such that
\[\Delta t \leq \frac{L}{ \rho_{\max} \, v_{\max} \, n},\]
then for any $h \in\llbracket0,\mathtt{H}-1\rrbracket$ and $i \in\llbracket0,n-1\rrbracket$ we have
\begin{equation}
\mathtt{x}_{i+1}^{h} - \mathtt{x}_{i}^{h} \geq \ell/ \rho_{\max}
\quad\Longrightarrow\quad
\mathtt{x}_{i+1}^{h+1} - \mathtt{x}_{i}^{h+1} \geq \ell/ \rho_{\max}.
\label{e:Disturbed}
\end{equation}
\end{proposition}

\begin{proof} 
If $\mathtt{x}_{i+1}^{h}$ is moving backward and $\mathtt{x}_{i+1}^{h} - \mathtt{x}_{i}^{h} \geq \ell/ \rho_{\max}$, then we have
\begin{align*}
\mathtt{x}_{i+1}^{h+1} - \mathtt{x}_{i}^{h+1}
&=
\mathtt{x}_{i+1}^{h} - \mathtt{x}_{i}^{h} - v_+\left(\frac{\ell}{\mathtt{x}_{i+1}^h-\mathtt{x}_{i}^h}\right) \, \Delta t + v_+\left(\frac{\ell}{\mathtt{x}_i^h-\mathtt{x}_{i-1}^h}\right) \, \Delta t
\\&\geq
\mathtt{x}_{i+1}^{h} - \mathtt{x}_{i}^{h} - v\left(\frac{\ell}{\mathtt{x}_{i+1}^h-\mathtt{x}_{i}^h}\right) \, \Delta t,
\end{align*}
and therefore, using in particular the affine choice of $v$, we have
\begin{align*}
\mathtt{x}_{i+1}^{h+1} - \mathtt{x}_{i}^{h+1} \geq \frac{\ell}{ \rho_{\max}}
&\Longleftarrow
\mathtt{x}_{i+1}^{h} - \mathtt{x}_{i}^{h} - v\left(\frac{\ell}{\mathtt{x}_{i+1}^h-\mathtt{x}_{i}^h}\right) \, \Delta t
\geq \frac{\ell}{ \rho_{\max}}
\\\Longleftrightarrow
\Delta t 
\leq \frac{\mathtt{x}_{i+1}^{h} - \mathtt{x}_{i}^{h} - \frac{\ell}{ \rho_{\max}}}{v\left(\ell/(\mathtt{x}_{i+1}^h-\mathtt{x}_{i}^h)\right)}
= \frac{\mathtt{x}_{i+1}^{h} - \mathtt{x}_{i}^{h}}{v_{\max}}
&\Longleftarrow
\Delta t \leq \frac{\ell}{ \rho_{\max} \, v_{\max}} = \frac{L}{ \rho_{\max} \, v_{\max} \, n}.
\end{align*}
The remaining cases are analogous and are therefore omitted.\qed
\end{proof} 
\noindent
Notice that for $h=0$ the assumption in \eqref{e:Disturbed} holds true for any $i \in\llbracket0,n-1\rrbracket$ by \eqref{e:Marillion}.
\medskip

Let us recall the physical meaning of the parameter $\alpha\geq0$.
The case $\alpha=0$ corresponds to pedestrians moving towards the closest exit, regardless of the overall density distribution.
This is a typical behaviour in panic situations, see \cite{MR3076426} and the references therein.
On the other hand, as $\alpha$ grows, so does the importance of dividing the crowd into two groups with the same number of pedestrians.
Such mechanism penalizes regions with high number of pedestrians in the choice of the nearest exit.

\begin{figure}[!htbp]
\centering
\begin{tikzpicture}[x=.5mm, y=.5mm]
\node[anchor=south west,inner sep=0pt] at (0,0) { \includegraphics[height=25mm]{
		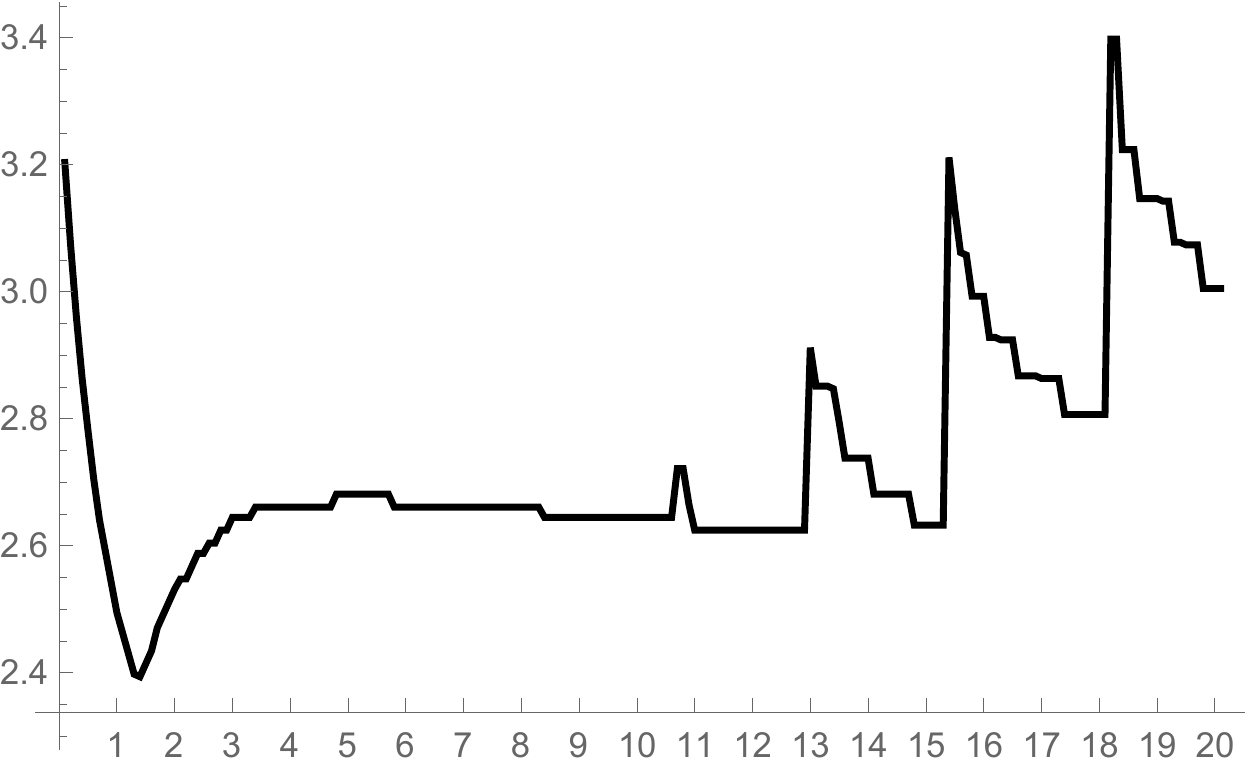}};
\node at (83,0) {\strut$\alpha$};
\node[left] at (0,50) {$T_{\rm mic}$};
\end{tikzpicture}
\caption{Microscopic evacuation time as a function of $\alpha$.}
\label{f:eva}
\end{figure}
These considerations lead us to study the microscopic evacuation time $T_{\rm mic}$ as a function of $\alpha$, i.e., $T_{\rm mic}=T_{\rm mic}(\alpha)$ (recall that the evacuation time of the model of Section~\ref{s:mpa} was proved to be finite, see Remark~\ref{rem:finite-time-evac}).
To this aim, we fix $v_{\max}=1$, $\rho_{\max}=1$, consider the initial datum
\[\overline{\rho}(x) :=
\left\{\begin{array}{@{}l@{\quad}l@{}}
0.9&\hbox{if }-1 \leq x<-0.5,
\\
0.9&\hbox{if }-0.4 \leq x<0,
\\
0&\hbox{otherwise},
\end{array}\right.
\]
and take $n=200$, $\ell = 0.00405$ and $\Delta t=L/( \rho_{\max} \, v_{\max} \, n) = 0.00405$.
By letting then $\alpha$ vary in $[0,20]$ with step $0.1$, we obtain the graph represented in \figurename~\ref{f:eva}.
Some comments on this figure are in order.
First, for $\alpha$ sufficiently small $T_{\rm mic}$ seems to have a Lipschitz dependence on $\alpha$.
On the other hand, there exist $\alpha_1<\alpha_2<\alpha_3<\ldots$ across which $T_{\rm mic}$ has discontinuities, whose amplitude grows for higher values of $\alpha$.
Second, we notice that $T_{\rm mic}$ has in $[0,\alpha_1]$ a graph analogous to that obtained, for instance, in \cite[\figurename s~5, 8]{MR3554543}, \cite[\figurename~9]{MR3460619}, \cite[\figurename~12]{MR3277564} and \cite[\figurename s~8, 9, 10]{ZHAO2020104517}, even if the variable and the setup for the simulations differ from our.
On the other hand, even if discontinuous evacuation times have been already considered in the literature, see for instance \cite[\figurename~9]{MR3554543}, \cite[\figurename~13]{MR2668282}, \cite[\figurename~6]{MR2817547}, to the best of authors' knowledge, the behaviour in \figurename~\ref{f:eva} (with more than one jump discontinuity) is new.
At last, the minimum in \figurename~\ref{f:eva} is achieved for $\alpha\approx1.3$ and the corresponding evacuation time is $T_{\rm mic}\approx2.39355$.

We further investigate this example by picturing some sample solutions of the scheme that we enrich with an approximation of the turning curve.
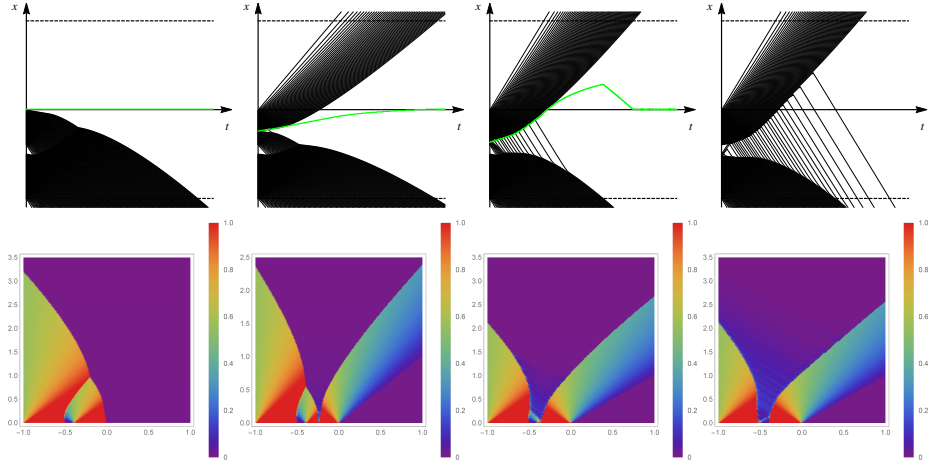
\begin{figure}[!htbp]\centering
\resizebox{\linewidth}{!}{
\begin{tikzpicture}[x=1mm, y=1mm]
\node[anchor=south west,inner sep=0pt] at (0,0) { \includegraphics[height=50mm]{
		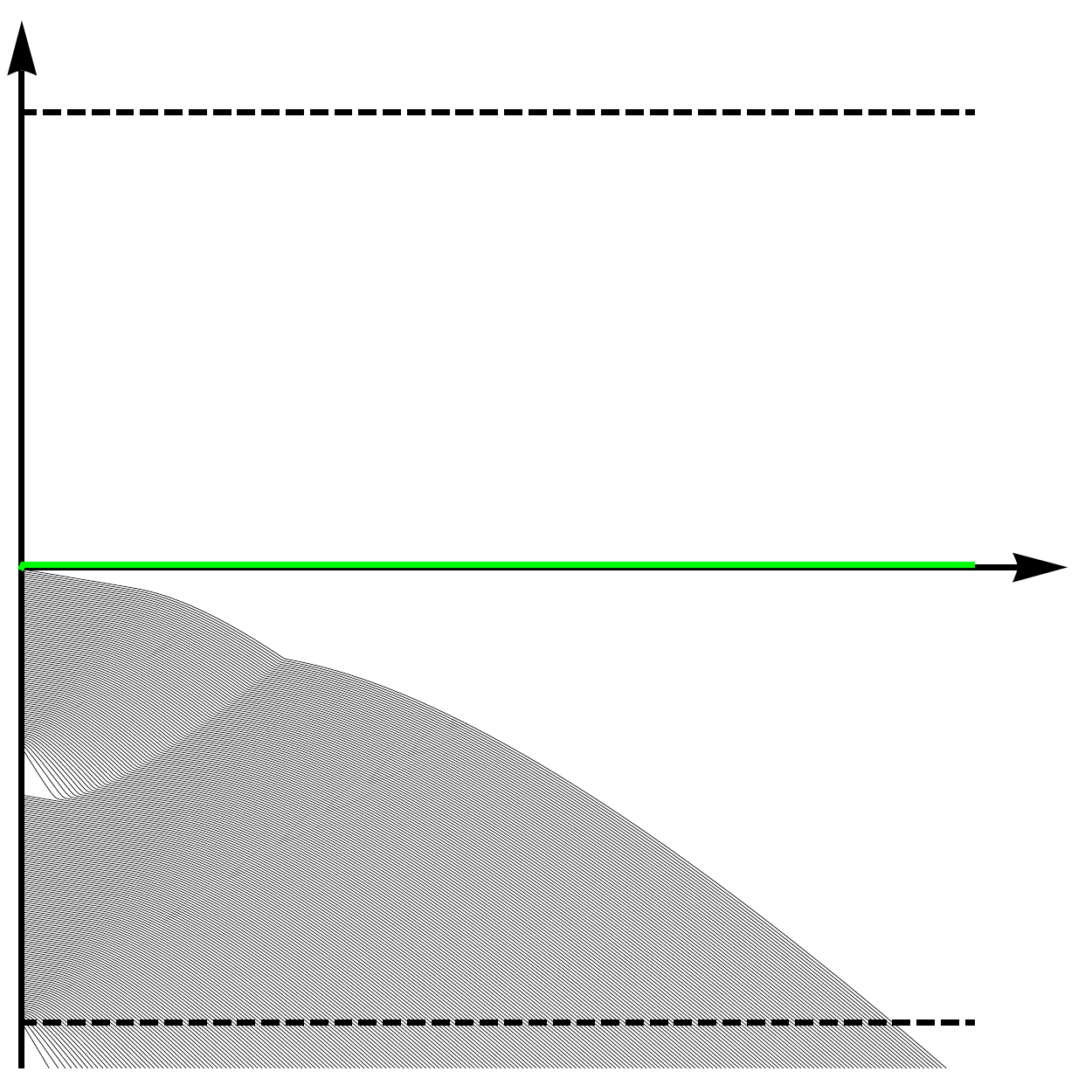}};
\node at (48,20) {\strut$t$};
\node[left] at (0,48) {$x$};
\end{tikzpicture}
\begin{tikzpicture}[x=1mm, y=1mm]
\node[anchor=south west,inner sep=0pt] at (0,0) { \includegraphics[height=50mm]{
		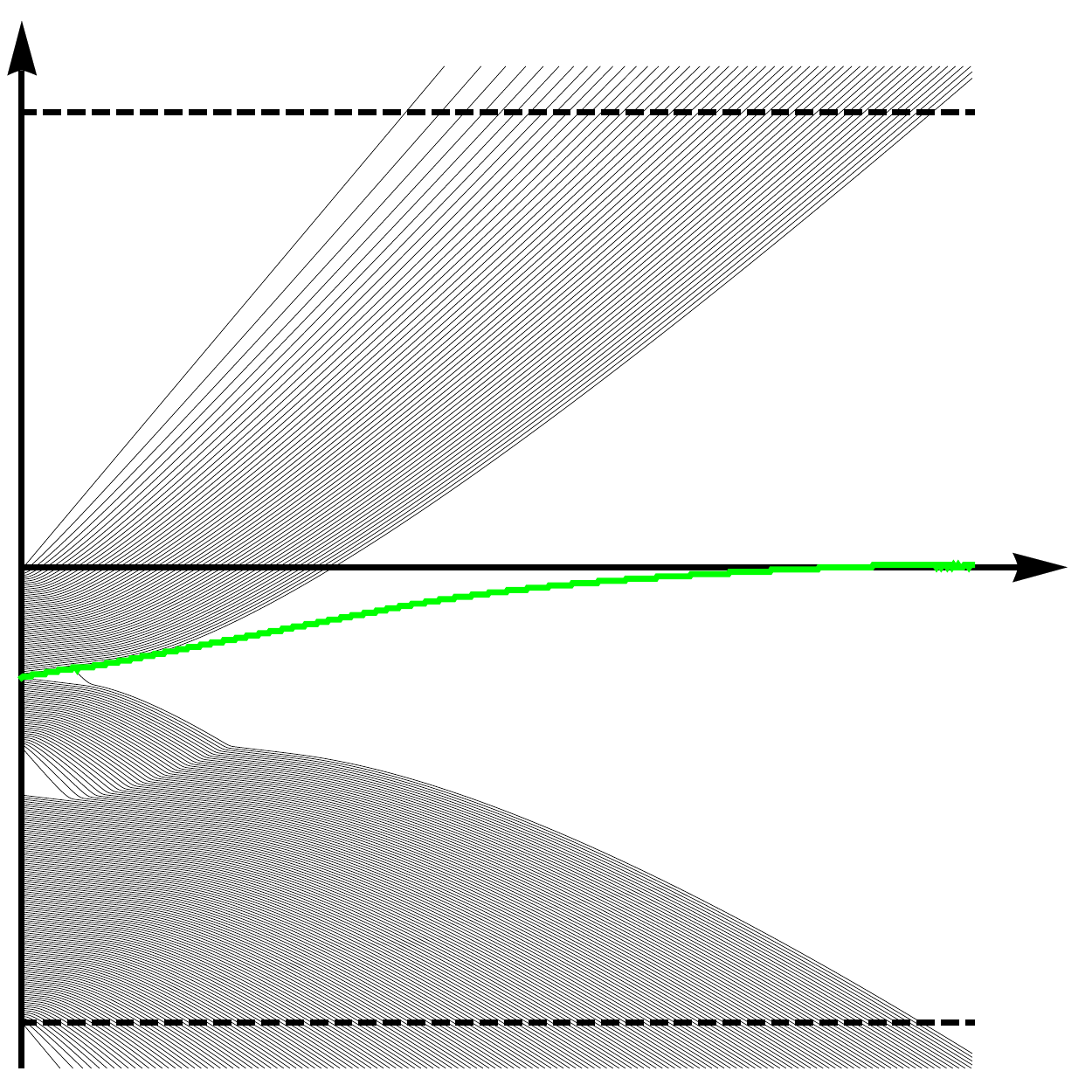}};
\node at (48,20) {\strut$t$};
\node[left] at (0,48) {$x$};
\end{tikzpicture}
\begin{tikzpicture}[x=1mm, y=1mm]
\node[anchor=south west,inner sep=0pt] at (0,0) { \includegraphics[height=50mm]{
		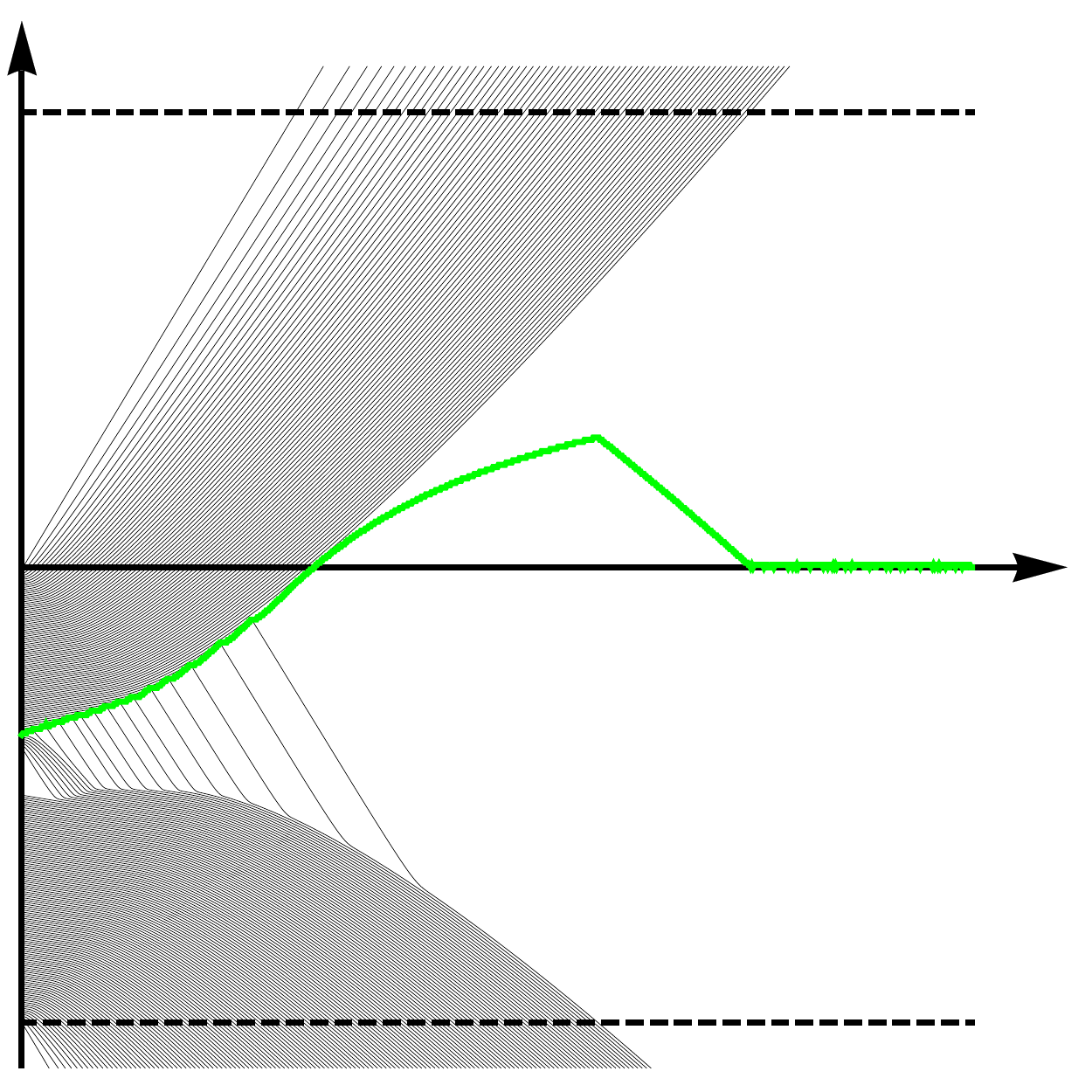}};
\node at (48,20) {\strut$t$};
\node[left] at (0,48) {$x$};
\end{tikzpicture}
\begin{tikzpicture}[x=1mm, y=1mm]
\node[anchor=south west,inner sep=0pt] at (0,0) { \includegraphics[height=50mm]{
		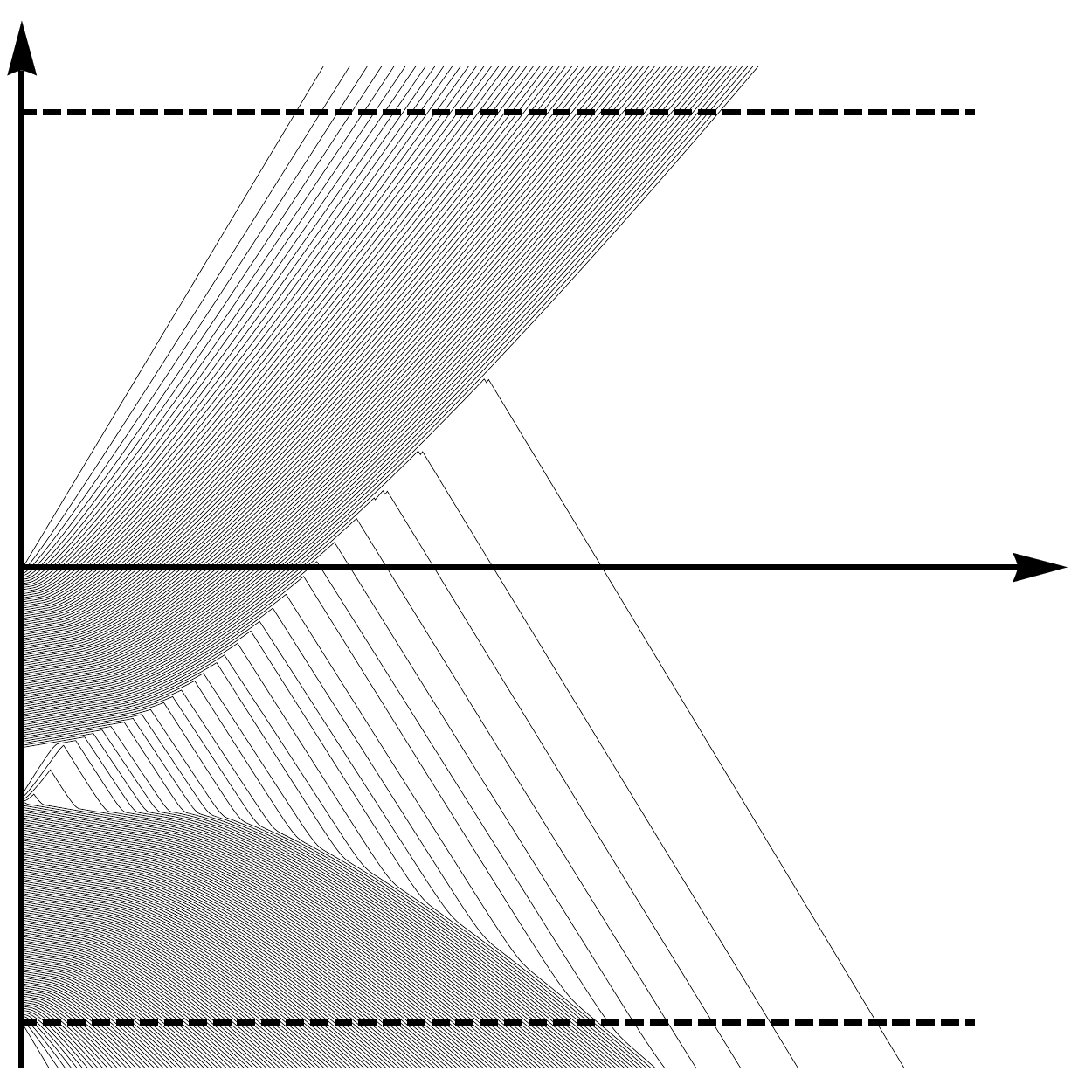}};
\node at (48,20) {\strut$t$};
\node[left] at (0,48) {$x$};
\end{tikzpicture}}
\\
\resizebox{\linewidth}{!}{
\includegraphics[height=50mm]{
	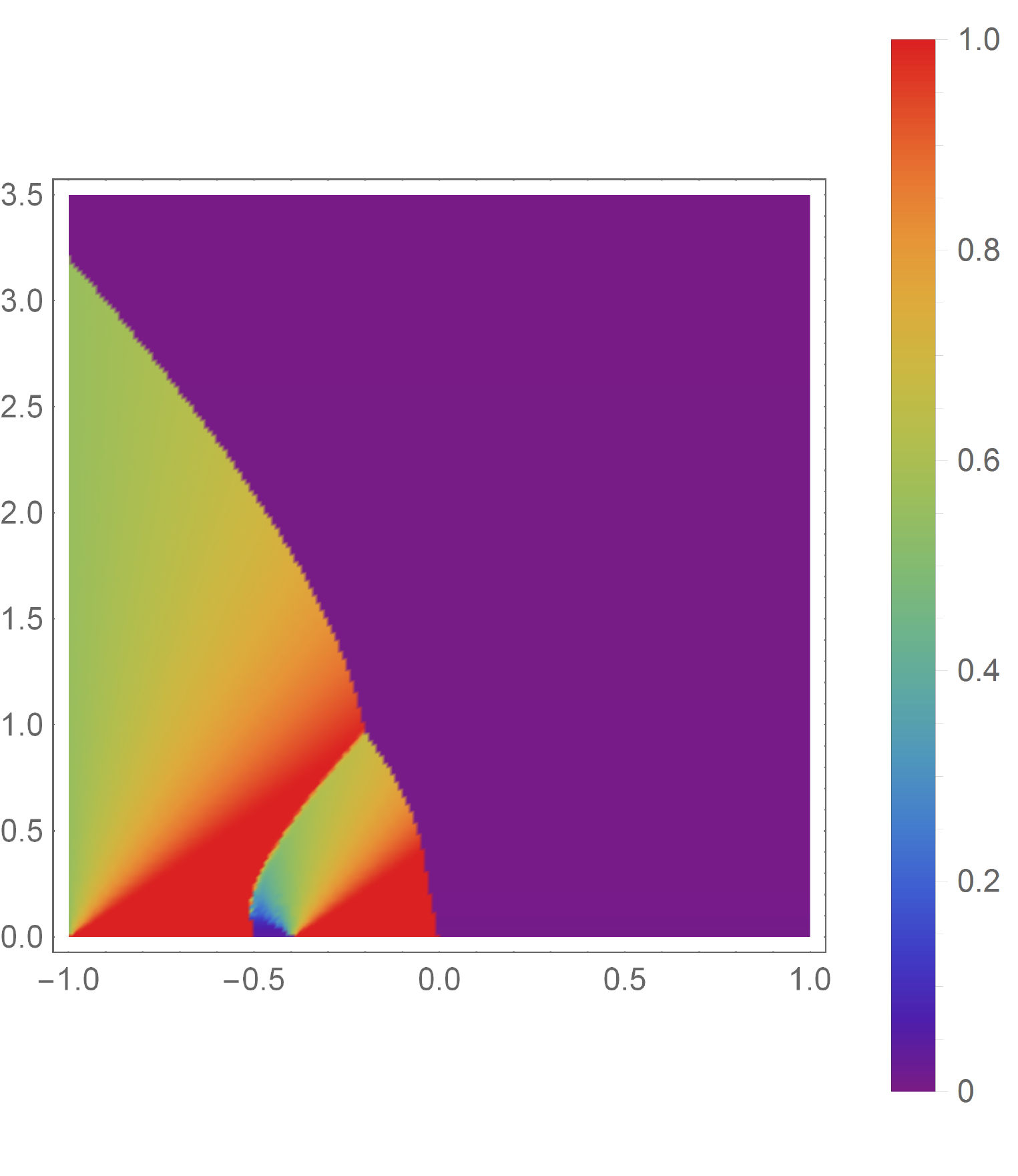}
\includegraphics[height=50mm]{
	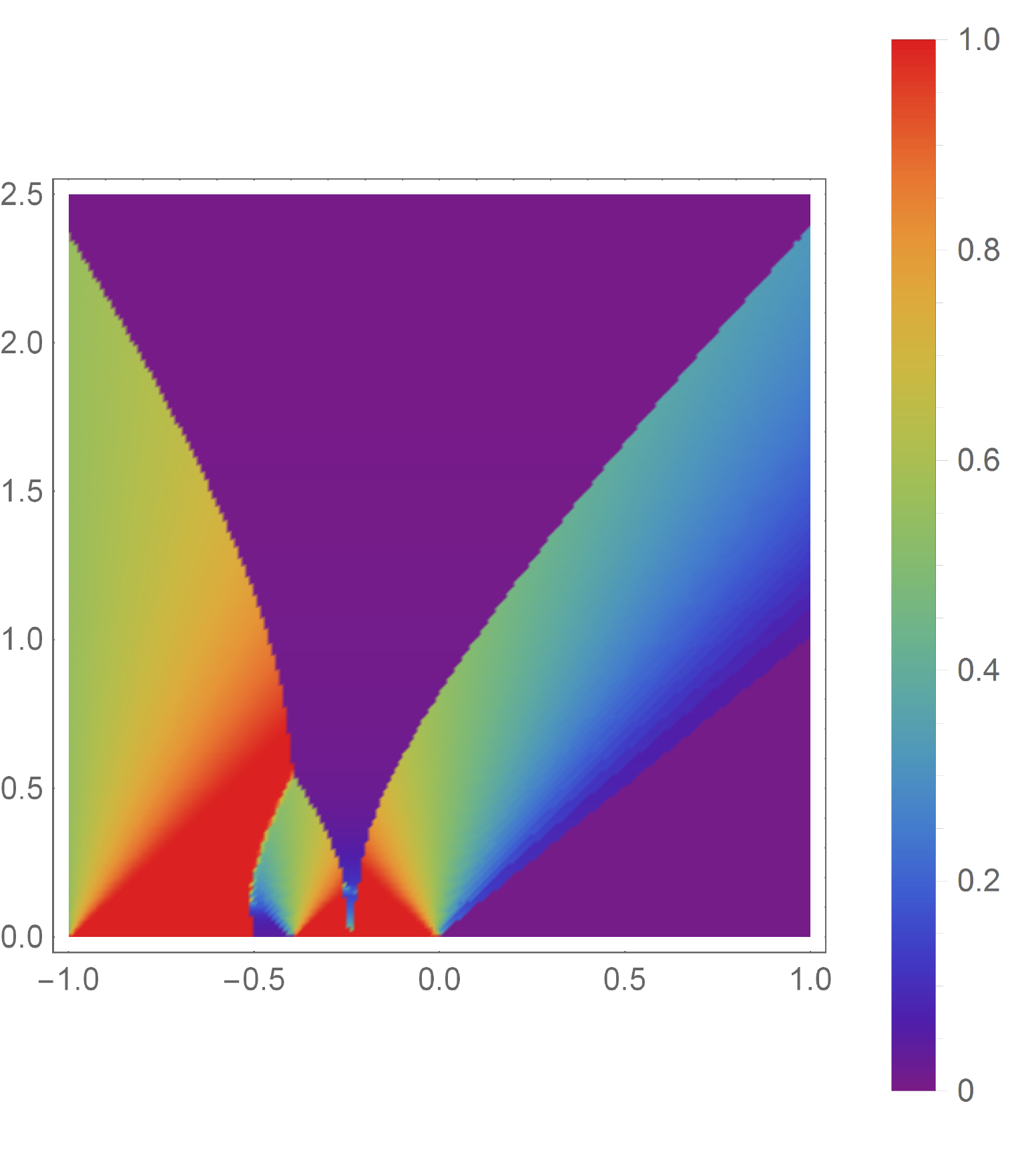}
\includegraphics[height=50mm]{
	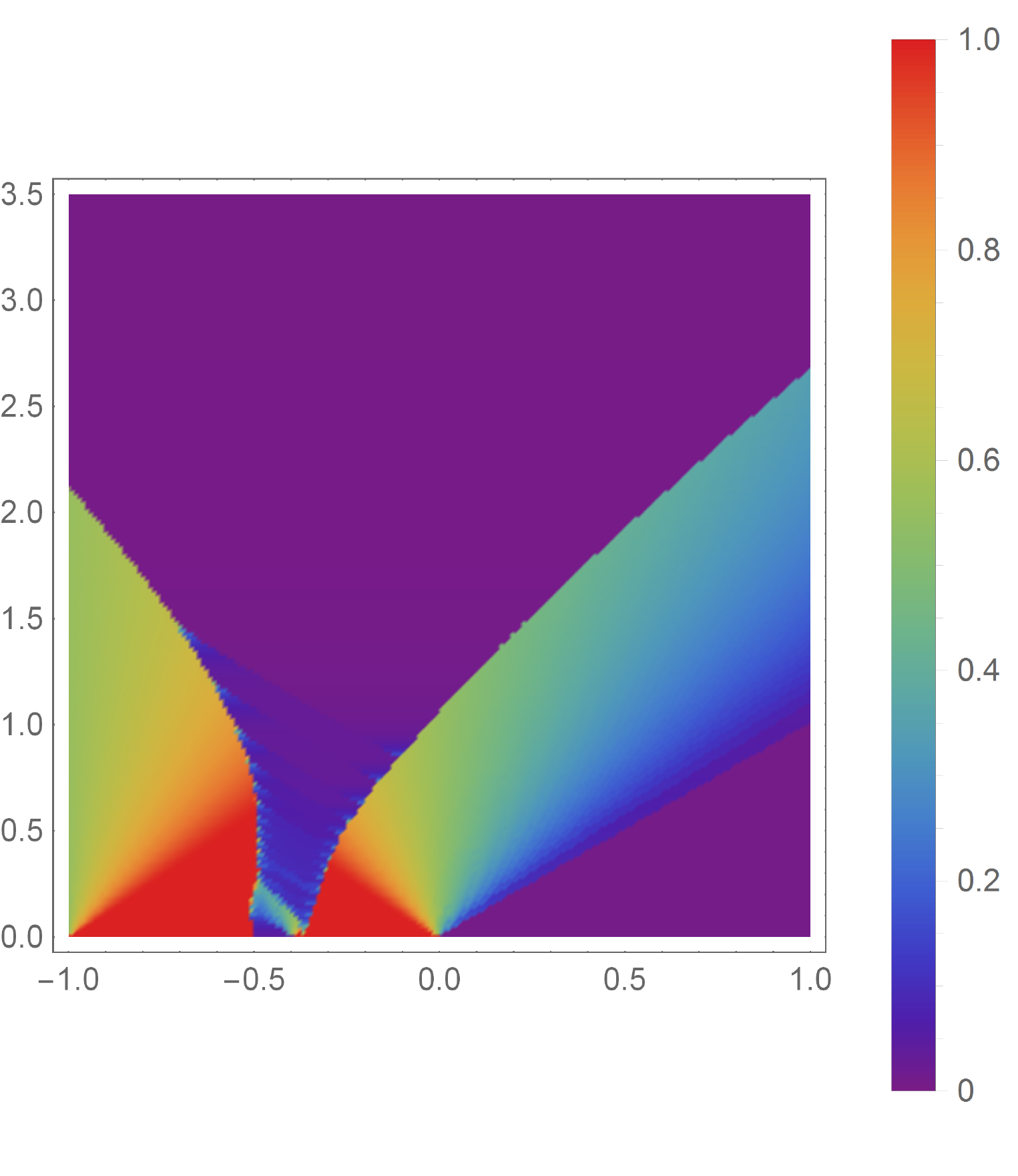}
\includegraphics[height=50mm]{
	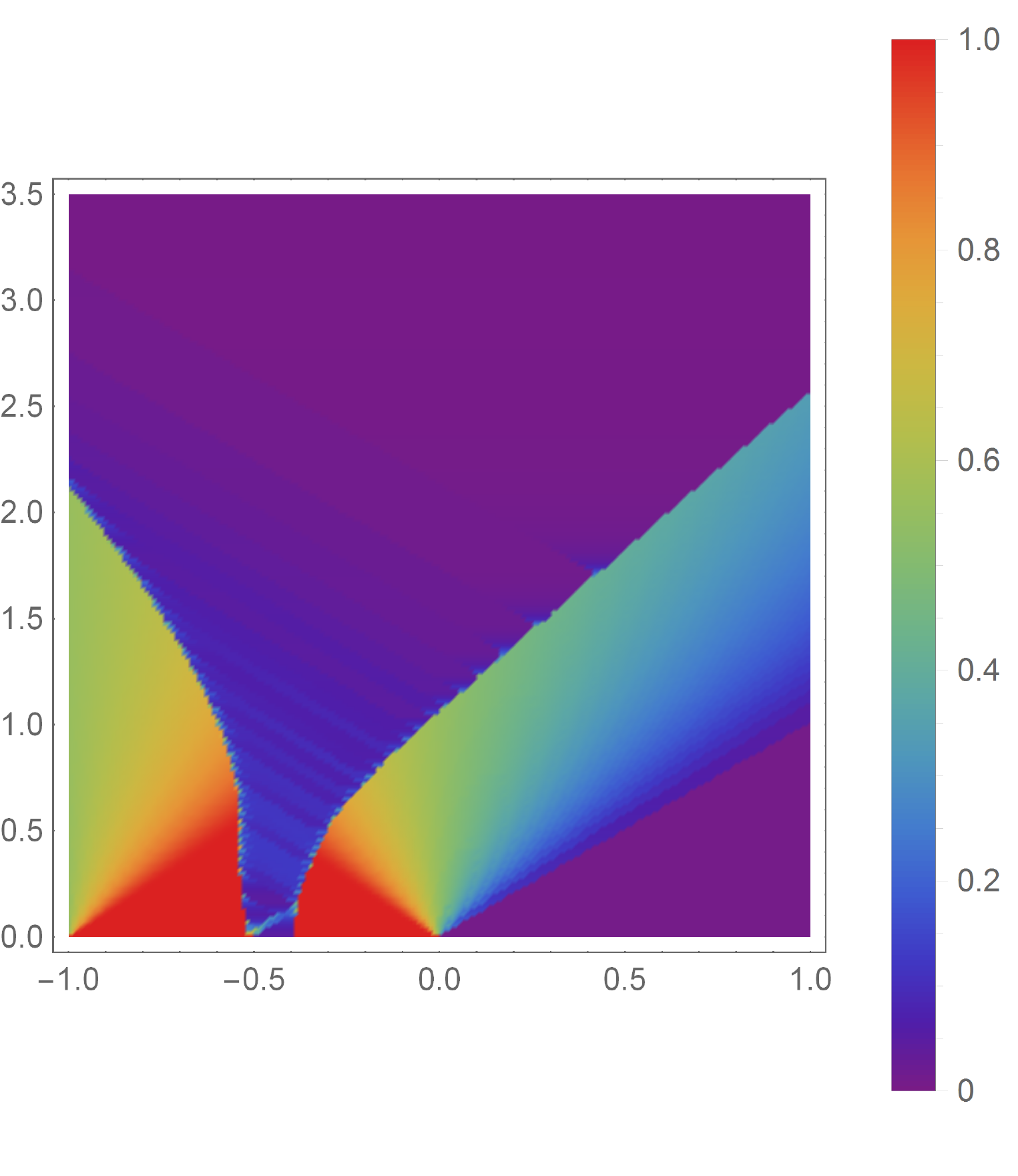}}
\caption{The particle paths and the turning curve, above, and the corresponding approximate macroscopic density, below, for $\alpha=0$, $\alpha=1.3$, $\alpha=5$ and $\alpha=18.9$, respectively.
The green lines are the turning curves.}
\label{f:cases}
\end{figure}
In order to avoid interpolation issues near the discontinuity points, we choose to display the continuous version $\xi^n$ of the turning curve rather than the exact turning curve $\zeta^n$ implicitly computed in the algorithm.
Here, $\xi^n$ is approximated by numerically solving the equation
\begin{align*}
&\frac{\xi^h}{\ell} + \alpha \Biggl( 
\begin{aligned}[t]
\# \{ i \in \llbracket0,n\rrbracket : -1 \leq \mathtt{x}_i^h < \xi^h\} - 1
&+ \frac{\xi^h-\mathtt{x}_i^h}{\mathtt{x}_{i+1}^h-\mathtt{x}_i^h}\Bigg\vert_{i=\# \{ i \in \llbracket0,n\rrbracket : \mathtt{x}_i^h < \xi^h\}}\\
&+ \frac{\mathtt{x}_{i+1}^h+1}{\mathtt{x}_{i+1}^h-\mathtt{x}_i^h}\Bigg\vert_{i=\# \{ i \in \llbracket0,n\rrbracket : \mathtt{x}_i^h < -1\}}
\Biggr)
\end{aligned}
\\={}&
\frac{\alpha}{2}
\Biggl( 
\begin{aligned}[t]
\# \{ i \in \llbracket0,n\rrbracket : -1 \leq \mathtt{x}_i^h < 1\} - 1
&+ \frac{1-\mathtt{x}_i^h}{\mathtt{x}_{i+1}^h-\mathtt{x}_i^h}\Bigg\vert_{i=\# \{ i \in \llbracket0,n\rrbracket : \mathtt{x}_i^h < 1\}}\\
&+ \frac{\mathtt{x}_{i+1}^h+1}{\mathtt{x}_{i+1}^h-\mathtt{x}_i^h}\Bigg\vert_{i=\# \{ i \in \llbracket0,n\rrbracket : \mathtt{x}_i^h < -1\}}
\Biggr).
\end{aligned}
\end{align*}
The above equation is a discrete approximation of the equation
\[\xi(t)+\alpha \int_{-1}^{\xi(t)} \rho(t,y) \, {\rm{d}} y = \frac{\alpha}{2} \int_{-1}^{1} \rho(t,y) \, {\rm{d}} y,\]
which is equivalent to \eqref{e:cost0}, but is numerically simpler because the unknown $\xi$ appears in just one integral.
The approximate macroscopic density is defined by
\[ \rho(t,x) := \sum_{i=0}^{n} \left[ \frac{\ell}{\mathtt{x}_{i+1}^h-\mathtt{x}_i^h} \, \mathbbm{1}_{[\mathtt{x}_i^h,\mathtt{x}_{i+1}^h)}(x) \right]_{h=\lfloor t/\Delta t\rfloor} .
\]
In \figurename~\ref{f:cases} we illustrate the evacuation behavior of the many-particle Hughes model by displaying the particle paths together with the turning curve, and the corresponding approximate macroscopic density for $\alpha \in\{0,1.3,5,18.9\}$.

\begin{acknowledgements}
The authors thank Marco Di Francesco for valuable discussions, for a review of our manuscript and particularly for the valuable comments and suggestions that have significantly improved it.
The second and the third authors are members of GNAMPA.
MDR acknowledges the support of the National Science Centre, Poland, Project \lq\lq Mathematics of multi-scale approaches in life and social sciences\rq\rq\ No.~2017/25/B/ST1/00051, by the INdAM-GNAMPA Project 2020 \lq\lq Dalla Buona Posizione alla Teoria dei Giochi nelle Leggi di Conservazione\rq\rq\ and by University of Ferrara, FIR Project 2019 \lq\lq Leggi di conservazione di tipo iperbolico: teoria ed applicazioni\rq\rq.  This paper has been supported by the RUDN University Strategic Academic Leadership Program.
\end{acknowledgements}

%
%

\bibliographystyle{spmpsci}      

\end{document}